\documentclass[11pt]{article}

\usepackage{amsfonts}
\usepackage{amssymb,amsmath,amsthm}
\usepackage{latexsym}
\usepackage{fullpage}
\usepackage{hyperref}
\usepackage{subcaption}

\usepackage{pgf,tikz,pgfplots}
\pgfplotsset{compat=1.15}
\usepackage{mathrsfs}
\usetikzlibrary{arrows}
\usepackage{cleveref}

\newtheorem{theorem}{Theorem}[section]

\newtheorem{lemma}[theorem]{Lemma}
\newtheorem{proposition}[theorem]{Proposition}
\newtheorem{cor}[theorem]{Corollary}
\newtheorem{defn}[theorem]{Definition}

\newtheorem{claim}[theorem]{Claim}

\theoremstyle{definition}
\newtheorem{remark}[theorem]{Remark}

\newcounter{tenumerate}

\def\P{\mathbb{P}}

\newcommand{\one}{\1}

\renewcommand{\epsilon}{\varepsilon}

\newcommand{\1}{\mathbf{1}}

\DeclareMathOperator{\var}{Var}

\newcommand{\E}{{\mathbb E}}
\newcommand{\remove}[1]{}

\renewcommand{\leq}{\leqslant}
\renewcommand{\geq}{\geqslant}
\newcommand{\diam}{\mathsf{diam}}

\def\XXint#1#2#3{{\setbox0=\hbox{$#1{#2#3}{\int}$}
		\vcenter{\hbox{$#2#3$}}\kern-.5\wd0}}

\newcommand\numberthis{\addtocounter{equation}{1}\tag{\theequation}}

\begin{document}
	
\title{Correlation length of the two-dimensional random field Ising model via greedy lattice animal}

\author{Jian Ding\thanks{Partially supported by NSF grant DMS-1757479 and DMS-1953848.}  \\ Peking University \and Mateo Wirth\footnotemark[1]  \\ University of Pennsylvania}

\maketitle

\begin{abstract}
For the two-dimensional random field Ising model where the random field is given by i.i.d.\ mean zero Gaussian variables with variance $\epsilon^2$, we study (one natural notion of) the correlation length,
which is the critical size of a box at which the influences of the random field and of the boundary condition on the spin magnetization are comparable.
We show that as $\epsilon \to 0$, at zero temperature the correlation length scales as
$e^{\Theta(\epsilon^{-4/3})}$ (and our upper bound applies for all positive temperatures).
\end{abstract}

\definecolor{qqqqcc}{rgb}{0,0,0.8}
\definecolor{rvwvcq}{rgb}{0.08235294117647059,0.396078431372549,0.7529411764705882}
\definecolor{wrwrwr}{rgb}{0.3803921568627451,0.3803921568627451,0.3803921568627451}
\definecolor{cqcqcq}{rgb}{0.7529411764705882,0.7529411764705882,0.7529411764705882}
\maketitle

\section{Introduction}
Let $\{h_v \,:\, v \in \mathbb{Z}^2\}$ be i.i.d.\ Gaussian random variables with mean zero and variance $1$.
For $N \geq 1$, let $\Lambda_N = \{v\in \mathbb Z^2: |v|_\infty \leq N\} \subset \mathbb Z^2$ be the box
of side length $2N$ centered at the origin $o$. For $u,v \in \mathbb Z^2$ with $|u-v| = 1$ (where $|\cdot|$
denotes the Euclidean norm), we say $u$ and $v$ are adjacent and write $u \sim v$. For $\epsilon \geq 0$, the
random field Ising model (RFIM) Hamiltonian $H^{\pm}$ on the configuration space $\{-1, 1\}^{\Lambda_N}$ with
plus (respectively, minus) boundary condition and external field $\{\epsilon h_v \,:\, v \in \Lambda_N\}$ is
defined to be
\begin{equation}\label{eq-def-hamiltonian}
H^{\pm}(\sigma,\Lambda_N, \epsilon h) =
	- \Big(
		\sum_{u \sim v, u,v\in \Lambda_N} \sigma_u \sigma_v \pm
		\sum_{u \sim v, u \in \Lambda_N, v \notin \Lambda_N} \sigma_u +
		\sum_{u \in \Lambda_N}\epsilon h_u \sigma_u
	\Big)\,,
\end{equation}
for $\sigma\in \{-1, 1\}^{\Lambda_N}$, where in the first sum each unordered edge appears once. 
For $\beta \geq 0$, let $\mu^\pm_{\beta, \Lambda_N, \epsilon h}$ be the Gibbs measure on $\{-1, 1\}^{\Lambda_N}$ at inverse-temperature $\beta$, defined as
\begin{equation}\label{eq-def-mu}
\mu^\pm_{\beta, \Lambda_N, \epsilon h}(\sigma) = \tfrac{1}{Z} e^{-\beta H^{\pm}(\sigma, \Lambda_N, \epsilon h)}\,,
\end{equation}
where $Z$ is the partition function so that $\mu^\pm_{\beta, \Lambda_N, \epsilon h}(\sigma)$ is a probability measure.
Note that $\mu^\pm_{\beta, \Lambda_N, \epsilon h}$ is a random measure which itself depends on $\{h_v\}$.
To clearly separate the two different sources of randomness, we will use $\P$ and $\E$ to refer to the
probability measure with respect to the external field $\{h_v\}$; we use $\mu^\pm_{\beta, \Lambda_N, \epsilon h}$ to denote the
 Ising measures and $\langle \cdot \rangle_{\mu^\pm_{\beta, \Lambda_N, \epsilon h}}$ to
denote the expectations with respect to the Ising measures.
For instance, $\langle \sigma^+_o\rangle_{\mu^+_{\beta, \Lambda_N, \epsilon h}}$ denotes the average value of the spin at the origin when
we sample $\sigma^+\in \{-1, 1\}^{\Lambda_N}$ according to $\mu^+_{\beta, \Lambda_N, \epsilon h}$.
We are interested in the following quantity which measures the influence of the boundary condition:
\begin{equation}\label{eq-def-magnetization}
 \mathtt m_{\beta, \Lambda_N, \epsilon} =
 	\frac{1}{2} \mathbb{E} \Big[
 		\langle \sigma^+_o\rangle_{\mu^+_{\beta, \Lambda_N, \epsilon h}} -
 		\langle \sigma^-_o\rangle_{\mu^-_{\beta, \Lambda_N, \epsilon h}}
 		\Big]\,.
\end{equation}
For $\mathtt m\in (0, 1)$, we consider the following notion of correlation length:
\begin{equation}\label{eq-def-correlation-length}
\psi(\beta, \mathtt m, \epsilon) = \min\{N:  \mathtt m_{\beta, \Lambda_N, \epsilon} \leq \mathtt m\}\,,
\end{equation}
which (for large $\beta$) amounts to the critical scale where the random field has a comparable influence as the boundary condition on
the spin at the origin. Here we use the convention that $\min \emptyset = \infty$.
\begin{theorem}\label{thm :: main result}
 For every $\mathtt m\in (0, 1)$, there exists $C = C(\mathtt m)>0$ such that $\psi(\beta, \mathtt m, \epsilon) \leq e^{C \epsilon^{-4/3}}$
 for all $\beta \geq 0$ (including $\beta = \infty$), and that
 $\psi(\infty, \mathtt m, \epsilon) \geq e^{C^{-1}\epsilon^{-4/3}}$ for $\beta = \infty$.
\end{theorem}
\begin{remark}
The emergence of the $4/3$ exponent is somewhat unexpected, and it is reminiscent of the $4/3$-exponent in
upper bounds on distances for Liouville quantum gravity at high temperatures \cite{DG19}: the $4/3$-exponent arises from ``back-of-the-envelope'' computations that are similar in spirit for both scenarios (an interested reader may compare \cite[Section 2]{DG19} with Subsection~\ref{sec :: emergence}).
However, the random field Ising model and Liouville quantum gravity are two drastically different models, and as a result their mathematical treatments are different except that they both employ a framework of multi-scale analysis.
\end{remark}
\begin{remark}
During the submission of this paper, the lower bound here was extended to low temperatures (i.e., to large finite $\beta$) by \cite{DZ22} which takes the result at $\beta = \infty$ as an input. The key idea in \cite{DZ22} is to extend the Peierls argument in the construction of the Peierls mapping, where the additional novelty is to  also flip the signs of the disorder when flipping the signs of spins on a simply connected component. In addition, by \cite[Corollary 1.6]{DSS22} there is an exponential decay for $\beta < \beta_c$ and moreover the decaying rate is upper-bounded by that for $\epsilon = 0$. Furthermore, the behavior for moderate $\beta > \beta_c$ seems rather challenging and currently we have a weak belief that the $e^{\epsilon^{-4/3}}$-scaling for the correlation length holds for all $\beta > \beta_c$. Ultimately, it would be very interesting to completely understand the phase diagram of the mapping from $(\beta, \epsilon)$ to the rate of exponential decay (as proved in \cite{DingXia19, AHP20}), but this seems out of reach for now and we do not have any intuition beyond what has been discussed.
\end{remark}
\begin{remark}
More than one year after the arXiv post of this paper and more than half a year after the arXiv post of \cite{DZ22}, a paper \cite{BN22} was posted which proved an upper bound of  $\exp(e^{O(\epsilon^{-2})})$ and a lower bound of (the type of) $e^{\epsilon^{-2/3}}$ for the correlation length. In addition, we note that the notion for the correlation length in the upper bound of \cite{BN22} governs the rate of exponential decay and thus in terms of upper bound it is a stronger notion than the one used in this paper.
\end{remark} 
\begin{remark}\label{rem-exponential-decay-rate}
A very natural question is whether one can prove the scaling of $e^{\epsilon^{-4/3}}$ for the correlation length that governs the rate of exponential decay. As far as we can tell, to this end one needs to combine the techniques from \cite{DingXia19, AHP20} (see also\cite{BN22}) with methods in this paper. This does not seem to be trivial since the key point of \cite{DingXia19, AHP20} is to prove that the boundary influence has a polynomial decay with a large power, while in this paper in order to derive a contradiction currently it seems inevitable to assume in the contradiction hypothesis that the boundary influence is lower-bounded by a constant. Maybe, a vaguely plausible approach is to show that once the side length exceeds $e^{\epsilon^{-4/3}}$ the boundary influence will start seeing a decay and that also the tortuosity assumption employed in \cite{DingXia19, AHP20} for disagreement percolation would hold. But by all means this is a highly-nontrivial task and we feel better to leave it for future study and advise an interested researcher to keep their mind open.
\end{remark}

This result lies under the umbrella of the general Imry–Ma \cite{ImryMa75} phenomenon on the effect of disorder on phase transitions in two-dimensional physical systems.
We next give a brief review of the development in the particular case of RFIM.
In the limit of $N\to \infty$ with small fixed $\epsilon>0$, it was shown in \cite{AizenmanWehr89, AizenmanWehr90} that $\mathtt m_{\beta, \Lambda_N, \epsilon}$ decays to 0 for all $\beta \geq 0$, which also implies the uniqueness of the Gibbs state.
The decay rate was then improved to $1/\sqrt{\log \log N}$ in \cite{Chatterjee18}, to $N^{-c}$ (for some small $c > 0$) in \cite{AizenmanPeled19}, and finally to $e^{-cN}$ in \cite{DingXia19, AHP20}
(previously, exponential decay was shown in \cite{Berretti85, FI84, vDHKP95, CJN18} for large $\epsilon$).

In three dimensions and above, however, the behavior is drastically different from that in two dimensions: it was shown in \cite{Imbrie85} that long range order exists at zero temperature with weak disorder, i.e., $\mathtt m_{\infty, \Lambda_N, \epsilon}$ does not vanish as $N$ grows; later an analogous result was proved in \cite{BK88} (see also \cite[Chapter 7]{bovier2006}) at low temperatures.
A heuristic explanation for the different behaviors is as follows: in two dimensions the fluctuation of (the sum of) the random field in a box is of the same order as the size of the boundary,
while in three dimensions and above the fluctuation of the random field is substantially smaller than the size of the boundary.

In the limit as $\epsilon \to 0$, the scaling of the correlation length in both two dimensions and three
dimensions (at some ``critical'' temperature) has remained largely elusive even from the point of view of physics predictions despite extensive studies.
Previous works include (a partial list of) numeric studies
\cite{YN85, SPA98, FV99, Rieger_1993, Rigger95, SA01, PS02,SKBP14} and non-rigorous derivations
\cite{PIM81,GristenMa82, Binder83,GM83, DS84,Bray_1985, BricmontKupiainen88}. It is worth noting that
most of the studies in two dimensions were at zero temperature, but even in this case there was no consensus
on the scaling of the correlation length: while a common belief seemed to be that it scales like $e^{\epsilon^{-2}}$ (or
upper bounded by $e^{O(\epsilon^{-2})}$) as argued in \cite{GristenMa82, Binder83, BricmontKupiainen88, SPA98, SA01}, there were also other predictions including a scaling of $e^{\epsilon^{-1}}$ in a more recent work \cite{SKBP14}.
(We note that some of these papers studied our notion of correlation length, and some studied the notion which is the inverse of the rate of exponential decay, and some were not very careful in distinguishing these two notions.)
Prior to our work, the only mathematical result on the correlation length was (as far as we know) an upper bound of $e^{e^{O(\epsilon^{-2})}}$ from \cite{Chatterjee18, AizenmanPeled19}.

Our proof method for the upper bound on the correlation length shares the underlying philosophy of ``using the fluctuation of the random field to fight against the influence from the boundary'' with previous works \cite{AizenmanWehr90, Chatterjee18, AizenmanPeled19, DingXia19, AHP20}, and in particular in the sense that the proof strategy shares some similarity with
\cite{AizenmanWehr90} for deriving a contradiction for lower and upper bounds on difference of free energies. However, our strategy of deriving the lower bound on the difference of free energies (which is the key point for both \cite{AizenmanWehr90} and our proof for upper bound on the correlation length) is very different from that in \cite{AizenmanWehr90}. The proof of the lower bound of the correlation length is completely different from \cite{AizenmanWehr90, Chatterjee18, AizenmanPeled19, DingXia19, AHP20} since this is a bound in a different direction from these works. In fact, it shares some similarity with \cite{Chalker83, FFS84} in terms of a connection to greedy lattice animals, as we elaborate in what follows. 
Let $\mathcal A_N$ be the collection of all connected subsets of $\Lambda_N$ (i.e., lattice animals) that contain the origin and let $\mathfrak A_N \subset \mathcal A_N$ be the collection of all simply connected subsets in $\mathcal A_N$.
We define (the value of) the greedy lattice animal normalized by its boundary size as
\begin{equation}\label{eq-greedy-lattice-animal}
\mathcal S_N = \max_{A\in \mathcal A_N} \frac{\sum_{v\in A} h_v}{|\partial A|} \mbox{ and } \mathfrak S_N = \max_{A\in \mathfrak A_N} \frac{\sum_{v\in A} h_v}{|\partial A|}\,,
\end{equation}
where $|\partial A|$ is the number of edges with exactly one endpoint in $A$. Theorem~\ref{thm :: main result}
is deeply connected to the following result (see Section~\ref{sec :: Overview} for an extensive discussion).
\begin{theorem}\label{thm-greedy-lattice-animal}
There exists a constant $C>0$ such that for all $N \geq 3$ we have
$$C^{-1} (\log N)^{3/4} \leq \mathbb E[\mathfrak S_N] \leq \mathbb E [\mathcal S_N] \leq C(\log N )^{3/4}\,.$$
\end{theorem}
\begin{remark}
In Theorem~\ref{thm-greedy-lattice-animal} we described the maxima over both connected subsets and simply connected subsets for the following reasons: (1) Both upper and lower bounds can be obtained for simply connected subsets first and then it is relatively easy to translate the bound to connected subsets; (2) while it is easier to prove the lower bound on the correlation length using the upper bound for the maximum over connected subsets, fundamentally what governs the behavior seems to be the maximum over simply connected subsets as we see in three dimensions (see also the proof in \cite{DZ22} where the maximum over simply connected subsets plays a fundamental role). 
\end{remark}
There is an interesting {\bf historical development} on Theorem~\ref{thm-greedy-lattice-animal}. The formulation of the statement immediately reminded the authors of the greedy lattice animal normalized by its volume (either normalized by the volume of the animal or by the volume of the box which contains the animal); this has been extensively studied for general disorder distributions (see \cite{Lee93, CGGK93, GK94, Lee97, Lee97b, DGK01, Martin02, Hammond06}). In particular, a rather precise description was obtained for the greedy lattice animal in \cite{Hammond06}, including that for rather general distributions (including the Gaussian distribution) the greedy lattice animal in a $d$-dimensional box of side length $N$ normalized by $N^{d}$ converges to a fixed constant (where the limiting constant depends on the distribution and the dimension).
Despite a high degree of similarity in the definitions between the greedy lattice animal normalized by its boundary size and the version normalized by its volume, their behaviors seem to be quite different and the mathematical proofs in these two scenarios are largely different too: in some sense such difference is suggested in the $(\log N)^{3/4}$ growth of $\mathcal S_N$ whereas in the version normalized by its volume this was known to converge to a constant. 

In three dimensions and higher, it was shown in \cite{Chalker83, FFS84} that the simply connected greedy lattice animal normalized by its boundary size (i.e., the analogue of $\mathfrak S_N$ in higher dimensions) is $O(1)$, which played a useful role in the proof for the existence of long range order at zero temperature in \cite{Imbrie85, BK88}.
The $O(1)$ bound in three dimensions and higher and the $(\log N)^{3/4}$ growth in two dimensions for $\mathfrak S_N$ can be seen as a stronger version of the intuition underlying the Imry-Ma argument for the transition in dimension for statistical physics models with random field.
Finally, we remark that in retrospect the proof in \cite{Chalker83, FFS84} amounts to a non-trivial application of Dudley's integral bound \cite{Dudley67} (note that the actual proof was implemented in a self-contained manner).  

Initially, the authors thought that Theorem~\ref{thm-greedy-lattice-animal} was new and as a result provided a self-contained proof (for a slightly weaker version of Theorem~\ref{thm-greedy-lattice-animal}) in the first version of this paper. During the submission, we discovered \emph{in the literature} a non-obvious but deep connection between the greedy lattice animal normalized by its boundary size and the matching problem in Euclidean spaces. A fundamental problem is to match i.i.d.\ uniform points $X_1, \ldots, X_{N^d}$ in a $d$-dimensional box containing $N^d$ lattice points $y_1, \ldots, y_{N^d}$ (that is, to find a bijection $\pi$ between these two set of points) in a certain optimal way. A classic result of \cite{AKT84} proved that $\E [\min_\pi \frac{1}{N^d}\sum_{i=1}^{N^d} |X_{\pi(i)} - y_{i}|] = \Theta(\sqrt{\log N})$ for $d=2$. Since \cite{AKT84} there has been extensive work on matching problems, and one is encouraged to see \cite{Talagrand14} for an excellent account on the topic, which presents a unified proof via the majorizing measure theory. Of particular relevance to Theorem~\ref{thm-greedy-lattice-animal} is the celebrated work of \cite{LS89} which showed that $\E [\min_\pi\max_{1\leq i\leq N^d} |X_{\pi(i)} - y_{i}|] = O((\log N)^{3/4})$ for $d=2$. The power of $3/4$ is deeply connected to the power in Theorem~\ref{thm-greedy-lattice-animal} via Hall's marriage lemma as we next explain.

Putting Halls's marriage lemma into the context of the matching problem, it states that if for each lattice point $y_i$ there exists a collection of random points $A_i$ such that 
\begin{equation}\label{eq-Hall-condition}
|\cup_{i\in I}A_i| \geq |I| \mbox{ for all } I\subset \{1, \ldots, N^d\},
\end{equation}
 then there exists a bijection $\pi$ such that $X_{\pi(i)}\in A_i$. In light of this, a natural choice of $A_i$ is the collection of all random points in a ball of radius $r$ centered at $y_i$. As such, the result of \cite{LS89} essentially reduces to showing that \eqref{eq-Hall-condition} holds for $r = O((\log N)^{3/4})$. It is plausible that in order to verify \eqref{eq-Hall-condition} one \emph{essentially} only needs to consider $I$ when $I$ is the set of lattice points in a simply connected subset $\mathsf I\subset \mathbb R^d$. Since the union of the balls centered at $I$ is an expansion of $\mathsf I$, that is, the union of $\mathsf I$ and all points with distance at most $r$ from $\mathsf I$, a moment of thinking should lead to that with high probability for typical $I$ (which turns out to be the ones we care most)
 $$\lambda(\mathsf I) + c r \mathrm{length}(\partial \mathsf I) \leq  |\cup_{i\in I}A_i| \leq \lambda(\mathsf I) + C r \mathrm{length}(\partial \mathsf I) $$
 where $c, C>0$ are constants, $\lambda(\mathsf I)$ is the number of random points in $\mathsf I$, and  $\mathrm{length}(\partial \mathsf I)$ is the length of the boundary curve for $\mathsf I$. Since $\lambda(\mathsf I) - |I|$ is a mean-zero random variable, which can be roughly regarded as a Gaussian variable, and thus in spirit $\{\lambda(\mathsf I) - |I|\}_I$ resembles the lattice animal process. In light of this discussion, heuristically the result of \cite{LS89} reduces to $\max_I \frac{\lambda(\mathsf I) - |I|}{\mathrm{length}(\partial \mathsf I)} = O((\log N)^{3/4})$ which resembles the upper bound in Theorem~\ref{thm-greedy-lattice-animal}. Indeed, this connection was nicely explained in \cite{Talagrand14}, which also nicely explains the conceptual difference for the behavior between $d = 2$ and $d\geq 3$.
 
Having explained the connection to the matching problem, we come back to what is most relevant to us, i.e., the proof of Theorem~\ref{thm-greedy-lattice-animal}. It turns out that a proof of Theorem~\ref{thm-greedy-lattice-animal} was essentially contained in \cite{Talagrand14}, and in Section~\ref{sec :: Lower bound proof} we present this in a more explicit manner without claiming any credit. In addition to that, in Section~\ref{sec :: Lower bound proof original} we still keep our ``original'' proof since we feel that our proof seems to explain some of the geometric intuition in an arguably more intuitive way and thus we feel that this framework of multi-scale analysis may turn out to be useful in some related problems (e.g., random metric of Liouille quantum gravity).

\medskip

We conclude the introduction by some discussions on future research. As a natural question, one may ask what is the correlation length for the random field Potts model. We expect that the
same scaling of $e^{\epsilon^{-4/3}}$ should occur. The nontrivial part is the upper bound, for which our
proof uses monotonicity properties of the Ising model in a substantial manner.

\section{Overview of the proof}\label{sec :: Overview}
\label{Prelims}
In this section we introduce the main idea behind the proof of Theorem~\ref{thm :: main result},
and in particular we give some intuition for the exponent $4/3$. We will then discuss the obstacles
that arise in making this proof sketch rigorous.

\subsection{Notation}
For a real (or integer-valued) vector $\mathbf x$ (in any dimension), we denote
its Euclidean norm by $|\mathbf x|$. For a finite set $A$, we denote its cardinality
by $|A|$. For $A \subset \mathbb R^2$ we denote the Lebesgue measure of $A$ by
$\lambda(A)$. For a curve $\eta$, we denote its length by $l(\eta)$. We use $A^c$
to denote the complement of the set (or event) $A$. If $A$ is an event, we denote
its indicator by $\one_A$.

In what follows, we let $c, c', c'',C,C',C'' > 0$ be arbitrary constants whose
values may change each time they appear, and may depend on $\mathtt m$ but not on
$\epsilon$ or $N$. Numbered constants $c_1, c_2,\ldots$ may still depend on $\mathtt m$ but
their values will be fixed throughout the paper.

We say two points $u,v \in \mathbb Z^2$ are adjacent to each other if $|u -v| = 1$,
in which case we write $u \sim v$. When convenient, we will think of
$\mathbb Z^2$ as being embedded in $\mathbb R^2$ in the obvious way.
For any set $A\subset \mathbb Z^2$, we let
$\partial A = \{(u,v): u\sim v,\, u\in A, v \in \mathbb Z^2 \setminus A\}$ denote the
edge boundary of $A$ in the nearest neighbor graph on $\mathbb Z^2$.

\subsection{Emergence of the $3/4$ exponent} \label{sec :: emergence}

Let $\sigma^{\pm}(\Lambda_N, \epsilon h)$ be the ground states with respect to the plus and minus boundary conditions, i.e., they are minimizers of the Hamiltonians $H^{\pm}(\Lambda_N, \epsilon h)$ respectively.
(Since our field $h$ has a continuous distribution, the ground state with respect to each boundary condition is unique with probability 1).
Suppose $\sigma^-_o(\Lambda_N, \epsilon h) = 1$ and $S$ is the connected component of $\{ v\in \Lambda_N : \sigma^-_v(\Lambda_N,\epsilon h) = 1\}$ that contains $o$.
Then necessarily we have $\sum_{v\in S} \epsilon h_v \geq |\partial S|$, because otherwise flipping spins on $S$ would decrease the Hamiltonian and contradict the definition of the ground state.
In other words,
\begin{equation}\label{eq :: ground state flip condition}
\sigma^-_o(\Lambda_N, \epsilon h) = 1 \mbox{ implies that }
\max_{A\in \mathcal A_N} \frac{\sum_{v\in A} \epsilon h_v}{|\partial A|} \geq 1\,.
\end{equation}
This explains why the greedy lattice animal normalized by its boundary size is connected to the random field Ising model.
From the discussion above, an upper bound on the greedy lattice animal directly gives a lower bound on the correlation length for $\beta = \infty$.
In what follows, we will sketch an argument leading to the emergence of $3/4$-exponent in the lower bound of Theorem~\ref{thm-greedy-lattice-animal}.

For convenience of exposition, we will pass to the continuum.
To each vertex $v \in \mathbb Z^2$ we can associate the axis-aligned unit square $R_v$ centered at $v$, and to each subset $A \subset \mathbb{Z}^2$ the set $\mathtt A = \cup_{v \in A} R_v$.
Notice that the perimeter of $\mathtt A$ (which we denote by $l(\partial \mathtt A))$ is equal to the boundary size $|\partial A|$.
Next, we let $W$ be a standard white noise on $\mathbb R^2$ such that $W(R_v) = h_v$ for each $v \in \mathbb Z^2$.
In particular, for any $A \subset \mathbb Z^2$ we have $\sum_{v \in A} h_v = W(\mathtt A)$.
We will sketch a procedure to construct a polygon $P \subset \mathbb [-N,N]^2$ (for $N \geq  e^{C \epsilon^{-4/3}})$ such that each side of $P$ has length at least 1 (we will refer to this as a polygon animal in what follows) and $\epsilon W(P) > l(\partial P)$.
The idea is to recursively expand $P$ by possibly joining to it a triangle $T$ such that the standard deviation of $\epsilon W(T)$ is of the same order as $l(\partial(P \cup T)) - l(\partial P)$.
We remark that we choose to add triangles instead of rectangles for the reason that
adding a triangle with the same area results in a substantially smaller increase in the perimeter.

We begin with the polygon $P_1 = [-N/2,N/2]^2$.
Having constructed $P_k$, we construct $P_{k+1}$ as follows.
For each side $s$ of $P_k$, we consider the isosceles triangle $T_s$ with base given by the ``middle'' segment of $s$ of length $l(s)/2$ and of height $\epsilon^{2/3} l(s)/8$ that points out of $P_k$.
We add $T_s$ to the polygon if $W(T_s) > 0$ (which occurs with probability $1/2$).
If we do not add $T_s$, we split $s$ into four sides of equal length.
We let $P_{k+1}$ be the polygon obtained by applying this procedure to each side of $P_k$.
See Figure~\ref{fig :: random set illustration} for an illustration of the process.

\begin{figure}
	\centering
	\begin{subfigure}{.33\textwidth}
		\centering
		\includegraphics[width=.6\linewidth]{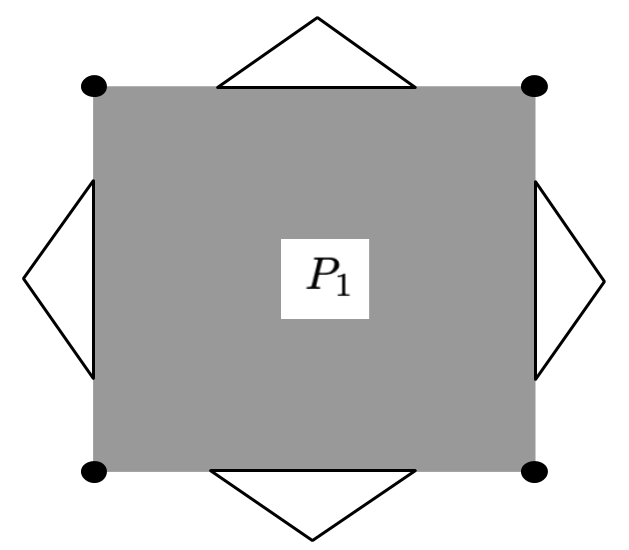}
	\end{subfigure}%
	\begin{subfigure}{.33\textwidth}
		\centering
		\includegraphics[width=.6\linewidth]{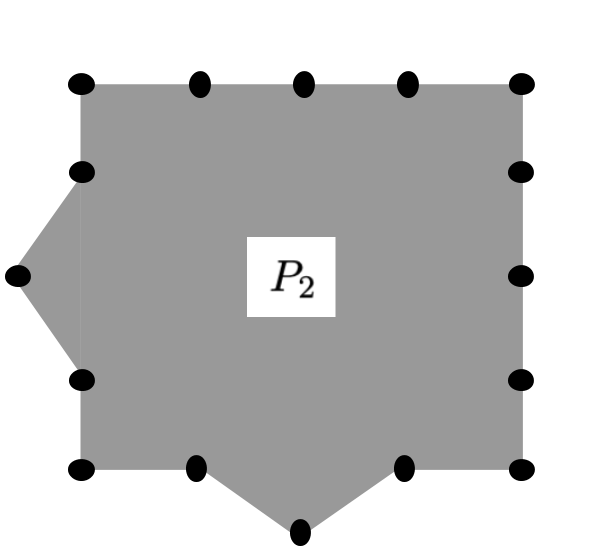}
	\end{subfigure}
	\begin{subfigure}{.33\textwidth}
		\centering
		\includegraphics[width=.6\linewidth]{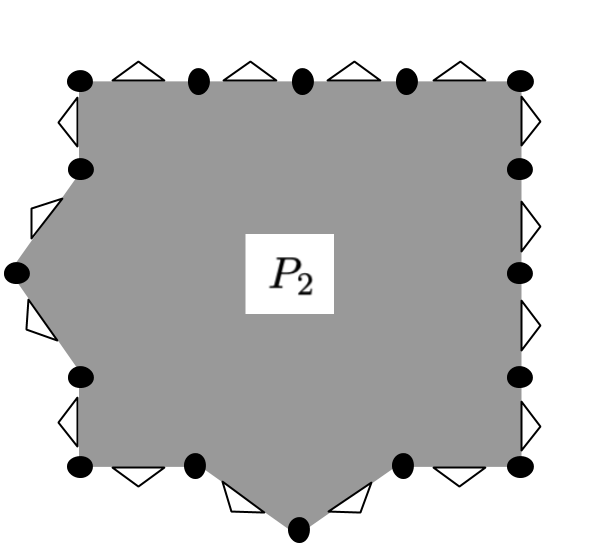}
	\end{subfigure}
	\caption{From left to right: $P_1$ with potential triangles to be added; $P_1$ with triangles added (i.e. $P_2$); $P_2$ with potential triangles to be added.}
	\label{fig :: random set illustration}
\end{figure}

Next, we let $a_k = \mathbb E[\epsilon W(P_k) - l(\partial P_k)]$.
Our goal is to lower bound $a_{k+1} - a_k$.
For each side $s$ of $P_k$, we have $\lambda(T_s) = \epsilon^{2/3} l(s)^2/32$ (recall that $\lambda$ denotes the Lebesgue measure on $\mathbb R^2$),
and an elementary calculation shows that adding $T_s$ to $P_k$ increases the perimeter by $\Delta_s < \epsilon^{4/3}l(s)/16$.
If we ignore the potential overlap between the triangles corresponding to different iterations of the scheme we would have $\mathbb{E}[\epsilon W(T_s) \mid W(T_s) > 0] > 2\Delta_s$.
Summing over all sides of $P_k$, we get that
\[
a_{k+1} - a_k \geq \tfrac{1}{16}\epsilon^{4/3} \mathbb E[l(\partial P_k)]
	\geq \tfrac{1}{16} \epsilon^{4/3} l(\partial P_1) \,.
\]
Further, since at each step each side $s$ is split into four sides of length at least $l(s)/4$ we see that for $k^* = \lfloor \log_{16} N \rfloor$ each side of $P_{k^*}$
 has length at least 1 deterministically.
This implies that for $N \geq 10^5\exp(10^5\epsilon^{-4/3})$, we have (noting that $a_1 = -l(\partial P_1)$)
\[
a_{k^*} = a_1 + \sum _{k = 1}^{k^*-1} (a_{k+1} - a_k) \geq \frac{1}{16} (k^*-1) \epsilon^{4/3} l(\partial P_1) - l(\partial P_1) \geq  l(\partial P_1) = 4N\,.
\]

The construction above captures the main idea of the proof for the lower bound in Theorem~\ref{thm-greedy-lattice-animal}: while we ignored a number of technical details and we carried out the analysis in the continuum, it is straightforward to complete a formal argument.
We will not do so since the proof of the upper bound on the correlation length contains a complete argument which is strictly more involved than the proof of the lower bound on the greedy lattice animal (formally one can follow the proof in Section~\ref{sec :: Upper bound proof} with $\Gamma(A) = \sum_{v\in A} \epsilon h_v$).

While the above construction suggests the emergence of the $4/3$ exponent in RFIM, it falls short of establishing either the upper or lower bound on the correlation length in Theorem~\ref{thm :: main result}.
In the next two subsections, we will point out the main obstacles and describe at an overview level our approaches to address these challenges.

\subsection{Upper bound on correlation length}\label{sec :: overview-upper}
Our goal is to prove that for every $\mathtt m \in (0,1)$ there exists $C_1 = C_1(\mathtt m) > 0$ (independent of $\beta$) such that for all $\epsilon \in (0,1)$ and $N \geq \exp(C_1 \epsilon^{-4/3})$,
\begin{equation} \label{eq-goal-upper-bound}
\mathtt m_{\beta, \Lambda_{4N}, \epsilon} \leq \mathtt m.
\end{equation}
(We have used $4N$ instead of $N$ in the above for later notational convenience.)
While the construction in Section~\ref{sec :: emergence} hints at the emergence of the $4/3$ exponent,
the following is a main obstacle in making this a rigorous proof for the upper bound on the correlation length even in the special case when $\beta = \infty$:
the existence of $A \in \mathfrak A_{4N}$ such that $\epsilon \sum_{v \in A} h_v > |\partial A|$ is not sufficient for $\sigma^{-}_o(\Lambda_{4N},\epsilon h) = 1$ (e.g., if $\epsilon h_v = 20$ for some $v\sim o$ and $\epsilon h_o = -5$, then $A = \{v, o\}$ satisfies the desired property but $\sigma^{-}_o(\Lambda_{4N},\epsilon h) = -1$; this is because when $|\epsilon h_v| > 4$ the ground state at $v$ agrees with the sign of $h_v$).
To overcome this challenge, we will define a suitable $\Gamma$-function for general $\beta$, and in the special case of $\beta = \infty$ the function (very roughly speaking) can certify $\sigma^{-}_o(\Lambda_{4N},\epsilon h) = 1$ (the rigorous meaning of this is via an argument by contradiction).
For $\Omega \subset \mathbb Z^2$ and an external field $f: \mathbb Z^2 \mapsto \mathbb R$, we define $H^\pm(\sigma, \Omega, f)$ and $\mu^\pm_{\beta, \Omega, f}$ as in \eqref{eq-def-hamiltonian} and \eqref{eq-def-mu} except replacing $\Lambda_N, \epsilon h$ by $\Omega, f$.
Define the free energy
\begin{equation}\label{eq-def-free-energy}
F^\pm(\Omega, f) = F^\pm(\beta, \Omega, f) = -\frac{1}{\beta} \log \sum_{\sigma\in \{-1, 1\}^\Omega} e^{-\beta H^\pm(\sigma, \Omega, f)}\,.
\end{equation}
For $A\subset \Omega$, our $\Gamma$-function is defined to be the difference of the free energies on $\Omega\setminus A$ with respect to the positive and negative boundary conditions, as follows:
\begin{equation}\label{eq :: Gamma def}
\Gamma(A, \Omega, f) = \Delta F(\Omega \setminus A, f) \mbox{ where } \Delta F(B,f) = F^+(B, f) - F^-(B, f)\,.
\end{equation}
Before proceeding, we make a few remarks about why we choose the $\Gamma$-function as the difference of free energies on $\Omega\setminus A$ instead of $A$.
In our analysis, we will let the reference domain be $\Omega = \Lambda_{2N}$ and construct
a sequence $(A_n)_{n \geq 1}$ with increasing (expected) value of $\Gamma$.
To this end, we need the increment $\Gamma(A \cup B) - \Gamma(A)$ to have nice monotonicity properties as a function of $\epsilon h$ so that we can keep track of the probabilistic behavior of the increment when employing a recursive construction as in Subsection~\ref{sec :: emergence}. The choice of $\Omega\setminus A$ gives the desired direction of monotonicity; see Lemma~\ref{lm :: Gamma increment monotonicity}.

With $\Gamma$ defined as in \eqref{eq :: Gamma def}, our proof proceeds by demonstrating a contradiction if we assume \eqref{eq-goal-upper-bound} fails.
On the one hand, we have the following upper bound (c.f. \cite[Proposition 5.2(iii)]{AizenmanWehr90}).
\begin{lemma}\label{lm :: Gamma bound}
$|\Gamma(A,\Omega,f)| \leq 2|\partial(\Omega\setminus A)|$ for all $(A, \Omega,f)$ with $A\subset \Omega$.
\end{lemma}
The proof in the case of the Ising model is elementary. It follows from the fact that $|H^+(\sigma,B,f) - H^{-}(\sigma,B,f)| \leq 2|\partial B|$ and
\[
\Delta F(\beta,B, f) = -\frac{1}{\beta}\log\left( \langle \exp(-\beta[H^{+}(\sigma, B, f) - H^{-}(\sigma, B, f)])\rangle_{\mu^{-}_{\beta,B,f}}\right)\,.
\]
On the other hand, assuming \eqref{eq-goal-upper-bound} fails, we will show that the variance of the increment $\Gamma(A \cup B, \Lambda_{2N}, \epsilon h) - \Gamma(A,  \Lambda_{2N}, \epsilon h)$  is comparable to that of $\sum_{v \in B} \epsilon h_v$ and then we can hope to follow the argument in Subsection~\ref{sec :: emergence} to construct a set whose $\Gamma$-function value is larger than its boundary size.
As mentioned, a crucial feature we use in proving this is a monotonicity property for the increment of the $\Gamma$-function, as incorporated in Lemma~\ref{lm :: Gamma increment monotonicity}.

With all these intuitions in place, the actual proof in Section 3 is written in a way that both fills in the gaps left by the heuristics from Section 2.2 and addresses the challenges from the random field Ising model. For the former, for instance, Figure~\ref{fig :: P, T, and T*} illustrates how we address the gap from correlations between different rounds of recursive constructions by making the decision for the triangle $T_{1, i}$ only based on disorder in the smaller blue triangle $T^*_{1, i}$. For the latter,  Lemma~\ref{lm :: stochastic domination} manifests the power of  Lemma~\ref{lm :: Gamma increment monotonicity} and says that the correlation through the Ising measure is in our desirable direction and Lemma~\ref{lm :: lower bound probability of inclusion} says that the marginal effect from the disorder in a triangle to our observable is similar to the white noise value of this triangle.

\subsection{Lower bound on correlation length} \label{sec:overview-lower}

In light of \eqref{eq :: ground state flip condition}, the lower bound on the correlation length for $\beta = \infty$  can be proved via an upper bound on the greedy lattice animal.
This is an example of the classic question of computing the (expected) supremum of a Gaussian process.
This has been well-understood in general, culminating in Talagrand's majorizing measure theorem in \cite{Talagrand87}, which improved previous results in \cite{Dudley67, Fernique71}:
as a highlight, an up-to-constant estimate for the supremum of a general Gaussian process was provided in terms of the (so-called) $\gamma_2$-functional associated with this process. Specifically for the example of our lattice animal process, the upper bound was already hinted in \cite{LS89} as we explained earlier, whose proof together with proofs for various results on matching problems were unified and streamlined in \cite{Talagrand14}. In particular, the following result was essentially contained in \cite{Talagrand14}.
\begin{proposition}\label{prop :: correlation length lower bound}
	Let $\mathfrak B_N$ be the collection of simply connected lattice animals contained in $\Lambda_N$. There exists a constant $C_1 > 0$ such that for
	$N > 1$ we have
	\[
	\mathbb{P} \big(
	\max_{B \in \mathfrak B_{N}} \frac{\sum_{v \in B}h_v}{|\partial B|} > C_1 (\log N)^{3/4} + u
	\big) \leq \exp (-u^2/2) \quad \forall u > 0\,.
	\]
\end{proposition}

To conclude this section, we prove the lower bound in Theorem~\ref{thm :: main result} and the upper bound in Theorem~\ref{thm-greedy-lattice-animal} using  Proposition~\ref{prop :: correlation length lower bound}.
\begin{proof}[Proof of lower bound in Theorem~\ref{thm :: main result} and upper bound in Theorem~\ref{thm-greedy-lattice-animal}]
The main step of the proof is relating the bound on simply connected lattice animals to a bound on lattice animals. Let $\mathcal B_N$ be the collection of connected lattice animals contained in $\Lambda_N$. We claim that 
\begin{equation}\label{eq-from-simply-connected-to-connected}
\max_{B \in \mathfrak B_N} \frac{|\sum_{v \in B}h_v|}{|\partial B|} = 
	\max_{B' \in \mathcal B_N} \frac{|\sum_{v \in B'} h_v|}{|\partial B'|}\,.
\end{equation}
For any lattice animal $B$, let $\tilde B$ be the collection of vertices that is enclosed by $B$, i.e., disconnected by $B$ from $\infty$.
Let $B_1, \ldots, B_k$ be the connected components of $\tilde B \setminus B$.
Note that $B_1, \ldots, B_k$ are simply connected, since if $v$ is separated from $\infty$ by $B_i$ then $v \in \tilde B$, and in addition since $B$ is connected it follows $v\not\in B$.
Since $\partial B$ is the disjoint union of $\partial \tilde B$ and $\partial B_1,\dots,\partial B_k$, we have
\begin{align*}
\frac{|\sum_{v \in B} h_v|}{|\partial B|} 
	&\leq \frac{1}{|\partial B|} \sum_{i = 1}^k \big|\sum_{v \in B_i}h_v \big| \\
	&= \sum_{i = 1}^k \frac{|\partial B_i|}{|\partial B|} \frac{|\sum_{v \in B_i}h_v|}{|\partial B_i|} \\
	&\leq \max_{i = 1,\dots,k} \frac{|\sum_{v \in B_i}h_v|}{|\partial B_i|},
\end{align*}
where the last inequality follows from the fact that the coefficients $|\partial B_i|/|\partial B|$ sum up to 1.
This completes the verification of \eqref{eq-from-simply-connected-to-connected}.
By Proposition~\ref{prop :: correlation length lower bound} (and the fact that $h$ is symmetric), the maximum on the left-hand side of \eqref{eq-from-simply-connected-to-connected} is of order $(\log N)^{3/4}$, which proves the upper bound in Theorem~\ref{thm-greedy-lattice-animal}. This also shows that it is less than $\epsilon^{-1}$ with high probability as long as $N \leq \exp(\epsilon^{4/3}/C)$. 
By \eqref{eq :: ground state flip condition} (and a symmetric condition for $\sigma^+_o(\Lambda_N, \epsilon h)$) this implies that $\sigma_o^{\pm}(\Lambda_N, \epsilon h)  = \pm1$ with high probability and thus completes the proof of the lower bound in Theorem~\ref{thm :: main result}.
\end{proof}
\section{Upper bound on correlation length}\label{sec :: Upper bound proof}

This section is devoted to the proof of the upper bound on the correlation length, as incorporated in \eqref{eq-goal-upper-bound}.
Recall the definition of $\Gamma$-function given in \eqref{eq :: Gamma def}.
Recall from Lemma~\ref{lm :: Gamma bound} that $\Gamma(A,\Omega, f) \leq 2|\partial (\Omega\setminus A)|$ for all $(A, \Omega,f)$.
With this at hand, the bulk of this section is to show that if \eqref{eq-goal-upper-bound} fails, there exists a random subset $\mathsf P^* \subset \Lambda_{2N}$ such that
\begin{equation}\label{eq :: contradiction}
\mathbb E[ \Gamma(\mathsf P^*, \Lambda_{2N},\epsilon h) - 2 |\partial (\Lambda_{2N}\setminus\mathsf P^*)|] > 0\,,
\end{equation}
which is a contradiction.
As mentioned in Subsection~\ref{sec :: overview-upper}, a key element of our analysis is a monotonicity property of the $\Gamma$-function which we incorporate in Lemma~\ref{lm :: Gamma increment monotonicity}.
In Subsection~\ref{sec :: randomized geometric construction}, we construct $\mathsf P^*$ by enhancing the procedure in Subsection~\ref{sec :: emergence} in order to address additional
 complications due to the complexity of the $\Gamma$-function.
In Subsection~\ref{sec :: probability analysis} we carry out the probabilistic analysis and prove \eqref{eq :: contradiction} under the assumption that \eqref{eq-goal-upper-bound} fails.

\subsection{Monotonicity property of the $\Gamma$-function}\label{sec :: Gamma property}

\begin{lemma}\label{lm :: Gamma increment monotonicity}
For disjoint subsets $A,B \subset \Omega$, we have that
	$
	(\Gamma(A \cup B, \Omega,f) - \Gamma(A, \Omega ,f))
	$
	is increasing in $\{f_v \,:\, v \in B\}$, decreasing in
	$\{f_v \,:\, v \notin A \cup B\}$ and does not depend on $\{f_v: v\in A\}$.
\end{lemma}
\begin{proof}
Recall the definition of $\Delta F$ in \eqref{eq :: Gamma def}. Write
\[
\Delta \langle \sigma_v \rangle_{\beta,\Omega, f} =
	\frac{1}{2}(\langle\sigma^+_v\rangle_{\mu^+_{\beta,\Omega, f}} - \langle\sigma^-_v\rangle_{\mu^-_{\beta,\Omega, f}})\,.
\]
 We compute partial derivatives and get that
\begin{equation}\label{eq-derivative-free-energy}
\partial_{f_v}\Delta F(A') = -2\Delta \langle \sigma_v \rangle_{\beta, A', f} \mathbf 1_{v\in A'}
\end{equation}
 for any $A'\subset \mathbb Z^2$ (where the minus sign inherits from that in the definition of free energy). Write
\begin{align*}
	G(A, B, \Omega, f)& = \Gamma(A \cup B, \Omega, f) - \Gamma(A , \Omega, f) \\
		&= \Delta F(\Omega \setminus (A \cup B), f) - \Delta F(\Omega \setminus A, f). \numberthis \label{eq-Gamma-expression}
\end{align*}
Using \eqref{eq-derivative-free-energy} and the  monotonicity of the Ising model (c.f. \cite[Section 2.2]{AizenmanPeled19}) we get that for $v \in \Omega \setminus (A \cup B)$
\[
\partial_{f_v}G(A, B, \Omega, f) = 2(\Delta \langle \sigma_v \rangle_{\beta,\Omega\setminus A, f} - \Delta \langle \sigma_v \rangle_{\beta,\Omega\setminus (A\cup B), f})  \leq 0\,,
\]
for $v \in B$
\[
\partial_{f_v}G(A, B, \Omega, f) = 2\Delta \langle \sigma_v \rangle_{\beta,\Omega \setminus A, f} \geq 0\,,
\]
and for $v \in A$, $\partial_{f_v}G(A, B, \Omega, f) = 0$.
This completes the proof of the lemma.
\end{proof}
It is also worth noting that it follows from the expressions obtained for the partial derivatives of $G$ that
\begin{equation}\label{eq-Gamma-partial-bound}
|\partial_{f_v}G(A,B,\Omega,f)| \leq 2 \mbox{ for all } A,B \subset \Omega \mbox{ and } v\in \Omega.
\end{equation}

\subsection{Randomized geometric constructions} \label{sec :: randomized geometric construction}

In this subsection we give the details of the construction of the random set $\mathsf P^*$ (following Subsection~\ref{sec :: emergence}) and prove a few geometric lemmas.

\subsubsection{Construction of $\mathsf P^*$}

In order to construct $\mathsf P^*$, we will recursively construct a sequence of polygons $(P_n)_{n\geq 1}$ contained in $[-2N,2N]^2$ and a corresponding sequence of subsets $(\mathsf P_n)_{n \geq 1}$ given by $\mathsf P_n = P_n \cap \mathbb Z^2$.
Let $\mathtt m \in (0,1)$ and let $\delta = 10^{-2}(\epsilon \mathtt m)^{2/3}$ (where $10^{-2}$ is chosen as a small but otherwise arbitrary constant).
As initialization for our procedure, we set $P_1 = [-N,N]^2$ and let $(S_{1,i})_{i = 1}^4$ be the sides of $P_1$, numbered in counter-clockwise order with $S_{1,1}$ being the bottom side.
We next describe our recursive construction.

For $n \geq 1$, assume $P_{n}$ has been constructed and that $P_n$ has $4^n$ sides $(S_{n,i})_{i=1}^{4^{n}}$ numbered in counter-clockwise order.
For each $i$, let $r_{n,i} = l(S_{n,i})/4$ and partition $S_{n,i}$ into four segments of length $r_{n,i}$.
Let $T_{n,i}$ be the isosceles triangle with base given by the two middle segments of $S_{n,i}$ and height $\delta r_{n,i}$ such that $T_{n,i}$ points out from $P_n$ (note that $T_{n, i}$ is measurable with respect to $P_n$; see Remark~\ref{remark-construction} (ii)).
Let $\mathsf T_{n,i} = T_{n,i}\cap \mathbb Z^2$.
Further, let $T^*_{n,i} \subset T_{n,i}$ be the triangle consisting of all points in $T_{n,i}$ which have distance at least $2\delta r_{n,i}/3$ from the base and let $\mathsf T^*_{n,i} = T^*_{n,i} \cap \mathbb Z^2$.
See Figure~\ref{fig :: P, T, and T*} for an illustration. We will decide whether to add the triangle $T_{n, i}$ to the polygon based on the current polygon and the field in $T^*_{n,i}$ only (instead of the field
in $T_{n, i}$); this ensures that
our construction explores disjoint regions in different iterations (see Lemma~\ref{lm :: Independence}).

\begin{figure}
\centering
\begin{subfigure}{.5\textwidth}
\centering
\includegraphics[width=.6\linewidth]{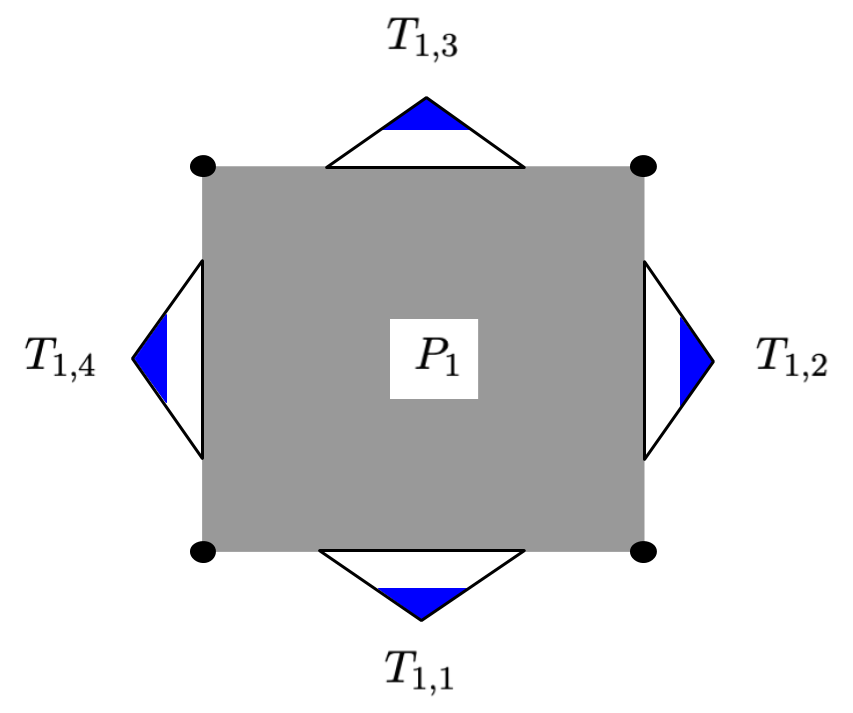}
\end{subfigure}%
\begin{subfigure}{.5\textwidth}
\centering
\includegraphics[width=.6\linewidth]{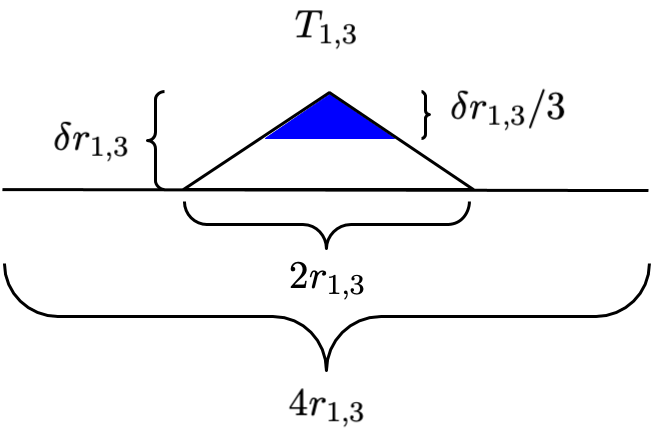}
\end{subfigure}
\caption{$P_1$ with $(T_{1,i})_{i = 1}^4$. The blue triangles are $(T^*_{1,i})_{i = 1}^4$.}
\label{fig :: P, T, and T*}
\end{figure}

In order to construct $P_{n+1}$, we will decide whether to add the triangle $T_{n,i}$ for $1\leq i\leq 4^n$ depending on whether the expected increase to the value of the $\Gamma$-function is larger than the resulting increase in the boundary size of the polygon.
To formalize this idea, we will recursively define a sequence of polygons $(P_{n,i})_{i =0}^{4^n}$ and their corresponding lattice subsets $\mathsf P_{n,i} = P_{n,i}\cap \mathbb Z^2$. For the base case, we let $P_{n,0} = P_n$.
For $1 \leq i \leq 4^n$ let $\mathcal F_{n,i}$ be the $\sigma$-algebra generated by $P_{n,i-1}$ and $\{h_v : v \in \mathsf T^*_{n,i}\}$ (by definition $T_{n, i}$ is measurable with respect to $P_{n}$ as mentioned earlier, and thus from our recursive construction below $T_{n, i}$ is also measurable with respect to $P_{n, i-1}$, as elaborated in Remark~\ref{remark-measurability}).
Note that $\mathcal F_{n, i}$ is not increasing.
In particular, $\mathcal F_{n, i}$ contains information about $\{h_v: v\in \cup_{i'=1}^{i-1} \mathsf T^*_{n, i'}\}$ only via $P_{n, i-1}$.
Define
\begin{equation}\label{eq-def-gamma-n-i}
\gamma_{n,i} =
\mathbb E [
\Gamma(\mathsf P_{n,i-1} \cup \mathsf T_{n,i}, \Lambda_{2N}, \epsilon h) -
\Gamma(\mathsf P_{n,i-1}, \Lambda_{2N}, \epsilon h) \mid \mathcal F_{n,i}
]\,,
\end{equation}
as the aforementioned expected increment of $\Gamma$ (see Remark~\ref{remark-construction} (iv)).
Then we define
\begin{align*}
P_{n,i} &= \begin{cases}
P_{n,i-1} \cup T_{n,i}, & \mbox{ if } \gamma_{n,i} \geq 10\delta^2r_{n,i}\,, \\
P_{n,i-1}, & \mbox{ if } \gamma_{n,i} < 10\delta^2r_{n,i}\,.
\end{cases}
\end{align*}
We let $P_{n+1} = P_{n,4^n}$ and let $(S_{n+1,i})_{i = 1}^{4^{n+1}}$ be the sides of $P_{n+1}$, numbered in counter-clockwise order so that $S_{n+1,j} \subset S_{n,i} \cup T_{n,i}$ for $1 \leq i \leq 4^n$ and $4(i-1) + 1 \leq j \leq 4i$.
That is, for each $i$ the sides of $P_{n+1}$ that ``come from'' $S_{n,i}$ are $(S_{n+1,j})_{j = 4(i-1)+1}^{4i}$. This concludes the construction of $P_{n+1}$.

Finally, we let $n^* = \lfloor \log_{16}(N)\rfloor$, $P^* = P_{n^*}$, and $\mathsf P^* = \mathsf P_{n^*}$.
This choice of $n^*$ ensures that $\delta r_{n,i}$ is large for all $n \leq n^*$, which will allow us to approximate $|\mathsf T^*_{n,i}|$ by the area of $T^*_{n,i}$.

Before proceeding, we make a few expository remarks on our construction.

\begin{remark}\label{remark-construction}
(i) We have assumed that for each $n \geq 1$, the triangles $(T_{n,i})_{i = 1}^{4^n}$
are disjoint and $T_{n,i} \cap P_n \subset S_{n,i}$ for all $i$. This is justified
by Lemma~\ref{lm :: No self-crossing}. We also note that if
$\gamma_{n,i} \leq 10\delta^2r_{n,i}$ (i.e. $T_{n,i}$ is not included in $P_{n+1}$),
then $S_{n,i}$ is split into four sides of $P_{n+1}$ with internal angle $\pi$ between
them. These two assumptions ensure $P_{n+1}$ is a polygon with $4^{n+1}$
sides.

(ii) It will be useful in our proof that the numbering of the sides of $P_n$
is deterministic so that the sequence $(T_{n,i})_{i = 1}^{4^n}$ is measurable with
respect to $P_n$. The specific choice given in the construction is made for convenience.

(iii) Our choice of $\delta$ is based on similar considerations to those given in Subsection~\ref{sec :: emergence}.
The condition $\gamma_{n,i} > 10 \delta^2 r_{n,i}$ is based on the following calculation.
Since $l(\partial (P_{n,i-1} \cup T_{n,i})) - l(\partial P_{n,i-1}) \leq \delta^2r_{n,i}$, if adding $T_{n,i}$ to $P_{n,i-1}$ increases $\Gamma$ by $10 \delta^2 r_{n,i}$ then the difference between $\Gamma$ and $8l(\partial P)$ will increase (the constant 8 will be explained in subsection~\ref{sec :: Gaussian volume proofs}).

(iv)  Note that $\gamma_{n,i}$ depends on $\{h_v: v\in \cup_{(k,j) < (n,i)}\mathsf T^*_{k,j}\}$ only through $P_{n,i-1}$ due to our particular choice of $\mathcal F_{n, i}$ (here $(k, j) < (n, i)$ if $k<n$, or $k=n$ and $j<i$).
The reason we choose $\mathcal F_{n, i}$ this way is that we can show the expected value of the derivative of the increment with respect to $h_v$ for $v \in \mathsf T^*_{n,i}$ is bounded from below
by $\mathtt m_{\beta, \Lambda_{4N}, \epsilon}$, which is at least $\mathtt m$ by our assumption that \eqref{eq-goal-upper-bound} fails. Therefore, the lower bound on the variance obtained this way is comparable to the upper bound
on the variance obtained from general Gaussian concentration inequality, and this is a useful property in our analysis later. If instead $\gamma_{n,i}$ was defined by conditioning on the field in $T^*_{k,j}$ for
$(k,j) \leq (n,i)$, then the field in previously rejected triangles would affect $\gamma_{n,i}$ but potentially only very weakly.
 This would mean our lower bound on the variance of $\gamma_{n,i}$ would be much smaller than the upper bound from Gaussian concentration inequality since now the upper
  bound would be from the field in a much larger region.
\end{remark}

\subsubsection{Geometric lemmas}
In this subsection we prove a few lemmas which ensure that the polygons $(P_n)_{n = 1}^\infty$ have desirable geometric properties.
\begin{lemma}\label{lm :: No self-crossing}
	For all $n \geq 1$, the triangles $(T_{n,i})_{n = 1}^{4^n}$ are disjoint and
	$T_{n,i} \cap P_n \subset S_{n,i}$ for $1 \leq i \leq 4^n$.
\end{lemma}
\begin{lemma}\label{lm :: polygons are in box}
	$T_{n,i} \subset [-2N,2N]^2$ for all $n \geq 1$ and $1 \leq i \leq 4^n$.
\end{lemma}
We first state and prove a lemma which easily implies Lemmas~\ref{lm :: No self-crossing} and~\ref{lm :: polygons are in box}.
We begin with some notation.
Let $\mathcal I = \{(n,i) \,:\, n \geq 1,\, 1 \leq i \leq 4^n\}$ and $\mathcal{G}$ be the directed forest with vertex set $\mathcal I$ and edge set
\[
\{((n,i), (n+1,j)) \,:\, 4(i-1)+1 \leq j \leq 4i\}\,.
\]
That is, there is an edge from $(n,i)$ to $(n+1,j)$ if $S_{n+1,j} \subset S_{n,i} \cup T_{n,i}$.
In this case we say $S_{n+1,j}$ is a child of $S_{n,i}$ (or $S_{n,i}$ is the parent of $S_{n+1,j}$).
We let $\mathcal G_{n,i}$ be the subtree of $\mathcal G$ rooted at $(n,i)$.
That is, the subgraph of $\mathcal G$ on the vertices $(k,j) \in \mathcal I$ for which there exists a directed path from $(n,i)$ to $(k,j)$.
If $(k,j)\in \mathcal G_{n,i}$ we call $(k,j)$ a descendant of $(n,i)$.
\begin{lemma}\label{lm :: Bound on children}
	Let $(n,i)\in \mathcal I$ and $\mathcal{T}_{n,i}$ be the isosceles triangle
	with base $S_{n,i}$ and height $2\delta r_{n,i}$ that contains $T_{n,i}$. Then
	for every $(k,j) \in \mathcal G_{n,i}$ we have
	$\mathcal T_{k,j} \subset \mathcal T_{n,i}$.
\end{lemma}
See Figure~\ref{fig :: P, T, and fancy T} for an illustration of $(\mathcal T_{1,i})_{i = 1}^4$.
\begin{figure}
	\centering
	\begin{subfigure}{.5\textwidth}
		\centering
		\includegraphics[width=.6\linewidth]{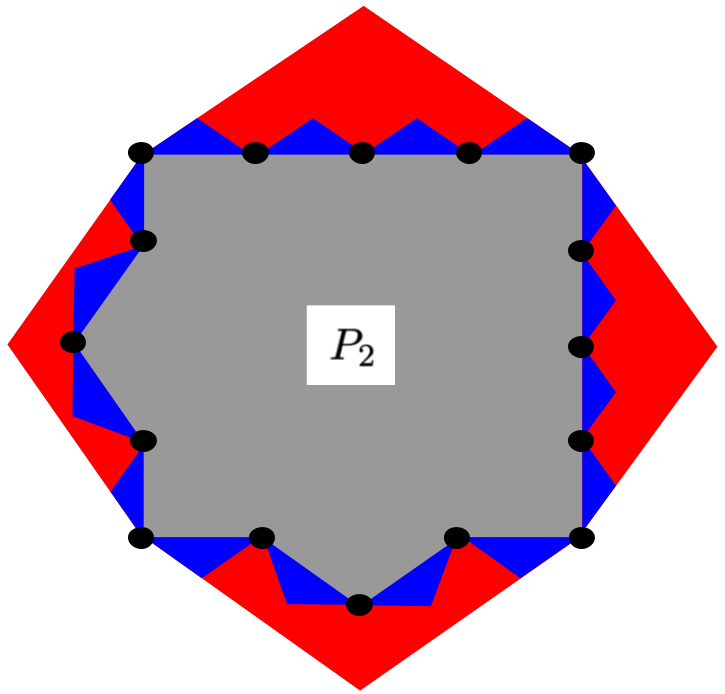}
	\end{subfigure}%
	\begin{subfigure}{.5\textwidth}
		\centering
		\includegraphics[width=.6\linewidth]{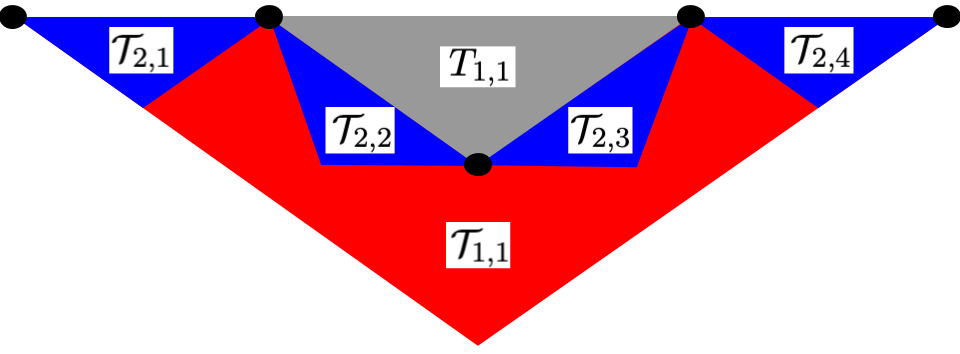}
	\end{subfigure}
	\caption{$P_2$ with $\cup_{i=1}^{16}\mathcal T_{2,i}$ in blue and $\cup_{i=1}^4\mathcal T_{1,i}\setminus \cup_{i=1}^{16}\mathcal T_{2,i}$ in red.}
	\label{fig :: P, T, and fancy T}
\end{figure}
\begin{proof}
	It suffices to show that if $(n+1,j)$ is a child of $(n,i)$ then $\mathcal{T}_{n,i}$ contains $\mathcal T_{n+1,j}$.
	For concreteness, we take $i = 1$ and therefore $1 \leq j \leq 4$.
	It is immediate that $\mathcal{T}_{n+1,j} \subset \mathcal{T}_{n,1}$ for $ j\in \{1,4\}$ and that $\mathcal{T}_{n+1,j} \subset \mathcal{T}_{n,1}$ for $j \in \{2,3\}$ if $T_{n,1}$ is not contained in $P_{n+1}$.
	Assuming $T_{n,1} \subset P_{n+1}$, we can use the fact that $T_{n,1}$ is similar to $\mathcal{T}_{n,1}$ (and in
	fact their sides are parallel) to show that the distance between $\partial T_{n,1} \setminus S_{n,1}$ and $\partial \mathcal{T}_{n,1} \setminus S_{n,1}$ is given by
	$
	d_{n,1} = \frac{\delta}{\sqrt{1 + \delta^2}}r_{n,1}.
	$
	See Figure~\ref{fig :: d_n} for an illustration.
	Further, the height of $\mathcal{T}_{n+1,2}$ and $\mathcal{T}_{n+1,3}$ is given by $\frac{\delta\sqrt{1+\delta^2}}{2}  r_{n,1}$.
	Since $\delta < 1$, this height is strictly smaller than $d_{n,1}$ and therefore $\mathcal{T}_{n+1,2}$ and $\mathcal{T}_{n+1,3}$ are contained in $\mathcal T_{n,1}$ as claimed.
\end{proof}
\begin{figure}
	\centering
	\includegraphics[width=.6\linewidth]{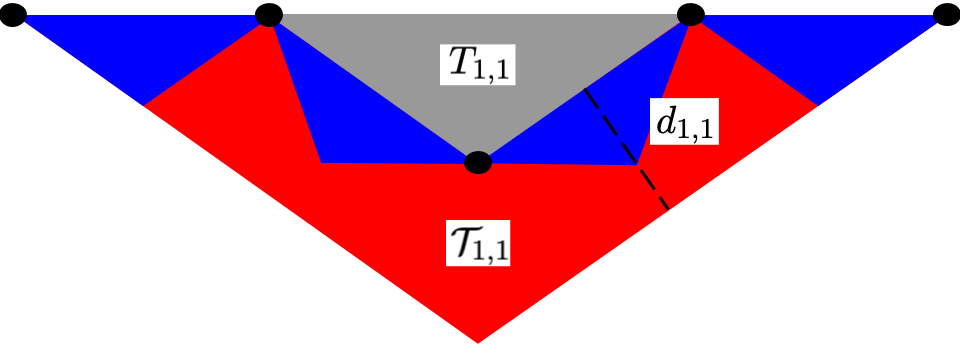}
	\caption{$d_{1,1}$ is the distance between $\partial T_{1,1} \setminus S_{1,1}$ and $\partial \mathcal T_{1,1}\setminus S_{1,1}$.}
	\label{fig :: d_n}
\end{figure}
\begin{proof}[Proof of Lemma~\ref{lm :: No self-crossing}]
	Let $\theta = \arctan(\delta)$ and note that $\theta$ is the internal angle (with respect to $\mathcal{T}_{n,i}$) between $S_{n,i}$ and the other sides of $\mathcal{T}_{n,i}$.
	The same holds for $T_{n,i}$.
	Since $\delta < 1$, we have $\theta < \pi/4$.
	
	It suffices to show that for every $(n,i), (n,j) \in \mathcal I$ we have $\mathcal T_{n,i} \cap \mathcal T_{n,j} = S_{n,i} \cap S_{n,j}$.
	We prove this by induction.
	It clearly holds for $P_1 = [-N,N]^2$.
	By Lemma~\ref{lm :: Bound on children}, if it holds for $P_n$ then
	$\mathcal{T}_{n+1,i} \cap \mathcal{T}_{n+1,j} = S_{n+1,i} \cap S_{n+1,j}$ when $(n+1,i)$ and $(n+1,j)$ are not siblings (i.e. they do not have the same parent).
	When $(n+1,i)$ and $(n+1,j)$ are siblings, it is immediate that $\mathcal{T}_{n+1,i} \cap \mathcal{T}_{n+1,j} = \emptyset$ unless $S_{n+1,i}$ and $S_{n+1,j}$ are adjacent (i.e. $|i - j|$ = 1).
	Assuming $S_{n+1,i}$ and $S_{n+1,j}$ are adjacent we note that the external (with respect to $P_{n+1}$) angle between them is at least $\pi - \theta$.
	Recall that the internal (with respect to $\mathcal{T}_{n+1,i}$) angle between $S_{n+1,i}$ and the other sides of $\mathcal{T}_{n+1,i}$ is $\theta$ and the same holds for $j$.
	Since $3 \theta < \frac{3\pi}{4} < \pi$, we see that $\mathcal{T}_{n+1,i} \cap \mathcal{T}_{n+1,j} = S_{n+1,i} \cap S_{n+1,j}$ (see Figure~\ref{fig :: 3 theta} for an illustration of this argument).
\end{proof}
\begin{figure}
	\centering
	\includegraphics[width=.6\linewidth]{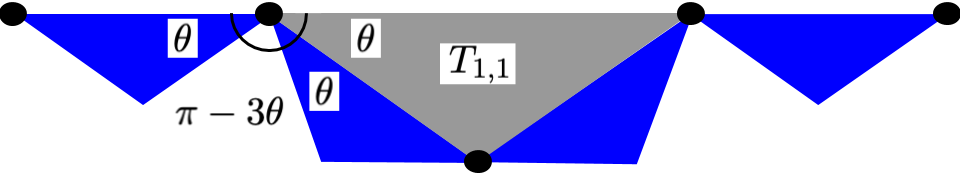}
	\caption{The fact $\theta < \pi/3$ ensures that $\mathcal T_{2,1}$ and $\mathcal T_{2,2}$ intersect only at their common vertex.}
	\label{fig :: 3 theta}
\end{figure}

\begin{proof}[Proof of Lemma~\ref{lm :: polygons are in box}]
	Since $\delta < 1/2$ we have $\mathcal T_{1,j} \subset [-2N,2N]^2$ for $j = 1,2,3,4$ so the conclusion follows from Lemma~\ref{lm :: Bound on children}.
\end{proof}

We prove a few more lemmas that will be useful for probabilistic analysis in Subsection~\ref{sec :: probability analysis}.
\begin{lemma}\label{lm :: Perimeter lemma}
	Let $P$ be a polygon with $q$ sides and $\mathsf P = P \cap \mathbb{Z}^2$. Then
	$
	|\partial \mathsf P| \leq \sqrt{2} l(\partial P) + 2q
	$.
\end{lemma}
\begin{proof}
	Note that $|\partial \mathsf P|$ is bounded above by the number of edges that intersect $\partial P$ (if $\partial P$ contains a vertex in $\mathbb{Z}^2$ we count this as two intersections).
	In addition, the number of edges intersecting any line segment is upper bounded by 2 plus the $\ell_1$ distance between its endpoints which is in turn bounded by 2 plus $\sqrt{2}$ times the Euclidean length of the segment.
	This yields the desired bound.
\end{proof}

\begin{lemma}\label{lm :: Independence}
	Let $(n,i)$, $(k,j) \in \mathcal I$ be such that $(n,i) \neq (k,j)$.
	Then $T^*_{n,i} \cap T^*_{k,j} = \emptyset$.
\end{lemma}
\begin{proof}
	We assume without loss of generality that $n \leq k$.
	Let $j'$ be the unique integer such that $(k,j)$ is a descendant of $(n,j')$ (if $k = n$ then $j = j'$).
	By Lemma~\ref{lm :: Bound on children} we have $T^*_{n,i} \subset \mathcal T_{n,i}$ and $T^*_{k,j} \subset \mathcal T_{n,j'}$.
	We showed in the proof of Lemma~\ref{lm :: No self-crossing} that if $j' \neq i$ then $\mathcal T_{n,i} \cap \mathcal T_{n,j'} = S_{n,i} \cap S_{n,j'}$, which implies $T^*_{n,i} \cap T^*_{k,j} = \emptyset$ since $T^*_{n,i} \cap S_{n,i} = \emptyset$.
	Therefore, we assume $j' = i$ (that is, $S_{k,j}$ is a descendant of $S_{n,i}$).
	Note that this implies that $k > n$.
	To conclude the proof, we consider separately the case that $T_{n,i}$ is contained in $P_{n+1}$ and the case that it is not.
	For concreteness, we let $i = 1$.
	If $T_{n,1}$ is not contained in $P_{n+1}$, then $T^*_{n,i}$ is disjoint from $\mathcal T_{n+1,a}$ for $a \in \{1,2,3,4\}$ since the base of $\mathcal T_{n+1,a}$ is a subset of $S_{n,1}$, the height of $\mathcal T_{n+1,a}$ is $\delta r_{n,i}/2$, and $T^*_{n,i}$ consists of points with distance at least $2 \delta r_{n,1}/3$ from $S_{n,1}$ (see Figure~\ref{fig :: Tstar_out} for an illustration).
	By Lemma~\ref{lm :: Bound on children}, $T^*_{k,j} \subset \mathcal T_{n+1,a}$ for some $a \in \{1,2,3,4\}$ so it follows $T^*_{n,i}$ is disjoint from $T^*_{k,j}$.
	If $T_{n,1}$ is a subset of $P_{n+1}$, then so is $T^*_{n,1}$. By Lemma~\ref{lm :: No self-crossing}, $T^*_{k,j}$ is disjoint from $P_{k}$ which contains $P_{n+1}$ (because $k \geq n+1$), so $T^*_{n,i}$ and $T^*_{k,j}$ are disjoint.
\end{proof}
For the next lemmas, we consider $\mathcal I$ to be ordered by lexicographical ordering (i.e. $(n', i') < (n, i)$ if $n'<n$, or $n'=n$ and $i'<i$).
For $(n,i) \in \mathcal I$, let
\begin{equation}\label{eq-def-Z-n-i}
Z_{n,i} = \one_{\gamma_{n,i} > 10 \delta^2 r_{n,i}}\,.
\end{equation}
\begin{lemma}\label{lm :: T star is in or out}
	Let $(k,j),(n,i) \in \mathcal I$.
	If $(k,j) \leq (n,i)$ and $Z_{k,j} = 0$ then $T^*_{k,j} \cap P_{n,i} = \emptyset$.
\end{lemma}
\begin{proof}
	If $Z_{k,j} = 0$, then by Lemma~\ref{lm :: Bound on children} we have $P_{n,i} \cap \mathcal T_{k,j} \subset \bigcup_{a = 4(j-1) + 1}^{4j} \mathcal T_{k+1,a}$.
	Since $T^*_{k,j}$ is contained in $\mathcal T_{k,j}$, it suffices to show that if $(k+1,a)$ is a child of $(k,j)$ and $Z_{k,j} = 0$ then $T^*_{k,j} \cap \mathcal T_{k+1,a} = \emptyset$, which was shown in the proof of Lemma~\ref{lm :: Independence} (see Figure~\ref{fig :: Tstar_out} for an illustration).
\end{proof}

\begin{figure}
	\centering
	\includegraphics[width=.6\linewidth]{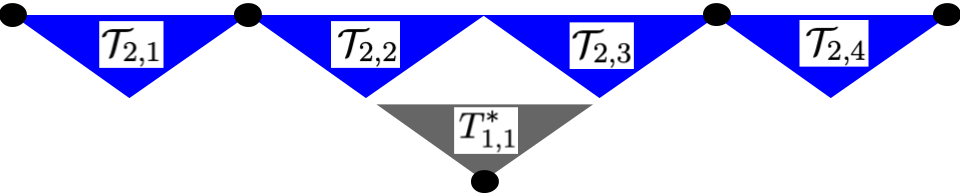}
	\caption{If $Z_{1,1} = 0$, then $T^*_{1,1}$ is disjoint from $(\mathcal T_{2,i})_{i=1}^4$.}
	\label{fig :: Tstar_out}
\end{figure}

\begin{lemma}\label{lm :: encoding}
	For $(n,i) \in \mathcal I$ the collection $\{Z_{k,j} \,:\, (k,j) \leq (n,i)\}$ is measurable with respect to $P_{n,i}$.
\end{lemma}
\begin{remark}\label{remark-measurability}
Given $\{Z_{k,j} \,:\, (k,j) < (n,i)\}$, we can recover the construction up until the $(n,i)$-th step, so we can recover $\{P_{k,j} \,:\, (k,j) < (n,i)\}$ and in particular we can recover $\{P_1, \ldots, P_n\}$.
Since $T_{k, j}$ is measurable with respect to $P_{k}$, it follows from Lemma~\ref{lm :: encoding} that the collection $\{T_{k,j} \,:\, (k,j) \leq (n,i)\}$ is measurable with respect to $P_{n,i-1}$.
\end{remark}
\begin{proof}[Proof of Lemma~\ref{lm :: encoding}]
	First, we prove that $Z_{k,j} = \one_{T_{k,j} \subset P_{n,i}}$.
	By definition, if $Z_{k,j} = 1$ then $T_{k,j} \subset P_{k,j} \subset P_{n,i}$.
	By Lemma~\ref{lm :: T star is in or out}, if $Z_{k,j} = 0$ then $T_{k,j}$ is not contained in $P_{n,i}$.
		
	Therefore, it suffices to show that $T_{k,j}$ is measurable with respect to $P_{n,i}$.
	We prove this by induction on $(k,j)$.
	It clearly holds for $k = 1$ because $(T_{1,j})_{j = 1}^4$ are deterministic.
	If $k \geq 2$ and $T_{s,a}$ is measurable with respect to $P_{n,i}$ for all $(s,a) < (k,j)$, then it follows that $\{Z_{s,a}\,:\, (s,a) < (k,j)\}$ is measurable with respect to $P_{n,i}$ and in particular $P_{k}$ is measurable with respect to $P_{n,i}$.
	Since $T_{k,j}$ is measurable with respect to $P_k$ this concludes the proof.
\end{proof}

\subsection{Probabilistic analysis of the geometric construction} \label{sec :: probability analysis}

In this subsection, we provide the probabilistic analysis of our randomized geometric construction.
A key ingredient is a resampling inequality, leveraging the monotonicity of the increments of the $\Gamma$-function established in Lemma~\ref{lm :: Gamma increment monotonicity}.

\subsubsection{A resampling inequality}
For $(n,i) \in \mathcal I$, we let
\[
B_{n,i} = \bigcup_{(k,j) \in \mathcal I,\, (k,j) \leq (n,i)} \mathsf T^*_{k,j}\,,
\]
be the set of vertices in $\mathbb Z^2$ where the external field is explored for the construction of $P_{n,i}$.
By Lemma~\ref{lm :: Independence}, $\mathsf T^*_{n,i} \cap B_{n,i-1} = \emptyset$.
\begin{lemma}\label{lm :: stochastic domination}
	For $(n,i) \in \mathcal I$, let $g$ be a random field such that $g_v = h_v$ for $v \notin B_{n,i-1}$ and $\{g_v \,:\, v \in B_{n,i-1}\}$ is a collection of independent mean-zero Gaussian variables with variance 1 that is independent of $h$.
	Recall that $\mathcal F_{n,i}$ is the $\sigma$-algebra generated by $P_{n,i-1}$ and $\{h_v \,:\, v \in \mathsf T^*_{n,i}\}$.
	Let
	\[
	\tilde \gamma_{n,i} =
		\mathbb E[\Gamma(\mathsf P_{n,i-1}\cup \mathsf T_{n,i}, \Lambda_{2N}, \epsilon g) -
		\Gamma(\mathsf P_{n,i-1}, \Lambda_{2N}, \epsilon g) \mid \mathcal F_{n,i}]\,.
	\]
	Then $\gamma_{n,i} \geq \tilde \gamma_{n,i}$ almost surely.
\end{lemma}
In words, the lemma states that if we resample the field on $B_{n,i-1}$ after constructing $P_{n,i-1}$, the expected increment to $\Gamma$ from adding $\mathsf T_{n,i}$ to $\mathsf P_{n,i-1}$ decreases.


\begin{proof}[Proof of Lemma~\ref{lm :: stochastic domination}]
Let $C_{n,i-1} = B_{n,i-1} \cap P_{n,i-1}$ and $D_{n,i-1} = B_{n,i-1} \setminus P_{n,i-1}$.
By Lemma~\ref{lm :: encoding}, $B_{n,i-1}$ is measurable with respect to $P_{n,i-1}$.
It follows that $C_{n,i-1}$ and $D_{n,i-1}$ are measurable with respect to $P_{n,i-1}$.
By Lemma~\ref{lm :: T star is in or out},
\begin{align*}
C_{n,i-1} &= \bigcup_{(k,j) \in \mathcal I,\, (k,j) < (n,i),\, Z_{k,j} = 1} \mathsf T^*_{k,j}\,, \\
D_{n,i-1} &= \bigcup_{(k,j) \in \mathcal I,\, (k,j) < (n,i),\, Z_{k,j} = 0} \mathsf T^*_{k,j}\,.
\end{align*}
Let $Q$ be a polygon such that $\mathbb P(P_{n,i-1} = Q) > 0$.
We let $B_{n,i-1}(Q)$ be the value of $B_{n,i-1}$ on the event $\{P_{n,i-1} = Q\}$, and similarly for $C_{n,i-1}(Q)$ and $D_{n,i-1}(Q)$. Note that the event $\{P_{n,i-1} = Q\}$ is measurable with respect to $h|_{B_{n,i-1}(Q)}$ (here $h|_A$ denotes the restriction of $h$ to $A$).

We claim that the event $\{P_{n,i-1} = Q\}$ is decreasing with respect to $h |_{D_{n,i-1}(Q)}$ and increasing with respect to $h|_{C_{n,i-1}(Q)}$.
That is, if $f$ is a realization of the field such that $P_{n,i-1}(f) = Q$ and $f'$ is a realization of the field such that
 $f'_v \geq f_v$ for all $v \in C_{n,i-1}(Q)$ and $f'_v \leq f_v$ for all $v \in D_{n,i-1}(Q)$, then $P_{n,i-1}(f') = Q$.
To see this, we prove inductively  $P_{k,j}(f) = P_{k,j}(f')$ for each $(k,j) \leq (n,i)$.
It clearly holds for $(k,j) = (1,0)$ since $P_{1,0} = [-N,N]^2$ deterministically.
 If $(k,j) \leq (n,i)$ and $P_{k,j-1}(f) = P_{k,j-1}(f')$,
then $\gamma_{k,j}(f) \leq \gamma_{k,j}(f')$ if $Z_{k,j}(f) = 1$ and $\gamma_{k,j}(f) \geq \gamma_{k,j}(f')$ if $Z_{k,j}(f) = 0$ (this is because  $\gamma_{k,j}$ is a function of $(P_{k, j-1}, h|_{\mathsf T^*_{k, j}})$ and is increasing in $h|_{\mathsf T^*_{k, j}}$ for fixed $P_{k, j-1}$).
This implies that $Z_{k,j}(f) = Z_{k,j}(f')$ and as a result $P_{k,j}(f) = P_{k,j}(f')$, completing the proof by induction.

By the FKG inequality for product measures \cite{FKG}, we get that conditional on $\{P_{n,i-1} = Q\}$ we have the following:
 $(h|_{C_{n,i-1}(Q)}, -h|_{D_{n,i-1}(Q)})$ stochastically dominates $(g|_{C_{n,i-1}(Q)}, -g|_{D_{n,i-1}(Q)})$ (note the minus sign for the field on $D_{n,i-1}(Q)$).
By construction, $h|_{\Lambda_{2N} \setminus B_{n,i-1}(Q)}= g|_{\Lambda_{2N} \setminus B_{n,i-1}(Q)}$ on $\{P_{n,i-1} = Q\}$.
Therefore, conditional on $\{P_{n,i-1} = Q\}$ and on $h|_{\mathsf T^*_{n,i}(Q)}$ (thus also conditional on $g|_{\mathsf T^*_{n,i}(Q)}$ since $h|_{\mathsf T^*_{n,i}(Q)} = g|_{\mathsf T^*_{n,i}(Q)}$), we deduce that the field $(h|_{\mathsf P_{n,i-1}(Q) \cup \mathsf T_{n,i}(Q)}, -h|_{\Lambda_{2N}\setminus (\mathsf P_{n,i-1}(Q) \cup \mathsf T_{n,i}(Q))})$ stochastically dominates the field
	$(g |_{\mathsf P_{n,i-1}(Q) \cup \mathsf T_{n,i}(Q)}, -g|_{\Lambda_{2N}\setminus (\mathsf P_{n,i-1}(Q) \cup \mathsf T_{n,i}(Q))})$.
Let $\Delta_Q: \mathbb R^{\Lambda_{2N}} \to \mathbb R$ be the function given by
\[
\Delta_Q(f) = \Gamma(\mathsf P_{n,i-1}(Q)\cup \mathsf T_{n,i}(Q), \Lambda_{2N}, f) -
	\Gamma(\mathsf P_{n,i-1}(Q), \Lambda_{2N}, f)\,.
\]
By Lemma~\ref{lm :: Gamma increment monotonicity}, $\Delta_Q$ is increasing in $f|_{\mathsf P_{n,i-1}(Q) \cup \mathsf T_{n,i}(Q)}$ and decreasing in $f|_{\Lambda_{2N}\setminus (\mathsf P_{n,i-1}(Q) \cup \mathsf T_{n,i}(Q))}$.
It follows that given $\{P_{n,i-1} = Q\}$ and $h|_{\mathsf T^*_{n,i}}$, we have that $\Delta_Q(\epsilon h)$ stochastically dominates $\Delta_Q(\epsilon g)$.
Since
\begin{align*}
\gamma_{n,i} \one_{P_{n,i-1} = Q} = \mathbb E [\Delta_Q(\epsilon h) \mid \mathcal F_{n,i}]\one_{P_{n,i-1} = Q}
\end{align*}
and 
\begin{align*}
\tilde \gamma_{n,i} \one_{P_{n,i-1} = Q} = \mathbb E [\Delta_Q(\epsilon g) \mid \mathcal F_{n,i}]\one_{P_{n,i-1} = Q}\,,
\end{align*}
this proves the lemma.
\end{proof}

\subsubsection{Quantitative probabilistic analysis}\label{sec :: Gaussian volume proofs}

We first show that each triangle $T_{n, i}$ has a decent probability to be included in $\mathsf P^*$.
Recall that $\delta = 10^{-2}(\mathtt m\epsilon)^{2/3}$.

\begin{lemma}\label{lm :: lower bound probability of inclusion}
For $\mathtt m \in (0,1)$, there exist constants $C_2, c_2>0$ (depending on $\mathtt m$) such that the following holds.
Suppose that \eqref{eq-goal-upper-bound} fails for some $N \geq e^{C_2 \epsilon^{-4/3}}$.
Then for all $1 \leq n \leq \log_{16}(N)$ and $1 \leq i \leq 4^n$
\[
\mathbb{P}(\gamma_{n,i} \geq 10\delta^2r_{n,i}) \geq c_2\,.
\]
\end{lemma}
\begin{proof}
In light of Lemma~\ref{lm :: stochastic domination}, in order to prove the lemma it suffices to show that for all $(n,i) \in \mathcal I$ with $n \leq \log_{16}(N)$ we have
\begin{equation}\label{eq :: tilde gamma tail bound}
\mathbb{P}(\tilde \gamma_{n,i} \geq 10\delta^2r_{n,i}) \geq c_2\,.
\end{equation}
As in the proof of Lemma~\ref{lm :: stochastic domination}, we will work conditionally on $P_{n,i-1}$.
We let $g$ be as in Lemma~\ref{lm :: stochastic domination}.
Note that $g$ is a collection of independent standard Gaussian random variables.
Recall the definition of $G$ given in \eqref{eq-Gamma-expression}. We have
\begin{equation}\label{eq :: tilde gamma}
\tilde \gamma_{n,i} = \mathbb E[
	G(\mathsf P_{n,i-1},\mathsf T_{n,i}, \Lambda_{2N}, \epsilon g) \big| P_{n,i-1}, g|_{\mathsf T^*_{n,i}}
	]\,.
\end{equation}
Since $G(A,B,\Omega,f)$ is an odd function of $f|_{\Omega\setminus A}$ for all fixed $(A,B,\Omega)$, we see $G(\mathsf P_{n,i-1},\mathsf T_{n,i}, \Lambda_{2N}, \epsilon g)$ is an odd function of $g|_{\Lambda_{2N}\setminus \mathsf P_{n,i-1}}$ when $P_{n, i-1}$ is fixed.
Since $g$ is independent of $P_{n,i-1}$ (because $P_{n,i-1}$ is measurable with respect to $h|_{B_{n,i-1}}$) and $g$ has a symmetric distribution, this implies that $\tilde \gamma_{n,i}$ is an odd function of $g|_{\mathsf T^*_{n,i}}$ when $P_{n, i-1}$ is fixed.
In particular we have
\begin{equation}\label{eq-tilde-gamm-meaan-0}
\mathbb E[\tilde \gamma_{n,i} \mid P_{n,i-1}] = 0\,.
\end{equation}
Next, we give a lower bound on the variance of $\tilde \gamma_{n,i}$.
By \eqref{eq :: tilde gamma} and the formulas derived in the proof of Lemma~\ref{lm :: Gamma increment monotonicity} for the partial derivatives of the increment of $\Gamma$ we obtain that for $v\in \mathsf T^*_{n, i}$
\[
\mathbb E[\partial_{g_v} \tilde \gamma_{n,i} \mid P_{n,i-1}] =
	2\epsilon \mathbb E \left[
		\Delta\langle\sigma_v\rangle_{\Lambda_{2N} \setminus \mathsf P_{n,i-1}, \epsilon g} \mid P_{n,i-1}
		\right]\,.
\]
Recall the definition of $\mathtt m_{\beta, \Lambda_N, \epsilon}$ in \eqref{eq-def-magnetization}.
For $\Omega \subset \mathbb Z^2$, define $\mathtt m_{\beta, \Omega, \epsilon}$ similarly by replacing $\Lambda_N$ with $\Omega$.
By monotonicity of the Ising model (c.f. \cite[Section 2.2]{AizenmanPeled19}), we have that $\mathtt m_{\beta, \Omega, \epsilon}$ is decreasing in $\Omega$, and therefore for $v\in \mathsf T^*_{n, i}$
\[
\mathbb E[\partial_{g_v} \tilde \gamma_{n,i} \mid P_{n,i-1}]  \geq
	2\epsilon \mathtt m_{\beta,  \Lambda_{2N} - v, \epsilon} \geq 2\epsilon \mathtt m_{\beta, \Lambda_{4N}, \epsilon} \geq 2\epsilon \mathtt m\,,
\]
where the last inequality follows from our assumption that \eqref{eq-goal-upper-bound} fails.
It then follows from \cite[Proposition~3.5]{Cacoullos82} that
\begin{equation}\label{eq :: gamma variance lower bound}
\E [\tilde \gamma_{n,i}^2\mid P_{n,i-1}] = \var[\tilde \gamma_{n,i} \mid P_{n,i-1}] \geq (2 \epsilon \mathtt m)^2 |\mathsf T^*_{n,i}|\,.
\end{equation}
In addition, by \eqref{eq-Gamma-partial-bound} we see that $\tilde \gamma_{n,i}$ is a Lipschitz function of $g|_{\mathsf T^*_{n,i}}$ with Lipschitz constant $2\epsilon \sqrt{|\mathsf T^*_{n, i}|}$ (with respect to the $\ell_2$ norm) for each fixed $P_{n,i-1}$.
Therefore, by \eqref{eq-tilde-gamm-meaan-0} and by the Gaussian concentration inequality (see \cite{Borell75, SudakovTsirelson74}, and see also \cite[Theorem 2.1]{Adler90} and \cite[Theorem 3.25]{Van16}) we get that
\begin{equation}\label{eq :: gamma tail upper bound}
\E [\tilde \gamma_{n,i}^4\mid P_{n,i-1}]
	\leq  10^5\epsilon^4|\mathsf T^*_{n,i}|^2\,.
\end{equation}
A simple computation gives that for any $t>0$
\begin{align*}
\E[\tilde \gamma_{n,i}^2\mid P_{n,i-1}] &\leq t^2  + \E(\tilde \gamma_{n,i}^2 \mathbf 1_{\tilde \gamma_{n,i}^2 \geq t^2}\mid P_{n,i-1})
\\
&\leq t^2  + \sqrt{\E(\tilde \gamma_{n,i}^4
\mid P_{n,i-1})} \sqrt{ \P(\tilde \gamma_{n,i}^2\geq t^2\mid P_{n,i-1})}\,.
\end{align*}
Setting $t =  \epsilon \mathtt m \sqrt{|\mathsf T^*_{n,i}|}$ and combining with \eqref{eq :: gamma tail upper bound} and
\eqref{eq :: gamma variance lower bound}, we obtain that
\begin{equation}\label{eq-prob-lower-bound-in-terms-area}
 \P\left(\tilde \gamma_{n,i}\geq  \epsilon \mathtt m \sqrt{|\mathsf T^*_{n,i}}| \mid P_{n,i-1} \right) = \frac{1}{2} \P(\tilde \gamma_{n,i}^2\geq ( \epsilon \mathtt m)^2 |\mathsf T^*_{n,i}| \mid P_{n,i-1}) \geq 10^{-5} \mathtt m^4\,,
\end{equation}
where the first equality follows from the fact that conditioned on $P_{n,i-1}$, the law of $\tilde \gamma_{n,i}$ is symmetric around 0.
It is obvious that the number of lattice points in any isosceles triangle in $\mathbb R^2$ with base length and height larger than 100 is at least half of the area of the triangle.
Since $N \geq e^{C_2 \epsilon^{-4/3}}$ and $1 \leq n \leq \log_{16}(N)$, we have that the base length and the height of $T^*_{n, i}$ (which are $\frac{2r_{n,i}}{3}$ and $\frac{\delta r_{n, i}}{3}$ respectively) are both larger than 100 as long as $C_2$ is a large enough constant.
Therefore,
\[
|\mathsf T^*_{n,i} | \geq 2^{-1}\lambda(T^*_{n,i}) =
	18^{-1}\delta r_{n,i}^2 \,.
\]
Combined with \eqref{eq-prob-lower-bound-in-terms-area} and $\delta = 10^{-2}(\epsilon \mathtt m)^{2/3}$, it completes the proof of \eqref{eq :: tilde gamma tail bound}.
\end{proof}

We are now ready to conclude the proof on the upper bound for the correlation length.
\begin{proof}[Proof of \eqref{eq-goal-upper-bound}]
 We will prove \eqref{eq :: contradiction} provided that \eqref{eq-goal-upper-bound} fails for $N \geq e^{C_1 \epsilon^{-4/3}}$ for a large enough constant $C_1$, and thus obtain a contradiction with Lemma~\ref{lm :: Gamma bound}.
 This in turn proves \eqref{eq-goal-upper-bound}, as required.

Since for each $n \geq 1$, $P_n$ has $4^n$ sides and $n^* \leq \log_{16}(N)$, we see that $P^*$ has at most $N^{1/2}$ sides.
By construction, $l(\partial P^*) \geq l(\partial P_1) = 8N$.
Therefore, by Lemma~\ref{lm :: Perimeter lemma} we have
\[
|\partial \mathsf P^*| \leq \sqrt{2} l(\partial P^*) + 2 N^{1/2} \leq
	2 l(\partial P^*)\,.
\]
In addition, $|\partial (\Lambda_{2N}\setminus\mathsf P^*)| = |\partial \mathsf P^*| + 16N$.
Therefore, it suffices to show that
\begin{equation}\label{eq :: upper bound final target}
\mathbb{E}[\Gamma(\mathsf P^*, \Lambda_{2N}, \epsilon h) - 8l(\partial P^*)] > 0\,.
\end{equation}
For $n \geq 1$, let $X_n = \Gamma(\mathsf P_n, \Lambda_{2N}, \epsilon h) - 8 l(\partial P_n)$.
For $(n,i) \in \mathcal I$ let $X_{n,i} = \Gamma(\mathsf P_{n,i}, \Lambda_{2N},\epsilon h) - 8l(\partial P_{n,i})$, and $Y_{n,i} = X_{n,i} - X_{n,i-1}$.
We assume from now on that $n < n^*$.
We have
\[
l(\partial (P_{n,i-1} \cup T_{n,i})) - l(\partial P_{n,i-1}) =
	2 \sqrt{1 + \delta^2}r_{n,i} - 2r_{n,i} \leq
	\delta^2 r_{n,i} \,.
\]
Recalling definition of $Z_{n, i}$ as in \eqref{eq-def-Z-n-i}, we get that
\begin{itemize}
\item if $Z_{n,i} = 0$ then $P_{n,i} = P_{n,i-1}$ and thus $Y_{n,i} = 0$;
\item if $Z_{n,i} = 1$ then 
\begin{align*}
	Y_{n,i} = &(\Gamma(\mathsf P_{n,i-1} \cup T_{n, i}, \Lambda_{2N},\epsilon h)- \Gamma(\mathsf P_{n,i-1}, \Lambda_{2N},\epsilon h)) \\
	&- 8(l(\partial (P_{n,i-1} \cup T_{n,i}))
 - l(\partial P_{n,i-1}))
\end{align*}
where the difference in the perimeter is bounded by $\delta^2 r_{n,i}$.
\end{itemize}
Altogether, we have that
\[
Y_{n,i} \geq
	Z_{n,i}\left[
		\Gamma(\mathsf P_{n,i-1} \cup \mathsf T_{n,i}, \Lambda_{2N}, \epsilon  h) -
		\Gamma(\mathsf P_{n,i-1}, \Lambda_{2N}, \epsilon  h) - 8\delta^2 r_{n,i}
		\right]\,.
\]
Recalling the definition of $\gamma_{n,i}$ as in \eqref{eq-def-gamma-n-i}, we obtain
\begin{align*}
&\mathbb{E}[
	Z_{n,i}(
		\Gamma(\mathsf P_{n,i-1} \cup \mathsf T_{n,i}, \Lambda_{2N}, \epsilon \cdot h) -
		\Gamma(\mathsf P_{n,i-1}, \Lambda_{2N}, \epsilon \cdot h)
		)]\\
 = &\mathbb{E}[\mathbb E(
	Z_{n,i}(
		\Gamma(\mathsf P_{n,i-1} \cup \mathsf T_{n,i}, \Lambda_{2N}, \epsilon \cdot h) -
		\Gamma(\mathsf P_{n,i-1}, \Lambda_{2N}, \epsilon \cdot h) \mid \mathcal F_{n, i})
		)
	] \\
=& \mathbb{E}[Z_{n,i}\gamma_{n,i}] \geq \mathbb{E}[10\delta^2 r_{n,i} Z_{n,i}]\,,
\end{align*}
where we used the fact that $Z_{n, i}$ is measurable with respect to $\mathcal F_{n, i}$.
Therefore,
\[
\mathbb E[Y_{n,i}] \geq
	2\delta^2 \mathbb E[r_{n,i}  Z_{n,i}]\,.
\]
It follows from the construction of $P_n$ that for every $(n,i) \in \mathcal I$ we have $l(S_{n,i}) \geq l(\partial P_1) 4^{-n}$.
Therefore, $r_{n,i} \geq l(\partial P_1) 4^{-n-1}$.
Plugging this into the previous display gives
\[
\mathbb E[Y_{n,i}] \geq
	2\delta^2 4^{-n-1}l(\partial P_1) \mathbb P(\gamma_{n,i} > 10\delta^2r_{n,i})\,.
\]
Finally, we will set $C_1 \geq C_2$ so we can apply Lemma~\ref{lm :: lower bound probability of inclusion} and get
\[
\mathbb E[Y_{n,i}] \geq
	2 c_2\delta^2 4^{-n-1}l(\partial P_1) \,.
\]
Summing over $i$ gives
\[
\mathbb E \left[ X_{n+1} - X_{n}\right]
	\geq  2^{-1}c_2\delta^2 l(\partial P_1)\,.
\]
Since $\Gamma$ is an odd function of $h$, we have $\mathbb E[X_{1}] = -l(\partial P_1)$.
Therefore
\begin{align*}
\mathbb E[X_{n^*}] = \mathbb E[X_1] + \sum_{n = 1}^{n^*-1} \mathbb E\left[X_{n+1} - X_{n}\right]
	\geq ( c_2\delta^2 (n^* - 1)/2 - 1)l(\partial P_1)\,.
\end{align*}
Plugging in $n^* = \lfloor \log_{16}(N) \rfloor$ and $\delta = 10^{-2}(\epsilon \mathtt m)^{2/3}$, we see  that $\mathbb E[X_{n^*}] > 0$
 (which is a rewrite of \eqref{eq :: upper bound final target}) for $N \geq e^{C_1 \epsilon^{-4/3}}$ provided that $C_1 \geq C_2$ is a large enough constant (depending on $\mathtt m$).
\end{proof}

\section{Upper bound on greedy lattice animal}\label{sec :: Lower bound proof}
This section is devoted to the proof of Proposition~\ref{prop :: correlation length lower bound}, which is essentially \cite[Theorem 4.4.2]{Talagrand14}. We make the connection between \cite[Theorem 4.4.2]{Talagrand14} and Proposition~\ref{prop :: correlation length lower bound} slightly more explicit and we claim no credit for material in this section. 

 For a Gaussian process $X$ indexed on a set $T$, define the canonical metric $d_X:T\times T \to [0,\infty)$ for $(T, X)$ by
\begin{equation}\label{eq :: canonical metric def}
	d_X(s,t) = \mathbb E[(X(s) - X(t))^2]^{1/2}\,.
\end{equation} 
Next, we review the $\gamma_{\alpha,\beta}$-functionals which measure the size of a metric space in a way that can be used to control the maximum of a Gaussian process. We begin with an auxiliary definition.
\begin{defn}\label{def :: admissible sequence}
	Given a set $T$, an admissible sequence on $T$ is an increasing sequence of partitions $(\Pi_n)_{n \geq 0}$ of $T$ such that $|\Pi_0|= 1$ and $|\Pi_n| \leq 2^{2^n}$ for $n \geq 1$.
\end{defn}
For a partition $\Pi_n$ of a set $T$ and an element $t \in T$, we will denote by $\pi_n(t)$ the element of $\Pi_n$ that contains $t$. Now we are ready to define the $\gamma_{\alpha,\beta}$ functionals.
\begin{defn}\label{def :: gamma2}
	Given a set $T$, a metric $d$ on $T$, and numbers $\alpha,\beta > 0$, let 
	\[
	\gamma_{\alpha,\beta}(T,d) = \inf_{(\Pi_n)} \sup_{t \in T} \left[\sum_{n \geq 0} \left(2^{n/\alpha} \diam(\pi_n(t),d)\right)^\beta \right]^{1/\beta}\,,
	\]
	where the infimum is taken over all admissible sequences and $\diam(\pi_n(t),d)$ is the $d$-diameter for $\pi_n(t)$. In addition, define  $\gamma_2(T, d) = \gamma_{2, 1}(T, d)$.
\end{defn}
With this definition in place, we can state Talagrand's majorizing measure theorem (see \cite{Talagrand87} and \cite[Theorem 2.2.22]{Talagrand14}) which gives a tight bound on the expectation of the supremum of a Gaussian process in terms of the $\gamma_2$-functional.  Write $\|X\|_T =\sup_{t\in T} X_t$.
\begin{theorem}\label{thm :: majorizing measure}
	There exists a universal constant $K$ such that the following holds. If $T$ is a set and $X$ is a centered Gaussian process indexed on $T$, we have
	\[
	\E[\|X\|_T] \leq K \gamma_2(T,d_{X})\,.
	\]
\end{theorem}

Next, we state the Borell–Sudakov-Tsirelson inequality. For a set $T$ and a Gaussian process $(X_t)_{t\in T}$ indexed on $T$, let $\sigma_{X}^2 =\sup_{t \in T} \var(X_t)$.  The Borell–Sudakov-Tsirelson inequality says that the tails of the maximum of $X$ behave roughly like those of a Gaussian random variable with variance $\sigma_{X}^2$ (see
\cite{Borell75, SudakovTsirelson74} or \cite[Theorem 2.1]{Adler90} for a proof):
\begin{equation}\label{eq :: Borel inequality}
	\mathbb P\left(\big|\|X_T\|  - \mathbb E[\|X_T\|] \big| > z \right)
	\leq 2\exp\left(-\frac{z^2}{2\sigma_{X}^2}\right) \mbox{ for all } z >0\,.
\end{equation}
Note that we do not need to assume $X$ is centered for \eqref{eq :: Borel inequality}. Note also that for any lattice animal $A$ we have $\var(\sum_{v \in A} h_v) = |A| \leq |\partial A|^2$, so if $T$ is a set of lattice animals and $X_t$ is the sum of the Gaussian variables in the lattice animal $t$  normalized by its boundary size, we have  $\sigma_{X}^2 \leq 1$.

Having introduced these tools, we turn to the proof of Proposition~\ref{prop :: correlation length lower bound}. It is more convenient to work with the unnormalized lattice animal processes, so we will partition the lattice animals by the lengths of their boundaries. The following lemma is the key to the proof of  Proposition~\ref{prop :: correlation length lower bound}.
\begin{lemma}\label{lm :: bound on animals of bounded length}
	For a vertex $v \in \mathbb Z^2$ and an integer $k \geq 2$, let $\mathfrak A_{v,k}$ be the collection of simply connected lattice animals $A$ such that $|\partial A| \leq 2^k$, $v \in A$, and there exists $u \sim v$ such that $u$ is not in $A$. For $A \in \mathfrak A_{v,k}$, let $Y_A = \sum_{v \in A} h_v$. Then there exists a constant $C > 0$ such that
	\[
	\P \left( \max_{A \in \mathfrak A_{v,k}} Y_A \geq C k^{3/4} 2^k + u 2^k\right) \leq 2 e^{-u^2/2} \quad \mbox{ for all } u > 0\,.
	\]
\end{lemma}
In Lemma~\ref{lm :: bound on animals of bounded length}, we restricted to $A$ containing $v$ \emph{on its boundary} so that we have $|\mathfrak A_{k,v}| \leq 2^{2^{k+1}}$ (as explained in the proof of Lemma~\ref{lm :: bound on animals of bounded length} below).
Lemma~\ref{lm :: bound on animals of bounded length} can be deduced as a consequence of the following two lemmas in \cite{Talagrand14}.
 \begin{lemma}\cite[Lemma 4.4.6]{Talagrand14}\label{lem-4.4.6-talagrand}
 	Let $n \geq 1$ and $T$ be a set such that $|T| \leq 2^{2^n}$. Let $d$ be a metric on $T$. Then ($\sqrt{d}$ is also a metric and)
 	\[
 	\gamma_{2}(T, \sqrt{d}) \leq n^{3/4} \gamma_{1,2}(T,d)^{1/2}\,.
 	\]
 \end{lemma}
 \begin{lemma}\cite[Proposition 4.4.5]{Talagrand14}\label{lem-4.4.5-talagrand}
	There exists a constant $C > 0$ such that
	\[
	\gamma_{1,2}(\mathfrak A_{v,k}, d_Y^2) \leq C 2^{2k}\,.
	\]
\end{lemma}
Note that \cite[Proposition 4.4.5]{Talagrand14} was stated in a slightly different context but the metric space it applies to is easily seen to be isomorphic to $\mathfrak A_{v,k}$ with distance $d_Y^2$ since $d_Y^2(A, A')  = \E[(Y_A - Y_{A'})^2]$ is simply the cardinality of the symmetric difference of $A$ and $A'$.  

\begin{proof}[Proof of Lemma~\ref{lm :: bound on animals of bounded length}]
In order to apply Lemma~\ref{lem-4.4.6-talagrand}, we need a bound on the cardinality of $\mathfrak A_{k,v}$. By considering a simply connected lattice animal $A$ as the lattice points enclosed by a loop consisting of $|\partial A|$ edges of the dual lattice $(1/2,1/2)  + \mathbb Z^2$, it is easy to see that $|\mathfrak A_{k,v}| \leq 2^{2^{k+1}}$ (this is because one can construct a loop by starting an edge near $v$ and adding new edges sequentially where each new edge has at most  4 choices). At this point, it is immediate from Lemmas~\ref{lem-4.4.6-talagrand} and \ref{lem-4.4.5-talagrand} that $\gamma_2(\mathfrak A_{k,v},d_Y) \leq C k^{3/4}2^k$. Thus  by Theorem~\ref{thm :: majorizing measure}, we have that $\E[\max_{A \in \mathfrak A_{v,k}} Y_A] \leq C k^{3/4} 2^k$.
Therefore, we can obtain Lemma~\ref{lm :: bound on animals of bounded length} by \eqref{eq :: Borel inequality} and the fact that $\var(Y_A) \leq 2^{2k}$ for all $A \in \mathfrak A_{v,k}$.
\end{proof}

\begin{proof}[Proof of Proposition~\ref{prop :: correlation length lower bound}]
	The proof is the same as the proof of \cite[Theorem 4.4.2]{Talagrand14} using \cite[Proposition 4.4.3]{Talagrand14}.  Note that the total length of all edges in $\Lambda_N$ is $2 \cdot (2N+1)\cdot 2N$ and let $k^* = \min\{k \,:\, 2^k \geq 2\cdot (2N+1) \cdot 2N  \}$. We have $k^* \leq C \log N$, and for any $A \in \mathfrak A_N$ there exists $2 \leq k \leq k^*$ such that $2^{k-1} \leq |\partial A| \leq 2^k$. Therefore, using Lemma~\ref{lm :: bound on animals of bounded length} and a union bound over $v$ and $k$ we have for some constant $C > 0$
	\[
	\P\left( \max_{A \in \mathfrak A_N} \frac{Y_A}{|\partial A|} \geq C (\log N)^{3/4} + x \right) \leq C e^{\log N-x^2/2}\,.
	\] 
	Letting $x = C'(\log N)^{3/4}+ u$ for a large enough constant $C'$ concludes the proof.
\end{proof}

\section{Alternative Proof of the Upper bound on greedy lattice animal}\label{sec :: Lower bound proof original}

This section presents an alternative proof of the lower bound in Theorem~\ref{thm :: main result}. 
As in the previous proof, it consists of an upper bound on the greedy lattice animal process. 
The proof is more geometric and for that reason it is easier to do the analysis in the continuum. 
For notational convenience, we use a curve to refer to an oriented piecewise linear curve (unless otherwise specified), and we call a maximal line segment in a curve (in terms of inclusion) a side of the curve and the endpoints of such segments the vertices of the curve. 
For $R > 1$, let $\mathcal P_R $denote the collection of positively oriented simple closed curves with sides of length at least 1 that are contained in $[-R, R]^2$. 
For $\eta  \in \mathcal P_R$, we let $P(\eta)$ denote the polygon enclosed by
$\eta$ (i.e. the set of points $\eta$ separates from infinity) and $\nu(\eta) = W(P(\eta))$ (where as in Section~\ref{sec :: emergence} $W$ is a white noise).  We will prove the following
\begin{proposition}
	There exists a constant $c_1 > 0$ such that $0  < \epsilon < c_1$ and $R \leq \exp(c_1 \epsilon^{4/3}/\log \epsilon
	^{-1})$ we have the following
	\[
	\P\left(\sup_{\eta \in \mathcal \epsilon P_R} \nu(\eta)) - \ell(\eta) \geq 0\right) \leq \exp(-c_1/\epsilon^2)
	\]
\end{proposition}
Our proof consists of a multi-scale analysis where we bound the value of polygon animals of growing diameters in an inductive manner.
Crucially, this requires a hierarchical structure on the polygon animal process (i.e, the Gaussian process associated with polygon animals), where ideally we can decompose a polygon animal into a sum of polygon animals in smaller scales.
Such decomposition is more or less obvious for convex polygon animals, but in general it seems non-trivial to have a schematic decomposition.
In order to address this, in Subsection~\ref{sec :: winding numbers} we will instead consider a decomposition for non-closed curves (as opposed to closed curves which are boundaries of polygon animals) for which we rely on the geometric notion of winding number in order to keep track of ``signs'' in the decomposition.

In Subsection~\ref{sec :: Gaussian process} we review a number of useful tools from the theory of Gaussian processes, including the Gaussian concentration inequality and Dudley's integral bound on the supremum.
To prepare for applications of Dudley's integral bound that will arise in our analysis, we formulate Lemma~\ref{lm :: generic bound for partitioned collection} which we will repeatedly apply later in order to bound the supremum of a Gaussian process using bounds on the suprema of its sub-components.
Having introduced these tools, we show in Subsection~\ref{sec :: reduction to first moment} that Proposition~\ref{prop :: correlation length lower bound} follows from a first moment bound, as incorporated in Proposition~\ref{prop :: lower bound reduction}.

The rest of the section is then devoted to the proof of Proposition~\ref{prop :: lower bound reduction}.
As a preliminary analysis, we prove in Subsection~\ref{sec:regularity} a regularity bound by an application of Dudley's integral bound, and in particular we control the supremum over a collection of polygon animals (in fact, curves) which are all perturbations of a certain animal.
Besides its application in the multi-scale analysis, this regularity bound plays a key role in verifying the base case for our inductive argument, as detailed in Subsection~\ref{sec :: lower base case}.
The core induction argument is then carried out in Subsections~\ref{sec :: lower induction} and \ref{sec :: lower deferred}, where we use the full power of the tools we have reviewed and developed thus far.
In this proof, we also see the main advantage of working with non-centered Gaussian variables (as opposed to working with centered Gaussian variables as given by polygon animals normalized by their
 boundary sizes): when decomposing a curve into sub-curves, the entropy grows in the number of ``significant turns'' the curve makes, but the mean decreases with the number of turns since significant turns lead to increments of the boundary length compared to a straight line (see Claim~\ref{clm :: skeleton length bound}).
Our decompositions of curves (and correspondingly our inductive proofs) are implemented in two similar steps, where we first decompose the curves in terms of their vertical oscillations and then in terms of their horizontal oscillations.
The first step is more important and requires more delicate analysis, which is carried out in Subsection~\ref{sec :: lower induction} with some lemmas deferred to the next subsection.
The second step, formulated as Lemma~\ref{lm :: strip bound}, is proved at the end of the paper.
On the one hand it employs a similar analysis with some simplifications since the bounds required are not as sharp as in the first step.
On the other hand, it includes a new ingredient (Claim~\ref{clm :: lower bound induction hypothesis}) which finally allows us to invoke the induction hypothesis.

\subsection{A multi-scale representation for polygon animals via winding numbers}\label{sec :: winding numbers}

The first challenge for a multi-scale analysis proof is a multi-scale representation for the polygon animal process.
That is, we wish to decompose a polygon animal in a big scale into polygon animals in a small scale.
However, it is not obvious how to carry this out, since we are unable to decompose a closed curve into a union of closed curves.
Therefore, we prefer to work with non-closed curves which can be decomposed into a concatenation of shorter curves.
To this end, we extend the definition of $\nu$ to non-closed curves, in a natural way where we add to a non-closed curve a line segment that goes from its end point to its start point.
Since our polygons can be arbitrary and in particular not necessarily convex, when decomposing a polygon by decomposing its boundary curve, the aforementioned extension for a (simple) piece of boundary may result in a non-simple closed curve (see the right picture of Figure~\ref{fig :: coarsening}) and thus we will also have to take into account the signs for regions in smaller polygons in the decomposition.
In order to do this, we will use the notion of winding number (for curves that are not necessarily simple), via which we obtain a decomposition of polygon animals as stated in Corollary~\ref{cor :: coarsening}.
In order to bound the variance of the Gaussian random variables associated with the polygon animals in the decomposition, we will bound corresponding winding numbers, as incorporated in Lemmas~\ref{lm :: splitting curve w bound} and \ref{lm :: good curve w bound}.

We now briefly review the notion of winding number, which counts the number of revolutions (in the counter-clockwise direction) a closed curve $\eta$ completes around a point.
For a curve $\eta : [0,1] \to \mathbb R^2$, we let $\eta^* = \eta([0,1])$ denote the points in $\eta$.
Note that if $\eta$ is a closed curve and $z \in \mathbb R^2 \setminus \eta^*$, then there exists a continuous parametrization of $\eta - z$ in polar coordinates.
If $(r,\theta):[0,1] \to \mathbb R^2$ is such a parametrization, then the winding number of $\eta$ around $z$, which we denote by $w(z,\eta)$, is given by (note that the winding number does not depend on the choice of the parametrization)
\[
w(z,\eta) = \frac{\theta(1) - \theta(0)}{2\pi}\,.
\]
For any closed curve $\eta$,  let $A = \mathbb R^2 \setminus \eta^*$.
Then $w(\cdot, \eta)$ is an integer-valued function on $A$ that is constant on each connected component of $A$ and is zero on the unbounded component of $A$.
See Figure~\ref{fig :: winding number def} for an illustration and \cite[Theorem 10.10]{Rudin87} for a proof.
\begin{figure}[h]
	\centering
	\includegraphics[scale = 0.35]{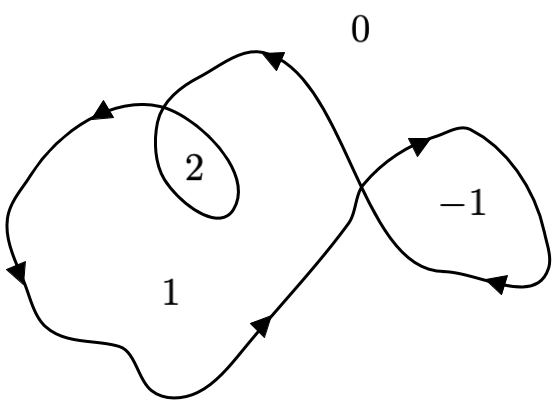}
	\caption{An illustration of the winding number of a curve on each component of its complement}
	\label{fig :: winding number def}
\end{figure}
For a closed curve $\eta$, we let $\mathfrak C(\eta)$ denote the collection of connected components of $\mathbb R^2 \setminus \eta^*$.
Since $w$ is constant on every element of $\mathfrak C(\eta)$, for $\mathcal C \in \mathfrak C(\eta)$ we abuse the notation and write $w(\mathcal C,\eta)$ for the value of $w$ on $\mathcal C$.
Finally, we let
\begin{equation}\label{eq-def-nu}
\nu(\eta) = \sum_{\mathcal C \in \mathfrak C(\eta)} w(\mathcal C,\eta)W(\mathcal C)\,.
\end{equation}
Note that this is consistent with $\nu(\eta) = W(P(\eta))$ for $\eta \in \mathcal P_R$.
We can think of $\nu(\eta)$ as the integral of $w(\cdot,\eta)$ with respect to $d W$.

Having defined $\nu$ for arbitrary closed curves, it remains to extend the definition to non-closed curves.
For a non-closed curve $\eta$, we let $a$ and $b$ be the start and end points of $\eta$ and as mentioned earlier we extend $\eta$ to a closed curve $\eta_{\circ}$ by concatenating $\eta$ with the line segment from $b$ to $a$.
We let $w(\cdot,\eta) = w(\cdot,\eta_{\circ})$ and $\nu(\eta) = \nu(\eta_{\circ})$.
So, for example, if $\eta$ is a straight line from $a$ to $b$, we have $\nu(\eta) = 0$.

The following notation will be useful later in this section.
For a (not necessarily closed) curve $\eta$, a point $z \notin \eta^*$, and any continuous parametrization $(r,\theta): [0,1] \to \mathbb R \times \mathbb R$ of $\eta - z$, we let
\[
\Delta(z,\eta) = \frac{\theta(1) - \theta(0)}{2\pi}\,.
\]
Note that $\Delta$ does not depend on the choice of the parametrization and it coincides with the winding number if and only if $\eta$ is closed.

For a curve $\eta$, we let $\eta^-$ denote the curve obtained by reversing the orientation of $\eta$.
Since $w(\cdot, \eta^-) = -w(\cdot,\eta)$, we have $\nu(\eta^-) = - \nu(\eta)$.
For two curves $\eta_1$ and $\eta_2$ such that the endpoint of $\eta_1$ coincides with the start point of $\eta_2$, we let $\eta_1 \eta_2$
denote their concatenation.
Since we want to decompose long curves into short ones, we need to relate $\nu(\eta_1 \eta_2)$ to $\nu(\eta_1)$ and $\nu(\eta_2)$.
To this end, we first prove the following lemma which allows us to calculate the change in $\nu(\eta)$ that results from changing a segment of $\eta$.
See the left picture of Figure~\ref{fig :: coarsening} for an illustration.
\begin{lemma}\label{lm :: changing arc}
For $a,b,u,v \in \mathbb R^2$, let $\eta_1$ be a curve from $a$ to $u$, let $\eta_2, \eta_3$ be curves from $u$ to $v$, and let $\eta_4$ be a curve from $v$ to $b$.
Let $\eta = \eta_1\eta_2 \eta_4$ and $\gamma = \eta_1 \eta_3 \eta_4$, and let $\ell$ be the line segment from $b$ to $a$.
Then for all $z\in \mathbb R^2 \setminus (\eta^* \cup \gamma^* \cup \ell^*)$,
\[
w(z,\eta) - w(z,\gamma) = w(z, \eta_2\eta_3^-)\,.
\]
In particular, we have $\nu(\eta) - \nu(\gamma) = \nu(\eta_2 \eta_3^-)$.
\end{lemma}
\begin{figure}	
	\centering
	\begin{subfigure}{.5\textwidth}
		\centering
		\includegraphics[width=.6\linewidth]{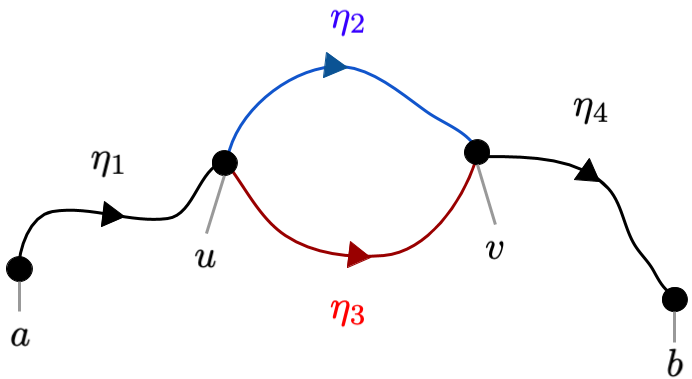}
	\end{subfigure}%
	\begin{subfigure}{.5\textwidth}
		\centering
		\includegraphics[width=.6\linewidth]{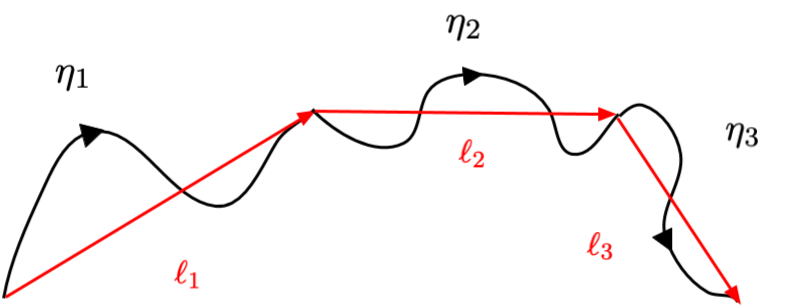}
	\end{subfigure}
\caption{The left picture is on changing a segment of a curve and the right picture is on partitioning a curve into segments.}
\label{fig :: coarsening}
\end{figure}
\begin{proof}
Let $\Delta_i = \Delta(z,\eta_i)$ for $1 \leq i \leq 4$.
We have
\begin{align*}
\Delta(z,\eta) = \Delta_1 + \Delta_2 + \Delta_3\,, \mbox{ and } \Delta(z,\gamma)
	= \Delta_1 + \Delta_3 + \Delta_4\,.
\end{align*}
Recall that $\ell$ is the line segment from $b$ to $a$.
We have
\[
w(z,\eta) - \Delta(z,\eta) = w(z,\gamma) - \Delta(z, \gamma) = \Delta(z,\ell)\,.
\]
Therefore
\begin{equation*}
w(z,\eta) - w(z,\gamma) = \Delta(z,\eta) - \Delta(z,\gamma) =
	\Delta_2 - \Delta_3 = w(z,\eta_2\eta_3^-)\,.
\end{equation*}
Recalling \eqref{eq-def-nu},
it follows that $\nu(\eta) - \nu(\gamma) = \nu(\eta_2 \eta_3^-)$ as claimed.
\end{proof}

The following corollary of Lemma~\ref{lm :: changing arc} will allow us to split
$\eta$ into segments which we can analyze separately. See
Figure~\ref{fig :: coarsening} for an illustration.
\begin{cor}\label{cor :: coarsening}
For $n \geq 1$, let $(\eta_1,\dots,\eta_n)$ be a partition
	of a curve $\eta$ (i.e., $\eta = \eta_1\dots\eta_n$). Let
	$\mathbf x = (x_1,\dots,x_{n+1})$ be such that $\eta_i$ is a curve from $x_i$
	to $x_{i+1}$. For $1 \leq i \leq n$, let $\ell_i$ be the line segment from $x_i$
	to $x_{i+1}$ and let $\eta' = \ell_1\dots\ell_n$ be the curve obtained by concatenating
	all these line segments.
	Let $\ell_{n+1}$ be the line segment from $x_{n+1}$ to $x_1$.
	Then for all $z\in \mathbb R^2 \setminus (\eta^* \cup (\eta')^* \cup  \ell_{n+1}^*)$,
	\[
	w(z,\eta) = w(z,\eta') + \sum_{i = 1}^n w(z,\eta_i)\,.
	\]
	In particular, we have $\nu(\eta) = \nu(\eta') + \sum_{i = 1}^n \nu(\eta_i)$.
\end{cor}
\begin{proof}
Let $\gamma_0 = \eta$, $\gamma_n = \eta'$, and for $1 \leq i < n$ let
$\gamma_i = \ell_1 \dots \ell_i \eta_{i+1} \dots \eta_n$. We have
\[
w(z,\eta) - w(z,\eta') = \sum_{i = 1}^{n} (w(z,\gamma_{i-1}) - w(z,\gamma_{i}))\,.
\]
Therefore, it suffices to show
\[
w(z,\gamma_{i-1}) - w(z,\gamma_{i}) = w(z,\eta_i) \quad  1 \leq i \leq n\,.
\]
This follows by applying Lemma~\ref{lm :: changing arc} to $\gamma_i$ and $\gamma_{i-1}$
with $a  = x_0$, $b = x_n$, $u = x_i$, and $v = x_{i+1}$.
\end{proof}
\begin{figure}
	\centering
	\begin{subfigure}{.32\textwidth}
		\centering
		\includegraphics[width=.6\linewidth]{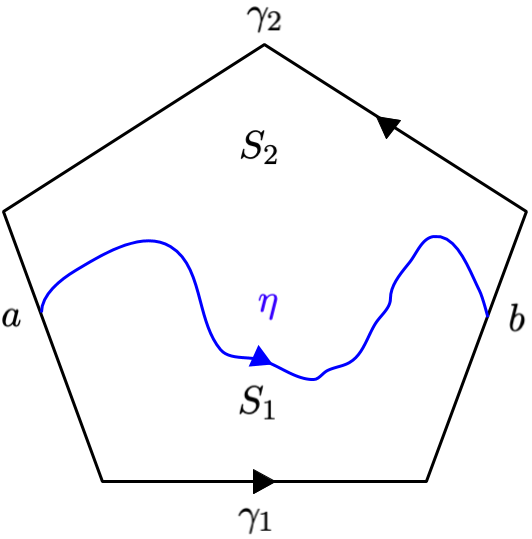}
	\end{subfigure}%
	\begin{subfigure}{.32\textwidth}
		\centering
		\includegraphics[width=.6\linewidth]{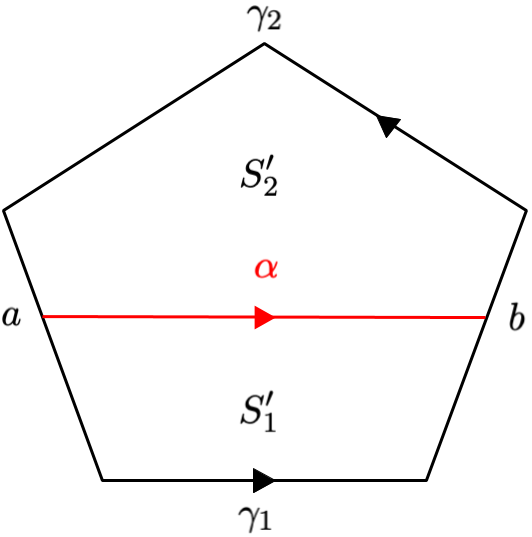}
	\end{subfigure}
	\begin{subfigure}{.32\textwidth}
	\centering
	\includegraphics[width=.6\linewidth]{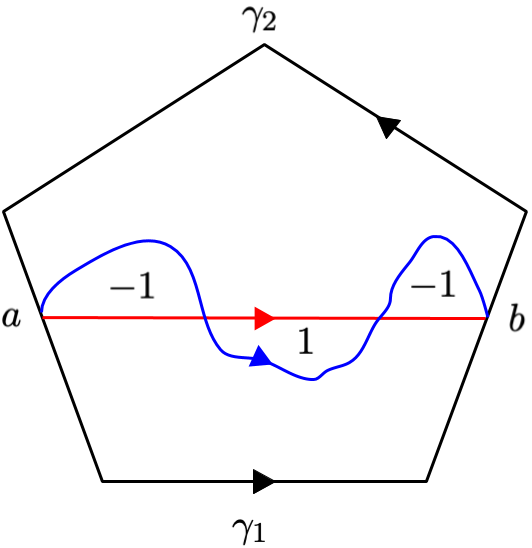}
	\end{subfigure}
	\caption{Illustration of Lemma~\ref{lm :: splitting curve w bound}. In all pictures, $S$ is the whole polygon.}
	\label{fig :: splitting curve}
\end{figure}
As hinted at the beginning of this subsection, in order to control the maximum of $\nu$ over a set of simple non-closed curves
$\mathcal H$, we need to control $\var(\nu(\eta))$ for
$\eta \in \mathcal H$. We cannot simply bound $\var(\nu(\eta))$ by the area enclosed
when we extend $\eta$ into a closed curve because the extended curve need not be simple (and therefore its winding numbers need not be bounded by 1).
Therefore, we will introduce some geometric conditions on curves that, when satisfied,
yield desirable bounds on winding numbers.
\begin{defn}\label{def :: splitting curve}
	We say a curve $\eta \in \mathcal H$ from $a$ to $b$ is a \emph{splitting curve} if there exists an
	open, bounded convex set $S$ such that $a,b \in \partial S$ and
	$\eta^* \setminus \{a,b\} \subset S$. In addition, we say $\eta$ splits $S$.
	See Figure~\ref{fig :: splitting curve} for an illustration of a splitting curve.
\end{defn}
\begin{lemma}\label{lm :: splitting curve w bound}
For a splitting curve $\eta$, we have $\max_{z \notin \eta^*} |w(z,\eta)| \leq 1$.
\end{lemma}
\begin{proof}
Let $S$ be such that $\eta$ splits $S$. First, note that $w(z, \eta) = 0$ for all
$z \notin S$, so we only need to show $|w(z, \eta)| \leq 1$ for $z \in S$. Let $\gamma$
denote the (not necessarily piecewise linear) curve obtained by giving $\partial S$
the positive orientation. Let $\gamma_1$ be the segment of $\gamma$ from $a$ to
$b$ and $\gamma_2$ be the segment of $\gamma$ from $b$ to $a$. Let
$\eta_1 = \gamma_1\eta^- $ and $\eta_2 = \eta \gamma_2 $. Note that $S \setminus \eta^*$
is the disjoint union of two simply connected sets $S_1$ and $S_2$ such that $\eta_1$
and $\eta_2$ are obtained by giving $\partial S_1$ and $\partial S_2$ the positive
orientation (see Figure~\ref{fig :: splitting curve} for an illustration). Thus,
\begin{align*}
\Delta(z,\gamma_1) + \Delta(z, \gamma_2) = w(z, \gamma) &= 1 \quad \mbox{ for } z \in S\,,\\
\Delta(z,\gamma_1) - \Delta(z,\eta) = w(z,\eta_1) &= 1 \quad \mbox{ for } z \in S_1\,, \\
\Delta(z,\eta) + \Delta(z,\gamma_2) = w(z,\eta_2) &= 1 \quad \mbox{ for } z \in S_2\,.
\end{align*}
Therefore,
\[
\Delta(z, \eta) = \begin{cases}
-\Delta(z,\gamma_2) &z \in S_1,\\
\Delta(z,\gamma_1) & z \in S_2.
\end{cases}
\]
That is, for every $z \in S_1$ (i.e., ``to the right'' of $\eta$), the angular
displacement is the same for $\eta$ and $\gamma_2^-$ and similarly for $z \in S_2$.
Next, let $\alpha$ be the line segment from $a$ to $b$ and note that since $S$ is convex
$\alpha$ splits $S$. Therefore, $S \setminus \alpha$ is the disjoint union of two convex
sets $S_1'$ and $S_2'$ analogous to $S_1$ and $S_2$ (see the middle picture of Figure~\ref{fig :: splitting curve}).
By the same argument as above, we get
\[
\Delta(z,\alpha) =  \begin{cases}
-\Delta(z,\gamma_2) &z \in S'_1,\\
\Delta(z,\gamma_1) & z \in S'_2.
\end{cases}
\]
Taking differences we obtain (see the right picture of Figure~\ref{fig :: splitting curve})
\begin{equation}\label{eq :: splitting curve w}
w(z,\eta) = \Delta(z,\eta) - \Delta(z,\alpha) = \begin{cases}
	-1 &z \in S_1 \cap S'_2,\\
	0 &z \in (S_1 \cap S'_1) \cup (S_2 \cap S'_2),\\
	1 &z \in S_2 \cap S'_1.
\end{cases}
\end{equation}
This concludes the proof.
\end{proof}
As we will show in Section~\ref{sec :: reduction to first moment}, in principle we only
need to control $\eta$ over the collection of splitting curves. However, our analysis
will proceed by partitioning curves into segments. This raises a difficulty because
it is not obvious how to partition a splitting curve $\eta$ into segments such that each
segment is also a splitting curve. Therefore, we need a weaker assumption on $\eta$ that
is preserved when partitioning $\eta$ into segments in some reasonable way.
\begin{figure}
	\centering
	\begin{subfigure}{.3\textwidth}
		\centering
		\includegraphics[width=.9\linewidth]{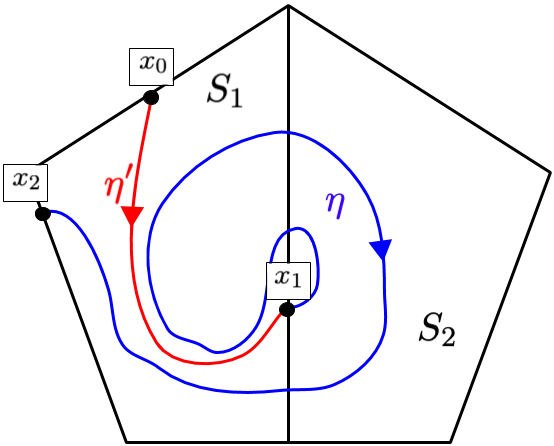}
	\end{subfigure}%
	\begin{subfigure}{.3\textwidth}
		\centering
		\includegraphics[width=.9\linewidth]{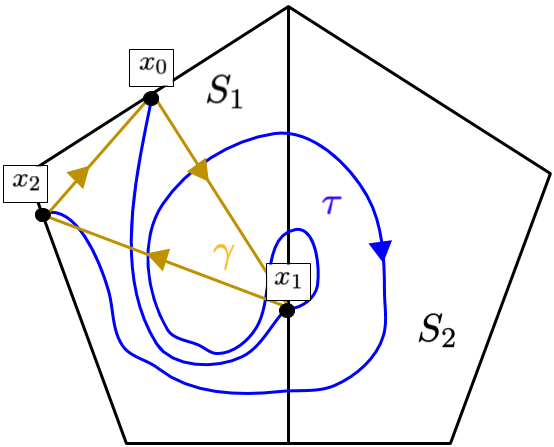}
	\end{subfigure}
	\begin{subfigure}{.3\textwidth}
		\centering
		\includegraphics[width=.9\linewidth]{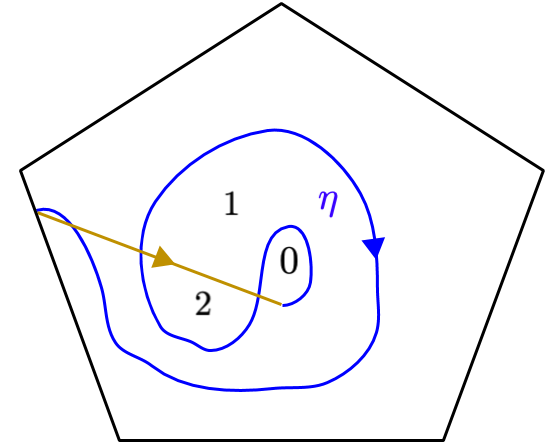}
	\end{subfigure}
	\caption{A good curve has winding numbers bounded by 3. In the pictures, $S_1$ is the polygon on the left and $S_2$ is the whole polygon.}
	\label{fig :: good curve}
\end{figure}
\begin{defn}\label{def :: good curve}
	Let $\eta$ be a simple curve from $a$ to $b$. We say $\eta$ is a good curve if there
	exist open, bounded convex sets $S_1$ and $S_2$ and a splitting curve $\eta'$ such that
	$\eta'$ splits $S_1$ and $\eta'\eta$ splits $S_2$. We call $(\eta', S_1, S_2)$
	a witness for $\eta$. See Figure~\ref{fig :: good curve} for an illustration of a good curve and a witness.
\end{defn}
\begin{lemma}\label{lm :: good curve w bound}
Let $\eta$ be a good curve. Then
$\max_{z \notin \eta^*} |w(z,\eta)| \leq 3$.
\end{lemma}
\begin{proof}
Let $(\eta', S_1, S_2)$ be a witness for $\eta$ and $(x_0,x_1,x_2)$ be such that
$x_0$ and $x_1$ are the start and end points of $\eta'$ and $x_1$ and $x_2$ are
the start and end points of $\eta$. Let $x_3 = x_0$ and for $1 \leq i \leq 3$
let $\ell_i$ be the line segment from $x_{i-1}$ to $x_i$. Finally, let
$\gamma = \ell_1\ell_2 \ell_3$ and $\tau = \eta'\eta$ (see the middle picture of Figure~\ref{fig :: good curve}
for an illustration). By Corollary~\ref{cor :: coarsening},
\[
w(z, \tau) = w(z, \gamma) + w(z, \eta') + w(z, \eta)\,.
\]
Note that $\tau$ splits $S_2$, $\eta'$ splits $S_1$, and $\gamma$ is a triangle. Thus, an application of Lemma~\ref{lm :: splitting curve w bound} yields that
$\max\{|w(z,\eta)|, |w(z,\gamma)|, |w(z,\eta')|\}\leq 1$.
This concludes the proof.
\end{proof}
\begin{figure}
	\centering
	\includegraphics[width=\linewidth]{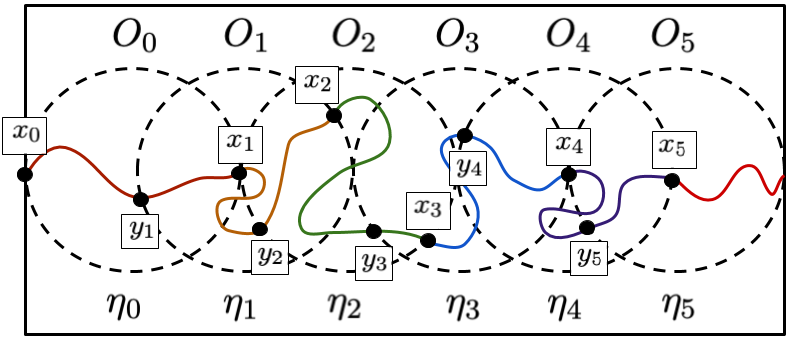}
	\caption{A good curve can be partitioned into good segments}
	\label{fig :: good curve decomposition}
\end{figure}
As mentioned, the advantage of working with good curves over splitting curves is
that it is possible to partition a good curve into segments such that each segment
is also a good curve. This is shown in the next lemma (see
Figure~\ref{fig :: good curve decomposition} for an illustration).
\begin{lemma}\label{lm :: splitting good curves}
	For $n \geq 1$ and a good curve $\eta$, let $(\eta_1,\dots,\eta_n)$ be a partition of $\eta$
	and let $(x_1,\dots,x_{n+1})$ be such that $x_i$ and $x_{i+1}$ are the start and
	end points of $\eta_i$ for $1\leq i\leq n$. Let $(O_1,\dots,O_n)$ be a sequence of
	bounded and convex open sets such that for $1 \leq i \leq n$ we have $x_i \in O_i$,
	$x_{i+1}$ is the first exit of $\eta_i$ from $O_i$, and $x_{i-1} \notin O_i$ if $i > 1$. Then $(\eta_1,\dots,\eta_n)$
	are good curves.
\end{lemma}
\begin{proof}
	Let $(\eta_0, O_0,S)$ be a witness for $\eta$. Since by our assumption $\eta_0\eta$ is
	contained in $S$ (except for its start and end points), we can assume
	$O_i \subset S$ for all $i$ (otherwise we can replace $O_i$ with $O_i \cap S$).
	Let $x_0$ be the start point of $\eta_0$ and note that $x_0 \notin O_1$. For
	$1 \leq i \leq n$, let $y_i \in \partial O_i$ be the last entrance point of $\eta_{i-1}$
	into $O_{i}$. Such a point exists because the start point of $\eta_{i-1}$
	(namely $x_{i-1}$) is not contained in $O_{i}$ (note that it is possible $y_{i} = x_{i-1}$). Let $\gamma_i$ be the segment of
	$\eta_{i-1}$ from $y_i$ to $x_i$. Since $\gamma_i$ is a segment of $\eta_{i-1}$,
	it is contained in $O_{i-1}$ (except for its end point). By construction, $\gamma_i$
	is also contained in $O_i$ (except for its start point). Therefore, $\gamma_i$
	splits $O_{i-1} \cap O_i$. Additionally, $\gamma_i \eta_i$ splits $O_i$. Therefore,
	$\eta_i$ is a good curve.
\end{proof}

\subsection{Further Gaussian process tools}\label{sec :: Gaussian process}
In this subsection we introduce one more from the theory of (sub)Gaussian processes.
In addition, we formulate Lemma~\ref{lm :: generic bound for partitioned collection},
which allows us to bound the supremum of an inhomogeneous Gaussian process by aggregating
bounds on the suprema of homogeneous components in a (suitably chosen) decomposition
of the whole process; this will be repeatedly applied in our multi-scale analysis later.

For a set $T$ and a Gaussian process $(X_t)_{t\in T}$ indexed on $T$, recall that $d_X$ denotes the cannonical metric on $T$ defined by $X$. For $q > 0$, we let $\mathcal N_{T,X}(q)$ be the smallest number of closed $d_X$
balls of radius $q$ that cover $T$. We say $T$ is totally bounded (with respect
to $X$) if $\mathcal N_{T,X}(q)$ is finite for all $q > 0$. Note that this implies
$T$ is separable with respect to $d_X$. As shown in \cite{Dudley67}
(see also \cite[Corollary 4.14]{Adler90}), there exists a universal constant $K$
such that for a centered Gaussian process $X$
\begin{equation}\label{eq :: entropy integral generic bound}
\mathbb{E}\left[\|X\|_S\right] \leq K \int_0^\infty \sqrt{\log \mathcal N_{X,S}(q)}dq\,.
\end{equation}

The following consequence of \eqref{eq :: Borel inequality} will be useful in the proof (c.f. \cite[Lemma 1.5]{Rigollet15}).
Let $Y = \|X\|_S - \mathbb E[\|X\|_S]$, then the following holds for $\sigma^2 = 8 \sigma_{X,S}^2$
\begin{equation}\label{eq :: max is subGaussian}
\mathbb{E} \left[e^{\theta Y}\right] \leq e^{\frac{\sigma^2 \theta^2}{2}} \mbox{ for all } \theta\in \mathbb R\,.
\end{equation}
We say a random variable $Y$ is sub-Gaussian with variance proxy $\sigma^2$ if $\mathbb E[Y] = 0$ and \eqref{eq :: max is subGaussian} holds.
Thus, $\|X\|_S - \mathbb E[\|X\|_S]$ is sub-Gaussian with variance proxy $8 \sigma_{X,S}^2$.
If $Y_1,\dots,Y_n$ are sub-Gaussian random variables with variance proxy $\sigma^2$,
then it is straightforward that (c.f. \cite[Theorem 1.14]{Rigollet15})
\begin{equation}\label{eq :: subGaussian max}
\mathbb E \left[ \max_{1 \leq i \leq n} Y_i \right] \leq \sigma\sqrt{2\log(n)}\,.
\end{equation}

As we will show in Subsection~\ref{sec :: reduction to first moment}, the key to prove Proposition~\ref{prop :: correlation length lower bound} is to bound the expected supremum of the polygon animal process.
This requires us to control all polygon animals including the ones with very complicated boundaries which are unlikely to be maximizers (e.g., a polygon animal that is fractal and thus has small volume to boundary ratio).
Despite the apparently low probability that they are maximizers, these complicated polygon animals raise a challenge for a rigorous bound since their entropy is very large
(if we group polygons by complexity, collections with higher complexity will have higher entropy).
A natural way to deal with this is to partition polygon animals into components depending on the level of complexity, and then to bound the supremum over each component, and finally to aggregate the bounds together.
The following lemma is formulated in order to accomplish this aggregation step (in various settings of our upcoming multi-scale analysis).
An important feature for the setting of the lemma is that the Gaussian variables that come from complicated objects have larger variances but smaller expectations.

\begin{lemma}\label{lm :: generic bound for partitioned collection}
Let $Y$ be a Gaussian process indexed on a set $\mathcal G$, and let $(\mathcal G_a)_{a \in A}$ be a partition
of $\mathcal G$ indexed by a set $A$. Suppose that $(A_n)_{n \geq 0}$ is a partition of $A$ such that
$|A_n| < \infty$ for all $n$. Let
\[
\mu_n := \max_{a \in A_n}\mathbb E \left[\|Y\|_{\mathcal G_a}\right] \mbox{ and }
	\sigma_n^2 := \sup_{a \in A_n, \, \eta \in \mathcal G_a} \var[Y(\eta)]\,.
\]
For $\mu\in \mathbb R$  and $\alpha,\beta,\gamma > 0$, suppose that the following holds for all $n$:
\begin{align*}
\mu_n \leq \mu - 2\alpha n\,, \quad \sigma_n^2 \leq \beta (n+1) \,\mbox{ and }\,
\sigma_n \sqrt{\log(|A_n|)} \leq \frac{\gamma + \alpha n}{4}\,.
\end{align*}
Then
\[
\mathbb E\left[\|Y\|_{\mathcal G}\right] \leq
	\mu + \gamma + \sqrt{2\pi \beta}  +
		4 \frac{\beta}{\alpha} \frac{e^{-\alpha^2/4\beta}}{1 - e^{-\alpha^2/4\beta}}\,.
\]
\end{lemma}
\begin{proof}
The proof  consists of routine applications of the Gaussian process tools introduced earlier. To begin, we let
\[
Z_a := \|Y\|_{\mathcal G_a} - \mathbb E\left[\|Y\|_{\mathcal G_a}\right]\,.
\]
By \eqref{eq :: max is subGaussian}, we have that $Z_a$ is sub-Gaussian with
variance proxy $8 \sigma_n^2$ for $a \in A_n$. Therefore, by \eqref{eq :: subGaussian max},
\begin{equation}
\mathbb E \left[ \max_{a \in A_n} \|Y\|_{\mathcal G_a}\right]
	\leq \mu_n + \mathbb E \left[\max_{a \in A_n}Z_a \right]
	\leq \mu_n + 4\sigma_n\sqrt{\log(|A_n|)}
	\leq \mu + \gamma - \alpha n\,.  \label{eq :: expectation on order n elements}
\end{equation}
It remains to take a maximum over $n$. Write
$M_n = \max_{a \in A_n} \|Y\|_{\mathcal G_a}$.
By \eqref{eq :: expectation on order n elements} and \eqref{eq :: Borel inequality},
for $t\geq 0$ and $n \geq 0$ we have
\begin{align*}
\mathbb P(M_n \geq \mu + \gamma + t) \leq
	2 \exp\left(-\frac{(t + \alpha n)^2}{2\sigma_n^2} \right)\leq
	2 \exp\left(-\frac{(t + \alpha n)^2}{2\beta (n+1)}\right)\,.
\end{align*}
Since $\|Y\|_{\mathcal G} = \max_{n \geq 0} M_n$, we get from a union bound that
\begin{align*}
\mathbb P(\|Y\|_{\mathcal G} \geq \mu + \gamma + t)
	\leq \sum_{n = 0}^\infty\mathbb P(M_n \geq \mu + \gamma + t) 	\leq 2 \sum_{n = 0}^\infty \exp\left(-\frac{(t + \alpha n)^2}{2\beta (n+1)}\right)\,.
\end{align*}
Integrating the preceding bound on the tail probability of $(\|Y\|_{\mathcal G} - \mu - \gamma)$, we obtain that
\begin{align*}
\mathbb E [\|Y\|_{\mathcal G} ] - \mu - \gamma
	\leq 2 \sum_{n = 0}^\infty \int_{0}^\infty \exp\left(-\frac{(t + \alpha n)^2}{2\beta (n+1)}\right) dt
	= 2 \sum_{n = 0}^\infty \sqrt{\beta (n+1)} \int_{\alpha n/\sqrt{\beta(n+1)}}^\infty e^{-u^2/2} du\,.
\end{align*}
Using the simple facts that
$
\int_x^\infty e^{-u^2/2}du \leq \frac{1}{x} e^{-x^2/2}$ for $x \geq 0$
and
$
\int_0^\infty e^{-u^2/2}du = \sqrt{\frac{\pi}{2}}$,
we get that
\begin{align*}
\mathbb E [\|Y\|_{\mathcal G} ] - \mu - \gamma
	&\leq \sqrt{2\pi \beta} +
		2\sqrt{2\beta}\sum_{n = 1}^\infty \sqrt{n} \int_{\alpha\sqrt{n/2\beta}}^\infty e^{-u^2/2}du \\
	&\leq \sqrt{2\pi \beta} + \frac{4\beta}{\alpha} \sum_{n = 1}^\infty e^{-\alpha^2 n/4\beta} = \sqrt{2\pi \beta} + \frac{4\beta}{\alpha} \frac{e^{-\alpha^2/4\beta}}{1 - e^{-\alpha^2/4\beta}}\,. \qedhere
\end{align*}
\end{proof}

\subsection{Reduction to first moment analysis}\label{sec :: reduction to first moment}
In this subsection we prove that Proposition~\ref{prop :: correlation length lower bound}
follows from a bound on the first moment of the supremum of $\eta$ over a
collection of non-closed curves, which is chosen as follows. For
$x,y \in \mathbb R^2$, we let $\mathcal{H}(x,y,\epsilon)$ denote the set of good
curves $\eta$ that satisfy the following conditions:
\begin{itemize}
	\item $\eta$ is contained in the open $19 |x-y| \times 18|x-y|$ rectangle centered around the line segment from $x$ to $y$;
	\item All sides of $\eta$, except possibly the first and last one, are of length
	at least 1;
	\item Let $u$ and $v$ be the start and end points of $\eta$, then
	$|u - x| \leq \epsilon^4|x-y|$and $|v - y| \leq \epsilon^4|x-y|$.
\end{itemize}
(In the above, the choices for constants $19, 18, 4$ are flexible as long as they are reasonably large.)
We let
\begin{equation}\label{eq-def-X-eta-epsilon}
X(\eta,\epsilon) = \epsilon \nu(\eta) - l(\eta)\,,
\end{equation}
and
\[
\mathfrak X(x,y,\epsilon) = \sup_{\eta \in \mathcal{H}(x,y, \epsilon)} X(\eta,\epsilon)\,.
\]
For notational convenience, we will suppress the dependence of $\mathcal H$, $\mathfrak X$, and $X$
on $\epsilon$. In addition, we will write $\mathcal H(R)$ and $\mathfrak X(R)$ for
$\mathcal H(o,(R,0))$ and $\mathfrak X(o,(R,0))$, respectively. Note that (by translation
and rotation invariance of the white noise) $\mathfrak X(x,y)$ has the same distribution as
$\mathfrak X(|x-y|)$. The following monotonicity property of $\mathfrak X$ will be useful in later analysis.
\begin{claim}\label{clm :: f monotonicity}
$\mathfrak X(R)/R$ is stochastically dominated by $\mathfrak X(S)/S$ for $R\leq S$.
\end{claim}
\begin{proof}
	Using the simple fact that in two dimensions the area of a region grows quadratically with the scaling, we observe that
 $\{\nu(\eta)/t: \eta\in \mathcal H\}$ has the same distribution as $\{\nu(\eta/t): \eta \in \mathcal H\}$ for $t>0$ and any collection of curves $\mathcal H$.
	Since $l(\eta)/t = l(\eta/t)$, it follows that $\mathfrak X(t)/t$ has the same distribution as the supremum of $X(\eta)$ over $t^{-1} \mathcal H(t)$.
	It is clear from the definition that $t^{-1}\mathcal H(t)$ is increasing in $t$ so the conclusion follows.
\end{proof}

We are now ready to state the major ingredient for the proof of
Proposition~\ref{prop :: correlation length lower bound}.
\begin{proposition}\label{prop :: lower bound reduction}
There exists a small constant $c_3 > 0$ such that for all $0 < \epsilon < c_3$ and
$R \leq 16 \exp(c_3\epsilon^{-4/3}/\log(\epsilon^{-1}))$ we have
$\mathbb E (\mathfrak X(R)/R) \leq -\frac{1}{2}$.
\end{proposition}
We next prove Proposition~\ref{prop :: correlation length lower bound}
assuming Proposition~\ref{prop :: lower bound reduction}. Note that if
Proposition~\ref{prop :: lower bound reduction} holds for some constant $c_3$, it also
holds if we decrease the value of $c_3$, and thus we may assume $c_3 \leq 10^{-10}$.
For the rest of this subsection, we fix $0 < \epsilon < c_3$ and $R>0$ such that
\begin{equation}\label{eq :: reduction R bound}
R \leq \exp\left(\frac{c_3}{\epsilon^{4/3}\log(\epsilon^{-1})}\right)\,.
\end{equation}

\noindent \textbf{Step 1: a partition of $\mathcal P_R$.}
Recall that $\mathcal P_R$ is the collection of oriented simple closed curves with sides of length at least 1 that are contained in $[-R,R]^2$.
For a curve $\eta$, we define its diameter by
\[
\diam(\eta) = \max \{|x - y| \,:\, x,y \in \eta^*\}\,.
\]
For $0 \leq k < \infty$, we let $\mathcal P_{R,k}$ be
the collection of curves $\eta \in \mathcal P_{R}$ with $\diam(\eta) \in [2^k, 2^{k+1})$.
Note that since the diameter of $[-R,R]^2$ is $2^{3/2}R$, we have $\mathcal P_{R,k} = \emptyset$ for all
$k > K = \lfloor \log_2(2^{3/2}R)\rfloor$.
We let $E_k$ be the event that there exists $\eta \in \mathcal P_{R,k}$ such that
$X(\eta)> 0$. We have
\begin{equation}\label{eq :: polygon sup union bound}
\mathbb P(X(\eta) > 0 \mbox{ for some } \eta \in \mathcal P_R)
	\leq \sum_{k = 0}^K\mathbb{P}(E_k)\,.
\end{equation}

\noindent \textbf{Step 2: from closed curves to non-closed curves.}  To bound
the probability of $E_k$, we will show that if $E_k$ occurs then there exist
$x,y \in [-2R,2R]^2 \cap ( \epsilon^4 2^{k} \cdot \mathbb Z^2)$ with $|x-y| \leq 2^{k+2}$
such that $\mathfrak X(x,y) > 0$. If $E_k$ occurs, there exists $\eta \in \mathcal P_{R,k}$ such that
$X(\eta) > 0$. Let $u,v \in \eta^*$ be such that $|u-v| = \diam(\eta)$. Let $\eta_1$ be
the segment of $\eta$ from $u$ to $v$ and $\eta_2$ be the segment of $\eta$ from $v$ to $u$.
Note that $\eta = \eta_1 \eta_2$ and that by Corollary~\ref{cor :: coarsening},
$\nu(\eta) = \nu(\eta_1) + \nu(\eta_2)$. Therefore, $X(\eta) = X(\eta_1) + X(\eta_2)$
so we have (possibly after relabeling $u$ and $v$) that $X(\eta_1) > 0$.
Note that if we let $S$ be the box obtained by translating and rotating $[0,|u-v|]\times[-|u-v|,|u-v|]$ so that it is centered around the line segment from $u$ to $v$, then $\eta_1$ splits $S$.
(See Figure~\ref{fig :: closed to open} for an illustration.)
Therefore, $\eta_1$ is a good curve.
\begin{figure}
	\centering
	\includegraphics[scale = 0.25]{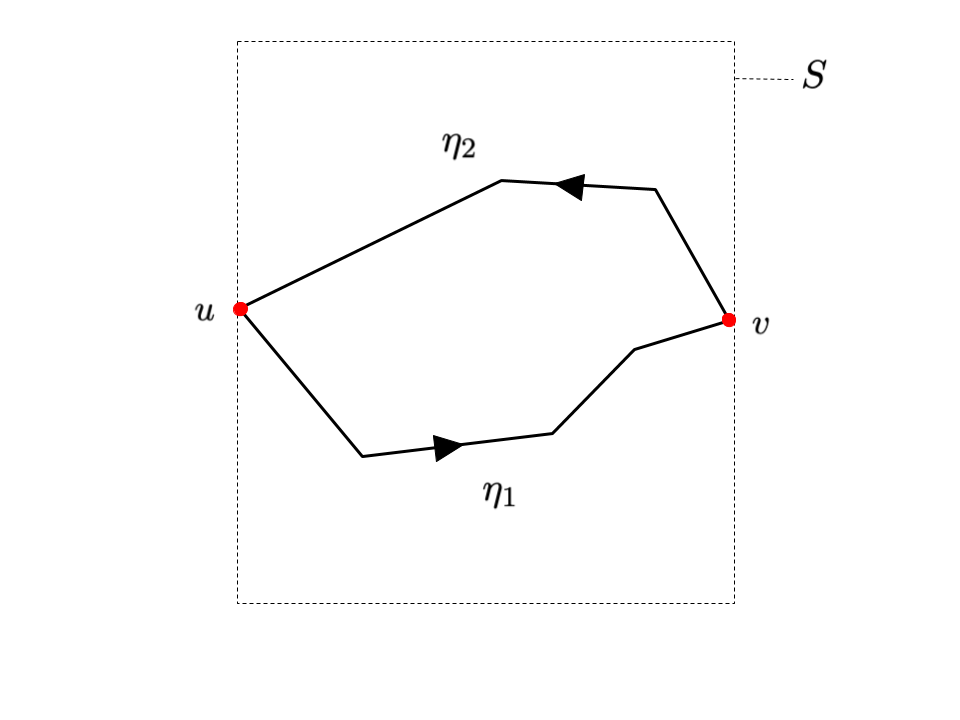}
	\caption{A closed curve partitioned into two splitting curves}
	\label{fig :: closed to open}
\end{figure}
Let $x,y \in \epsilon^4 2^k \cdot \mathbb Z^2$ be such that $u$ is in the axis-aligned
square of side length $\epsilon^4 2^k$  centered at $x$ and $v$ is in the corresponding
square centered at $y$.
Since by assumption $|u-v| =\diam(\eta) \in [2^{k}, 2^{k+1})$, we have $\eta_1 \in \mathcal H(x,y)$ and therefore $\mathfrak X(x,y) > 0$.
In addition, since we can assume $\epsilon < 1/2$ we have $|x-y| \leq 2^{k+2}$.
This concludes the verification of our claim at this step.

\noindent \textbf{Step 3: applying Proposition~\ref{prop :: lower bound reduction}.}
In order to conclude the proof, we need to bound the number of possible pairs of
$(x,y)$ that arise in {\bf Step 2}. Let $N_k = |(\epsilon^4 2^k \cdot \mathbb Z^2) \cap [-2R,2R]^2 |$.
Then the number of possible pairs of $(x,y)$ is at most $N_k^2$. We have
\[
N_k \leq
	\left(2[2R \epsilon^{-4} 2^{-k}] +1\right)^2 \leq
	20 R^2\epsilon^{-8} 4^{-k}\,.
\]
Therefore, we obtain by Claim~\ref{clm :: f monotonicity} and a union bound
\begin{equation}\label{eq :: Ek union bound}
\mathbb{P}\left(E_k \right) \leq
	400 R^4\epsilon^{-16} 16^{-k}\mathbb P(\mathfrak X(2^{k+2}) > 0)\,.
\end{equation}
Plugging this into \eqref{eq :: polygon sup union bound} and using
Claim~\ref{clm :: f monotonicity} again, we conclude
\begin{equation}\label{eq :: reduction second to last bound}
\mathbb P(X(\eta) > 0 \mbox{ for some } \eta \in \mathcal P_R) \leq
	800 R^4\epsilon^{-16}\mathbb P( \mathfrak X(2^{K+2}) > 0)\,.
\end{equation}
Since $2^{K+2} \leq 16 R$, by \eqref{eq :: reduction R bound} and Proposition~\ref{prop :: lower bound reduction} we have $\mathbb E (\mathfrak X(2^{K+2})/2^{K+2}) \leq -\frac{1}{2}$.
At this point, it is natural to apply Gaussian concentration inequality.
To this end, we need a bound on $\var[\nu(\eta)]$ for $\eta \in \mathcal H(S)$ (here $S$ is any positive number).
Since every curve $\eta \in \mathcal{H}(S)$ is contained in $(-9S,10S)\times (-9S,9S)$ and has winding number bounded by 3 (by Lemma~\ref{lm :: good curve w bound}), we have
\begin{equation}\label{eq :: H variance bound}
\var[\nu(\eta)] < (100 S)^2.
\end{equation}
Combined with the first moment bound, \eqref{eq :: Borel inequality}, and the fact
that $\var[X(\eta)] = \epsilon^2 \var[\nu(\eta)]$ this yields that
 \[
\mathbb P(\mathfrak X(2^{K+2}) > 0 )\leq 2\exp(-10^{-5}\epsilon^{-2})\,.
 \]
Plugging this into \eqref{eq :: reduction second to last bound} gives
\[
\mathbb P(X(\eta) > 0 \mbox{ for some } \eta \in \mathcal P_R) \leq
	1600 R^4\epsilon^{-16} \exp(-10^{-5}\epsilon^{-2}) \leq \exp(-c_3\epsilon^{-2})\,,
\]
where the last step used $\epsilon \leq c_3 \leq 10^{-10}$ and
 \eqref{eq :: reduction R bound}. This concludes the proof of Proposition~\ref{prop :: correlation length lower bound}.

\subsection{A regularity bound for highly correlated polygon animals}\label{sec:regularity}

In this subsection, we prove a bound on the supremum of $\nu$ over a collection of highly correlated polygon animals that can all be seen as perturbations of a single animal (see Lemma~\ref{lm :: max over simple collections}); this bound will be repeatedly applied in our multi-scale analysis (for instance, it will be the main ingredient in the proof of the base case of Proposition~\ref{prop :: lower bound reduction}).
We first introduce some notation.
For a sequence of points $\mathbf v = (v_0,\dots,v_n)$, let $\eta = \eta(\mathbf v)$ be the curve obtained by concatenating the line segments between the neighboring points in the sequence.
That is, $\eta = \ell_1\dots\ell_n$ where $\ell_i$ is the line segment from $v_{i-1}$ to $v_i$ for $1 \leq i \leq n$.
For convenience of exposition, we say that $\mathbf v$ describes $\eta$. Also recall that  $d_\nu(\eta^{i},\eta^{i+1}) =
	\mathbb E \left[(\nu(\eta^{i}) - \nu(\eta^{i+1}))^2\right]^{1/2}$.

\begin{lemma}\label{lm :: max over simple collections}
For $n \geq 1$, let $\mathcal G$ be a collection of sequences of $(n+1)$
points in $\mathbb R^2$. Let $\Delta_2 \geq \Delta_1 > 0, \Delta_3>0$ and $\sigma > 0$
be such that the following holds for all $\mathbf v, \mathbf w \in \mathcal G$:
\begin{itemize}
	\item $|v_i - v_{i-1}| \leq \Delta_1$ for $1 \leq i \leq n$;
	\item $|v_{n} - w_{0}| \leq \Delta_2$;
    \item $|v_i - w_i| \leq \Delta_3$ for $1\leq i\leq n$ and $|v_0 - w_0| \leq \min(\Delta_1,\Delta_3)$;
	\item $d_\nu(\eta(\mathbf v), \eta(\mathbf w)) \leq \sigma$.
\end{itemize}
Let
\begin{align*}
\Delta_4 =
	\min\big(\Delta_1 + \tfrac{1}{2\Delta_1} \big(\tfrac{\sigma}{n+1}\big)^2, \Delta_3\big) +
		\tfrac{1}{\Delta_1} \big(\tfrac{\sigma}{n+1}\big)^2 \mbox{ and }
\Delta_5 = \sqrt{2\pi \Delta_1^{(n-1)/(n+1)} \Delta_4} \Delta_2^{1/(n+1)}\,.
\end{align*}
Then the following holds for an absolute constant $C_4>0$
\[
\mathbb E \big[\sup_{\mathbf v \in \mathcal G} \nu(\eta(\mathbf v)) \big] \leq
	C_4 \sigma\sqrt{(n+1)\log\big(\tfrac{\Delta_5 (n+1)}{\sigma}\big)}.
\]
\end{lemma}
\begin{remark}
We have formulated Lemma~\ref{lm :: max over simple collections} in a slightly cumbersome way in order
to make it flexible enough that we can apply it in various settings throughout the proof.
\end{remark}

Our proof of Lemma~\ref{lm :: max over simple collections} is based on an application of the Dudley integral bound.
To this end, a major ingredient is a bound on covering numbers with respect to the canonical distance.
Thus, we first provide a bound on the canonical distance between (Gaussian variables associated with)
two curves, each described by a sequence of points.
\begin{claim}\label{clm :: canonical distance bound}
For  $n \geq 1$, let $\eta$ and $\eta'$ be curves described by $\mathbf v = (v_0,\dots,v_n)$ and
$\mathbf w = (w_0,\dots,w_n)$, respectively.
Write $w_{-1} = v_n$ and $v_{n+1} = w_0$. Then,
\[
d_\nu(\eta,\eta') \leq
	\sum_{i = 0}^{n} \sqrt{2^{-1}(|w_i - w_{i-1}| + |v_{i+1} - v_i|) |v_i - w_i|}\,.
\]
\end{claim}
\begin{proof}
Our proof is based on the idea of interpolation.
For $0 \leq i \leq n+1$, define the vector $\mathbf z^i$ by
\[
z_{j}^{i} = \begin{cases}
w_j & 0\leq j < i,\\
v_j &i \leq j \leq n.
\end{cases}
\]
Let $\eta^i$ be the curve described by $\mathbf z^{i}$. We have
$\eta^0 = \eta$, $\eta^{n+1} = \eta'$, and by the triangle inequality
\[
d_\nu(\eta, \eta') \leq \sum_{i = 0}^{n} d_\nu(\eta^{i},\eta^{i+1})\,,
\]
where we recall that  $d_\nu(\eta^{i},\eta^{i+1}) =
	\mathbb E \left[(\nu(\eta^{i}) - \nu(\eta^{i+1}))^2\right]^{1/2}$.
Let $\xi_i$ be the curve described by $(w_{i-1},w_i,v_{i+1},v_i, w_{i-1})$. By
Corollary~\ref{cor :: coarsening} we have $\nu(\eta^{i+1}) - \nu(\eta^i) = \nu(\xi_i)$.
Next, we let $\tau_{i,1}$ be the triangle described by $(w_{i-1}, w_i, v_i, w_{i-1})$ and
$\tau_{i,2}$ be the triangle described by $(v_i, w_i, v_{i+1}, v_i)$ and note that
$\nu(\xi_i) = \nu(\tau_{i,1}) + \nu(\tau_{2,i})$. See
Figure~\ref{fig :: change from one point} for an illustration.
\begin{figure}
	\centering
	\begin{subfigure}{.32\textwidth}
		\centering
		\includegraphics[width=.6\linewidth]{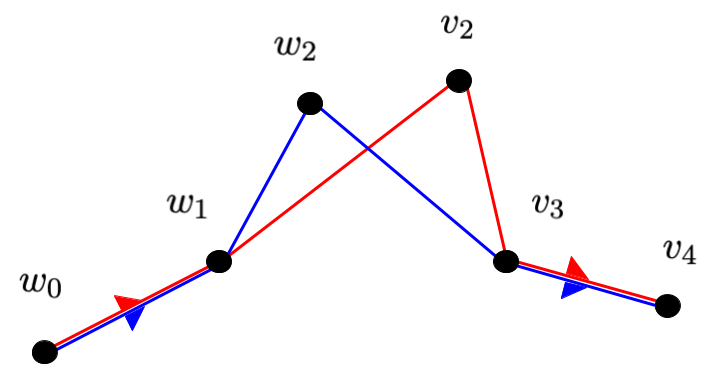}
	\end{subfigure}%
	\begin{subfigure}{.32\textwidth}
		\centering
		\includegraphics[width=.6\linewidth]{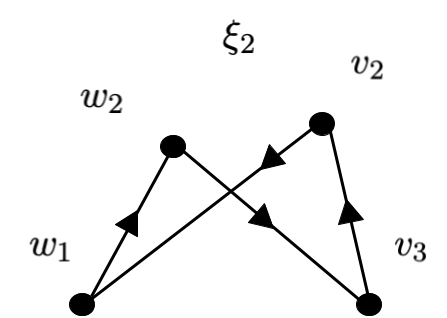}
	\end{subfigure}
	\begin{subfigure}{.32\textwidth}
		\centering
		\includegraphics[width=.6\linewidth]{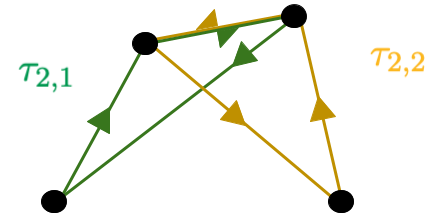}
	\end{subfigure}
	\caption{The difference between two curves that differ by one vertex is given by two triangles}
	\label{fig :: change from one point}
\end{figure}
Since $\tau_{1,i}$ and $\tau_{2,i}$ share a side (namely the segment between $w_i$
and $v_i$) but with opposite orientation, we have
$
\mathbb E[\nu(\tau_{1,i})\nu(\tau_{2,i}))] \leq 0$.
Altogether, this implies that
\[
\mathbb E[\nu(\xi_i)^2] \leq \mathbb E[\nu(\tau_{1,i})^2] + \mathbb E[\nu(\tau_{2,i})^2] \leq
	2^{-1}(|w_i - w_{i-1}| + |v_{i+1} - v_i|) \cdot |v_i - w_i|\,,
\]
where the second inequality used the fact that a triangle with two sides of length
$s_1$ and $s_2$ has area at most $s_1s_2/2$. Taking a square root and summing over
$i$ concludes the proof.
\end{proof}

\begin{proof}[Proof of Lemma~\ref{lm :: max over simple collections}]
As mentioned earlier, the key to the proof is a bound on the $s$-covering numbers of (the collection of curves
described by) $\mathcal G$ with respect to $d_\nu$, which we denote by
$\mathcal N_{\nu,\mathcal G}(s)$. Since by assumption
$d_\nu(\eta(\mathbf v),\eta(\mathbf w)) \leq \sigma$ for all sequences in $\mathcal G$, we have
$\mathcal N_{\nu,\mathcal G}(s) = 1$ for $s \geq \sigma$. Consequently, we assume
$s < \sigma$ throughout. Let $\delta(s) = s^2/(n+1)^2$. If
$\mathbf v, \mathbf w \in \mathcal G$ are such that
$|v_i - w_i| \leq \Delta_1^{-1} \delta(s)$ for $0 < i < n$ and
\[
\max(|v_0 - w_0|, |v_n - w_n|) \leq \Delta_2^{-1} \delta(s)\,,
\]
then by Claim~\ref{clm :: canonical distance bound}
\begin{equation}\label{eq :: canonical distance bound}
d_\nu(\eta,\eta') \leq
	\sqrt{\Delta_1}\sum_{i = 1}^{n-1} \sqrt{|v_i - w_i|} +
		\sqrt{\Delta_2}(\sqrt{|v_0 - w_0|} + \sqrt{|v_n - w_n|}) \leq s\,.
\end{equation}
Let $\mathcal T$ be a tiling of $\mathbb R^2$ by squares of side-length
$1$ centered on points in $\mathbb Z^2$. If $x$ and $y$ are in the same tile of
$\mathcal T$, then $|x-y| \leq \sqrt{2}$. Thus, by \eqref{eq :: canonical distance bound}
a sufficient condition for $d_\nu(\eta,\eta') \leq s$ is that  $v_i$ and $w_i$ are in the same tile of
$2^{-1/2}\Delta_1^{-1} \delta(s) \cdot\mathcal T$ for $0 < i < n$ and that $v_0$
and $w_0$ as well as $v_n$ and $w_n$ are in the same tile of
$2^{-1/2}\Delta_2^{-1}\delta(s) \cdot \mathcal T$. Therefore, it suffices to bound
the number of sequences of tiles $\mathbf v$ can occupy (i.e., $v_i$ is contained in the $i$-th tile in the sequence). To this end, we note that the number
of tiles in $a \cdot \mathcal T$ that intersect a given ball of radius $b$ is at most
\[
g(a,b) = \tfrac{\pi(b + \sqrt{2}a)^2}{a^2}\,.
\]
This can be seen by noting that all such tiles are contained in a ball of radius
$b + \sqrt{2}a$. We let
\[
b = \min \big(\Delta_1 + \tfrac{1}{2\Delta_1} \left(\tfrac{\sigma}{n+1}\right)^2, \Delta_3\big)\,,
\]
and $a_j(s) = 2^{-1/2}\Delta_j^{-1} \delta(s)$ for $j = 1,2$.
By assumption, $v_0$ is contained in a ball of radius $\min(\Delta_1,\Delta_3) \leq b$ so
the number of tiles of size $a_2(s)$ which $v_0$ can possibly occupy is at most $g(a_2(s), b)$. Next we wish to determine the number of possible tiles which $v_{i}$ can occupy given the tile $v_{i-1}$ occupies.
Let $T$ be the tile in $a_j(s) \cdot \mathcal T$ that $v_{i-1}$ occupies, where $j = 0$ if $i = 1$ and $j = 1$ otherwise.
Because $|v_{i} - v_{i-1}| \leq \Delta_1$, the tile  $v_i$ occupies must intersect the neighborhood of $T$ of radius $\Delta_1$.
Because $T$ itself has radius $2^{-1/2} a_j(s) = (2\Delta_1)^{-1} \delta(s)$ and $\delta(s) \leq (\frac{\sigma}{n+1})^2$, we conclude that the tile $v_i$ occupies must intersect the ball of radius
$(\Delta_1 + (2\Delta_1)^{-1} (\frac{\sigma}{n+1})^2)$ centered at $T$.
On the other hand, by our assumption, $v_i$ is contained in a ball of radius $\Delta_3$.
Altogether, given the tile $v_{i-1}$ occupies, we see that the number of tiles of size $a_1(s)$ which $v_i$ can possibly occupy is at most
$g(a_{1}(s), b)$. Similarly, given the tile $v_{n-1}$ occupies, we see that the number of tiles  of size $a_2(s)$ which  $v_n$ can possibly occupy is
at most $g(a_2(s), b)$. Using again the fact that $a_2(s) \leq a_1(s)$ and the fact that
$
b + \sqrt{2}a_1(s) \leq b + \sqrt{2}a_1(\sigma) = \Delta_4$,
we have
\begin{align*}
\mathcal N_\nu(\mathcal G, s) \leq g(a_1(s),b)^{n-1} g(a_2(s), b)^2\leq \pi^{n+1} \frac{\Delta_4^{2(n+1)}}{a_1(s)^{2(n-1)}a_2(s)^4}
	&\leq \left(\frac{2\pi \Delta_4^2}{\delta(s)^2}\right)^{n+1} \Delta_1^{2(n-1)}\Delta_2^4\,.
\end{align*}
Recalling that $\delta(s) = s^2/(n+1)^2$, we can further write
\begin{align*}
\mathcal N_\nu(\mathcal G, s)& \leq (2\pi \Delta_1^{2(n-1)/(n+1)}\Delta_2^{4/(n+1)}\Delta_4^2 (\delta(s))^{-2})^{n+1} \\
	&\leq (\sqrt{2\pi \Delta_1^{(n-1)/(n+1)} \Delta_4} \Delta_2^{1/(n+1)}(n+1)s^{-1})^{4(n+1)}  = (\Delta_5(n+1)s^{-1})^{4(n+1)}\,.
\end{align*}
Plugging this into
\eqref{eq :: entropy integral generic bound} gives that for a universal constant $K>0$
\begin{align*}
\mathbb{E}\big[\|\nu\|_{\mathcal G}\big]
\leq  2K \sqrt{n+1} \int_0^\sigma \sqrt{\log\big(\tfrac{\Delta_5(n+1)}{s}\big)}ds
= 2K \Delta_5 (n+1)^{3/2} \int_0^{\sigma/\Delta_5(n+1)} \sqrt{\log(s^{-1})}ds\,.
	\numberthis \label{eq :: supremum ugly bound}
\end{align*}
An elementary calculation shows that for $\delta \in (0,1)$,
\begin{align*}
\int_0^\delta \sqrt{\log(x^{-1})}dx
= 2\int_{\sqrt{\log(\delta^{-1})}}^\infty u^2e^{-u^2}du
\leq \delta \big(\sqrt{\log(\delta^{-1})} +
\tfrac{\sqrt{\pi}}{2}\big)\,. \numberthis \label{eq :: log int}
\end{align*}
Since $\Delta_4 \geq \Delta_1^{-1}\delta(\sigma)$ and $\Delta_2 \geq \Delta_1$, we have
\[
\tfrac{\Delta_5(n+1)}{\sigma} \geq
\tfrac{\sqrt{2\pi\Delta_1 \Delta_4}(n+1)}{\sigma }\geq
\sqrt{2\pi} > 2 \,.
\]
Therefore, $\sigma/\Delta_5(n+1) \leq 1/2$ and thus there exists a constant $c > 0$
such that
\[
\int_0^{\sigma/\Delta_5(n+1)} \sqrt{\log(s^{-1})}ds \leq
	c \tfrac{\sigma}{\Delta_5(n+1)} \sqrt{\log\big(\tfrac{\Delta_5 (n+1)}{\sigma}\big)}\,.
\]
Plugging this into \eqref{eq :: supremum ugly bound} completes the proof of the lemma.
\end{proof}

\subsection{Proof of Proposition~\ref{prop :: lower bound reduction}: base case} \label{sec :: lower base case}
In this subsection, we prove the base case bound for Proposition~\ref{prop :: lower bound reduction}
(see Corollary~\ref{cor :: lower bound base case}).

For our analysis, it will be useful to associate to each curve $\eta$ in $\mathcal H(R)$
a sequence of points $\mathbf v(\eta) = (v_0,\dots,v_m)$ such that $\mathbf v(\eta)$
describes $\eta$ and $|v_{i} - v_{i-1}| \leq 2$ for $1 \leq i \leq m$ and $|v_i - v_{i-1}| \geq 1$ for $2 \leq i \leq m-1$.
We define the sequence $\mathbf v(\eta)$ by the following recursive procedure. We let $v_0$ be the start point of $\eta$.
If the first side of $\eta$ has length at most 2, we let $v_1$ be the end point of this side; otherwise we let $v_1$ be the point on the first side of $\eta$ that is distance 1 away from $v_0$ (or the endpoint of the first side of $\eta$ if it is has length less than 1).
Then, we let $\eta'$ be the segment of $\eta$ from $v_1$ to its end point and let $\mathbf v(\eta)$
be the concatenation of $v_0$ and $\mathbf v(\eta')$.
That is, if a side of $\eta$ has length $l \geq 1$, it is split into $\lceil l -1 \rceil$ segments by $\mathbf v(\eta)$.

We let $\mathcal H(R, m)$ be the set of curves in $\mathcal H(R)$
such that $\mathbf v(\eta)$ has length $m+1$ (i.e. curves that are split into $m$
segments). Note that since each segment described by
$\mathbf v$ has length at most 2 and each curve in $\mathcal H(R)$ has length at
least $(1 - 2\epsilon^4)R$, we see that $\mathcal H(R, m)$ is
empty if $2m < (1-2\epsilon^4)R$. The following is a main step in verifying the base case bound.

\begin{lemma}\label{lm :: base case technical result}
There exists a constant $C_5 > 1$ such that the following holds. For $\epsilon < C_5^{-1}$,
$0 \leq R \leq 2 \epsilon^{-4}$, and $m$ satisfying that $\mathcal H(R, m)$ is non-empty, we have
\[
\mathbb E [ \|\nu\|_{\mathcal H(R, m)}]
	\leq C_5 R\sqrt{m \log(C_5 m/R)}\,.
\]
\end{lemma}

\begin{proof}
Without loss of generality we assume $C_5 > 10$. Let
$\mathcal G = \{ \mathbf v(\eta) \,:\, \eta \in \mathcal H(R,m)\}$ denote the sequences associated with the curves in $\mathcal H(R,m)$.
For all $\mathbf v, \mathbf w \in \mathcal G$, we have $|v_0| \leq R \epsilon^4 \leq 2$, $|v_{i} - v_{i-1}| \leq 2$ for $1 \leq i \leq m$, and $|v_m - w_0| \leq R(1+2\epsilon^4) \leq 2R$.
By \eqref{eq :: H variance bound} we have $\sigma_{\eta, \mathcal H(R)}^2 < (100 R)^2$, which implies
\[
d_\nu(\mathbf v, \mathbf w) \leq 2 \sigma_{\eta,\mathcal H(R)} \leq 200 R\,.
\]
Therefore, we can apply Lemma~\ref{lm :: max over simple collections} with
$\Delta_1 = 2$, $\Delta_2 = 2R$, $\sigma = \Delta_3 = 200 R$. Since $\mathcal H(R,m)$
is non-empty, we have $2m \geq (1-2\epsilon^4)R \geq R/2$ and thus
$\sigma/m \leq 800$. This implies that
\begin{align*}
\Delta_4 =
	\min\big(\Delta_1 + \tfrac{1}{2\Delta_1} \big(\tfrac{\sigma}{m+1}\big)^2, \Delta_3\big) +
		\tfrac{1}{\Delta_1} \big(\tfrac{\sigma}{m+1}\big)^2 \leq C\,.
\end{align*}
Additionally, $R^{1/m} \leq R^{4/R} \leq C$ and thus
\begin{align*}
\Delta_5 = \sqrt{2\pi \Delta_1^{(m-1)/(m+1)} \Delta_4} \Delta_2^{1/(m+1)} \leq C\,.
\end{align*}
Therefore, we conclude from Lemma~\ref{lm :: max over simple collections} that
\begin{equation*}
\mathbb E \big[ \|\nu\|_{\mathcal H(R,m)} \big]
	\leq C_4 \sigma\sqrt{(m+1)\log\big(\tfrac{\Delta_5 (m+1)}{\sigma}\big)}
	\leq C R \sqrt{m \log\big(\tfrac{Cm}{R}\big)}\,. \qedhere
\end{equation*}
\end{proof}

We next state and prove
the base case bound as a corollary.
\begin{cor}\label{cor :: lower bound base case}
There exists a constant $c_6 \in (0,1)$ such that
$\mathbb{E}[\mathfrak X(R)/R] \leq -\frac{3}{4}$ for $\epsilon \leq c_6^2$ and $R \leq c_6 \epsilon^{-1}$.
\end{cor}
\begin{proof}
We assume without loss of generality that $c_6 \leq 10^{-1}$. By
Claim~\ref{clm :: f monotonicity}, $\mathfrak X(R)/R$ is stochastically increasing in $R$,
so we can assume without loss of generality that
$R = c_6 \epsilon^{-1} \geq c_6^{-1}  \geq 10$. Let $\mathfrak X(R,m) = \|X\|_{\mathcal H(R, m)}$
and note that if $\eta \in \mathcal H(R, m)$ then
$l(\eta) \geq \max(m - 2, R(1-2\epsilon^4))$. Since we assume $\epsilon < 10^{-2}$,
we have $(1-2\epsilon^4) \geq 15/16$ and thus
\[
\tfrac{\mathfrak X(R, m)}{R}
	\leq \tfrac{\epsilon\|\nu\|_{\mathcal H(R, m)}}{R} -
	\big(\tfrac{m-2}{R} \vee \tfrac{15}{16}\big)\,.
\]
Applying
Lemma~\ref{lm :: base case technical result} and the assumption that
$\sqrt{R}\epsilon \leq c_6$ we obtain (note that $m \geq R/4$, since otherwise $\mathcal H(R, m)$ is empty)
\begin{align}\label{eq-bound-mu-n-base-case}
\mathbb E \big[ \tfrac{\mathfrak X(R, m)}{R}\big]
	&\leq 2C_5 c_6 \sqrt{\tfrac{m}{R}\log\big(\tfrac{2C_5m}{R}\big)} -
		\big(\tfrac{m-2}{R} \vee \tfrac{15}{16}\big) \leq -\tfrac{7}{8} - \tfrac{m}{30R}\,,
	\end{align}
provided that $c_6$ is chosen small enough.
Next, we bound $\frac{\mathfrak X(R)}{R} = \frac{1}{R}\sup_{m\,:\, \mathcal H(R, m) \neq \emptyset} \mathfrak X(R, m)$ by applying Lemma~\ref{lm :: generic bound for partitioned collection} (this is a relatively simple example of an application of Lemma~\ref{lm :: generic bound for partitioned collection}).
We need to put notations and set parameters in the context of Lemma~\ref{lm :: generic bound for partitioned collection}.
To this end, we let $\mathcal G = \mathcal H(R)$ and $Y_\eta = R^{-1} X(\eta, \epsilon)$ (recall the definition of $X(\eta, \epsilon)$ in \eqref{eq-def-X-eta-epsilon}).
We let $A = \{m \,:\, \mathcal H(R,m) \neq \emptyset\}$, let $k$ be the minimum of $A$, and let $A_n = \{k + n \}$ for $n \geq 0$. Further, we write $\mu = -7/8$ and $\alpha = 1/(60R) = \epsilon/(60 c_6)$, and thus recalling \eqref{eq-bound-mu-n-base-case} we have $\mu_n \leq \mu - 2\alpha n$ for $n \geq 0$.
In addition, by \eqref{eq :: H variance bound}, we have $\sigma_{X, \mathcal H(R)}^2 \leq (100 R \epsilon)^2$.
Therefore $\sigma_n^2 \leq (100 \epsilon)^2$ for all $n$ and thus we let $\beta = (100 \epsilon)^2$.
Finally, since $|A_n| = 1$ for all $n$, we have $\log(|A_n|) = 0$ and thus we let $\gamma = 0$.
We have verified that our choice of parameters satisfy all assumptions in Lemma~\ref{lm :: generic bound for partitioned collection}, and therefore an application of Lemma~\ref{lm :: generic bound for partitioned collection} yields that
\begin{align}
\mathbb E\big[\tfrac{\mathfrak X(R)}{R}\big]
	\leq \mu + \sqrt{2\pi \beta} + \tfrac{4\beta}{\alpha}\tfrac{e^{-\alpha^2/4\beta}}{1 - e^{-\alpha^2/4\beta}}\,. \label{eq-f-R-base-bound-1}
\end{align}
Note that there exists a constant $C > 0$ (which does not depend on $c_6$) such that
$\alpha^2/\beta \geq 1/(C  c_6^2)$, $\beta/\alpha \leq C \epsilon c_6 \leq C  c_6^3$, and
$\sqrt{\beta} \leq C \epsilon \leq C c_6^2$. Therefore, if we take $c_6$ to be a sufficiently small constant we have
\[
\sqrt{2\pi \beta} + \tfrac{4\beta}{\alpha}\tfrac{e^{-\alpha^2/4\beta}}{1 - e^{-\alpha^2/4\beta}} \leq \tfrac{1}{8}\,.
\]
Plugging this into \eqref{eq-f-R-base-bound-1}, we conclude that
\begin{equation*}
\mathbb E\big[ \tfrac{\mathfrak X(R)}{R}\big]  \leq \mu + \tfrac{1}{8} = -\tfrac{3}{4}\,. \qedhere
\end{equation*}
\end{proof}

\subsection{Proof of Proposition~\ref{prop :: lower bound reduction}: inductive step} \label{sec :: lower induction}
In this subsection we carry out the inductive step in the proof of
Proposition~\ref{prop :: lower bound reduction}, which is formulated in the next lemma.
\begin{lemma}\label{lm :: lower bound inductive step}
	There exists a constant
	$C_7 > 1$ such that if
	$\epsilon \in (0, C_7^{-1})$, and
	\[
	\mathbb E\big[\tfrac{\mathfrak X(10^{-1}R)}{10^{-1}R} \big] \leq
		-\tfrac{1}{2}\,,
	\]
	then
	\[
	\mathbb E\big[\tfrac{\mathfrak X(R)}{R}\big] \leq
		\mathbb E \big[\tfrac{\mathfrak X(10^{-1} R)}{10^{-1} R}\big] +
			C_7 \epsilon^{4/3}\log(\epsilon^{-1})\,.
	\]
\end{lemma}
Combined with Corollary~\ref{cor :: lower bound base case},
Lemma~\ref{lm :: lower bound inductive step} implies that if
$\epsilon < \min(C_7^{-1}, c_6)$ and
\[
R \leq \exp(1/(4C_7 \epsilon^{4/3}\log(\epsilon^{-1})))\,,
\]
then $\mathbb E \left[\frac{\mathfrak X(R)}{R}\right] \leq - \frac{1}{2}$
which concludes the proof of Proposition~\ref{prop :: lower bound reduction}.

The rest of the subsection is devoted to the proof of Lemma~\ref{lm :: lower bound inductive step}.
We first provide a short outline of the proof.
As hinted in Section~\ref{sec :: Overview}, in order to strike the optimal balance between the variance of the sum of the Gaussian variables and the increment of the boundary length, the optimal curve connecting $(0, 0)$ and $(R, 0)$ should have oscillations in the vertical direction in the order of $\epsilon^{2/3} R$.
A big chunk of our proof is devoted to making this intuition rigorous.
To this end, we will show that oscillations larger than $\epsilon^{2/3}\log(\epsilon^{-1})R$ are too
``costly'' and as a result $X(R)/R$ is approximately optimized (say within an additive error of
$C\epsilon^{4/3}\log(\epsilon^{-1})$) at a curve in $\mathcal H(R)$ contained in a horizontal strip
of height $4\epsilon^{2/3}\log(\epsilon^{-1})R$.
To do this, we decompose each curve $\eta \in \mathcal H(R)$ into segments contained in horizontal strips of height $4\epsilon^{2/3}\log(\epsilon^{-1}) R$ and then bound the supremum of $X/R$ (recall the definition of $X$ in \eqref{eq-def-X-eta-epsilon}) over curves contained in each such strip.
We achieve this by decomposing such curves into segments to which our induction hypothesis applies---since each such curve $\tau$ is contained in a narrow strip, we have very good control on the variance of $\nu(\tau)$, which is crucial for effective applications of the Gaussian concentration inequality.

Let us elaborate our proof strategy in more detail.
For each $\eta \in \mathcal H(R)$, we will construct a sequence of points $\mathbf v(\eta) = (v_0,\dots,v_{\kappa})$ (where $\kappa = \kappa(\eta)$ depends on $\eta$) which decomposes $\eta$ into segments with appropriate vertical oscillations.
For $1 \leq j \leq \kappa$, we let $\eta_j$ be the segment of $\eta$ from $v_{j-1}$ to $v_{j}$ and $\gamma_j$ be the line segment from $v_{j-1}$ to $v_{j}$.
We note that $\eta = \eta_1 \ldots \eta_\kappa$ and let $\gamma(\eta) = \gamma_1  \dots \gamma_\kappa$.
By Corollary~\ref{cor :: coarsening}, we get that
\begin{equation}\label{eq :: induction decomposition}
X(\eta) = \epsilon \nu(\gamma(\eta)) + \sum_{j = 1}^{\kappa} X(\eta_j)\,.
\end{equation}
We will then partition $\mathcal H(R)$ by grouping curves in terms of their associated sequences $\mathbf v$.
For each subset in the partition, we will bound the supremum of $X(\eta_j)$ (for $1\leq j \leq \kappa(\eta)$)
as $\eta$ ranges through the subset, by applying the induction hypothesis.
For this purpose, it is necessary that $\eta_1,\dots, \eta_\kappa$ are good curves, which is ensured by Lemma~\ref{lm :: splitting good curves}.
We then bound the supremum of $\nu(\gamma(\eta))$ as $\eta$ ranges through a subset in the partition by Lemma~\ref{lm :: max over simple collections}.
Finally, we apply Lemma~\ref{lm :: generic bound for partitioned collection} to aggregate the bounds obtained for each subset in the partition and obtain an upper bound for the supremum over the whole space.

Next, we precisely describe the necessary constructions for our multi-scale analysis, for which the key
task is to define $\mathbf v$.  Let $\rho = \epsilon^{2/3}\log(\epsilon^{-1})$. For $k \in \mathbb Z$,
define $\Pi_k$ to be a horizontal strip and $L_k$ to be a horizontal line by
\begin{equation}\label{eq-def-Pi-L}
\Pi_k = \{(x,y) \,:\, (k-1)\rho R < y < (k+1)\rho R\} \mbox{ and }
L_k = \{(x, k\rho R) \,:\, x \in \mathbb R\}\,.
\end{equation}
We first define a sequence $\mathbf w$, based on which we will define $\mathbf v$.

Let $\eta \in \mathcal H(R)$ be a curve with start point $a = (x_a,y_a)$ and end point $b  = (x_b,y_b)$.
We let $w_0 = a$ and $\eta'_0 = \eta$. For $i\geq 0$, as long as $w_i \neq b$, we have
$w_i \in y_a + L_{k_i}$ for some $k_i$. In this case, we let $w_{i+1}$ be the first point of $\eta'_{i}$ in
$y_a + \partial \Pi_{k_i}$ (or $b$ if no such point exists). We also let $\phi_{i+1}$ be the
segment of $\eta'_i$ from $w_i$ to $w_{i+1}$, and let $\eta'_{i+1}$ be the segment of
$\eta'_i$ from $w_{i+1}$ to $b$. Continuing this procedure until reaching $b$ (i.e. $w_n = b$ for some $n$)
produces a sequence of points $\mathbf w(\eta) = (w_0,\dots, w_{n})$ and a sequence of curves $(\phi_1,\dots,\phi_n)$
connecting these points. In addition, we have $w_0 = a$,
$w_{n} = b$, and for each $0 \leq i < n$ there exists $k_i$ such that
$w_i \in y_a + L_{k_i}$. See Figure~\ref{fig :: Induction V and W} for an illustration.
Let $(\eta',S_1,S_2)$ be a witness for $\eta$ (recall that each curve in $\mathcal H(R)$ is good), and $O_{i} = (y_a + \Pi_{k_{i-1}}) \cap S_2$ for $1 \leq i \leq n$.
Then  $(O_1,\ldots, O_n)$ and $(\phi_1,\ldots,\phi_n)$ satisfy the assumptions
of Lemma~\ref{lm :: splitting good curves} and therefore $(\phi_1, \ldots, \phi_n)$ are good curves.

\begin{figure}
	\centering
	\begin{subfigure}{.5\textwidth}
		\centering
		\includegraphics[width=.8\linewidth]{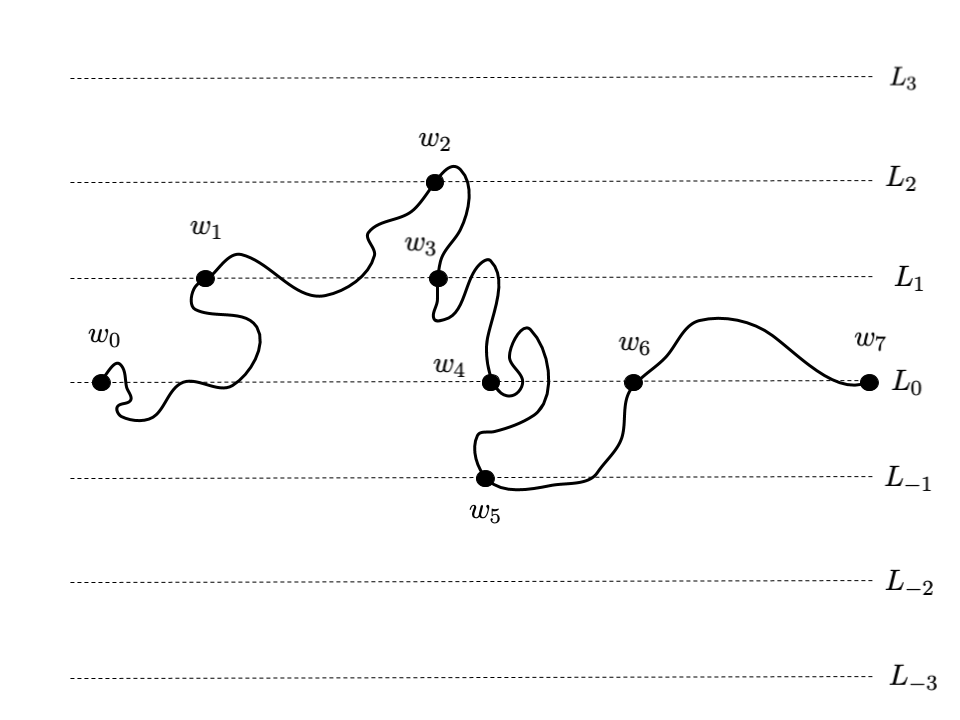}
	\end{subfigure}%
	\begin{subfigure}{.5\textwidth}
		\centering
		\includegraphics[width=.8\linewidth]{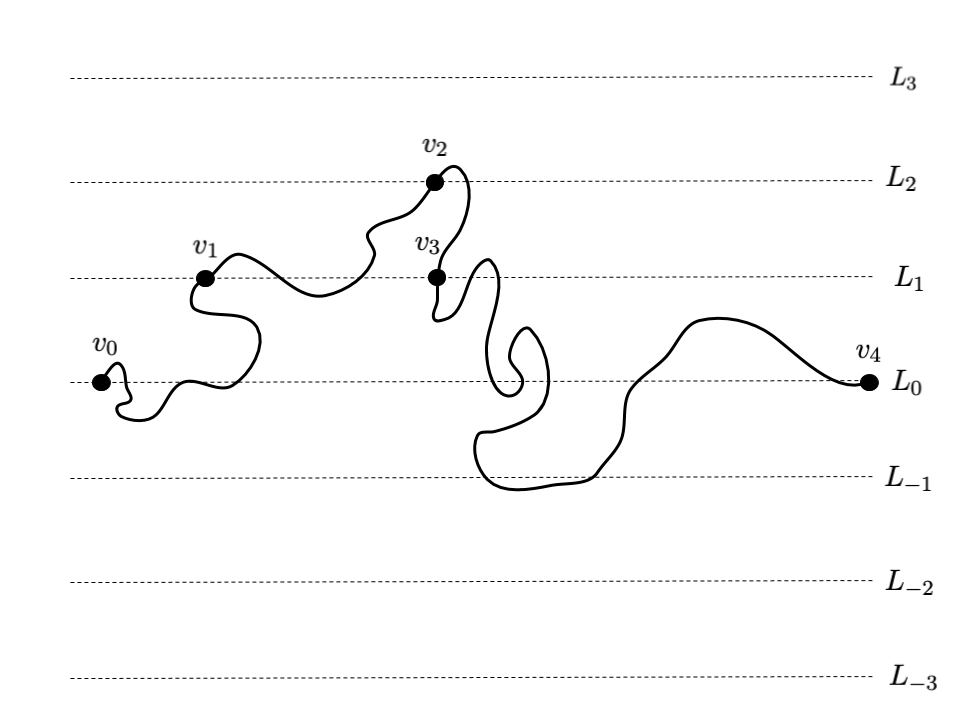}
	\end{subfigure}
	\caption{$\mathbf v(\eta)$ is obtained from $\mathbf w(\eta)$ by removing points. In the illustration $i^* = 2$.}
	\label{fig :: Induction V and W}
\end{figure}

To motivate the definition of $\mathbf v$, we first explain why we do not simply work with $\mathbf w$ in our analysis.
The issue is that if $w_{n-1} \in L_0$, then it is possible that $|w_n - w_{n-1}|$ is much smaller than $\rho R$ while the curve between $w_{n-1}$ and $b$ still has diameter of order $\rho R$.
Since our induction hypothesis applies to curves whose diameter is within a constant factor of the distance between its start and end points, this would complicate our analysis.
A naive solution to this problem is to let $\mathbf v = (w_0,\dots,w_{n-2},w_n)$ (i.e. to remove $w_{n-1}$) if $n > 1$ and $w_{n-1} \in L_0$ (and let $\mathbf v = \mathbf w$ otherwise).
However, the resulting decomposition of $\eta$ does not satisfy the assumptions of Lemma~\ref{lm :: splitting good curves} if $w_{n-3} \in L_0$.
To address this, we let $i^* = \max\{i \,:\, |k_i| = 2\}$ with the convention that $i^* = -1$ if $|k_i| \leq 1$ for $0 \leq i \leq n$, and let $\mathbf v(\eta) = (w_0,\dots,w_{i^*+1},w_n)$.
See Figure~\ref{fig :: Induction V and W} for an illustration.
We let $\kappa = i^*+2$ and for $1 \leq j \leq \kappa$, we let $\eta_j$ be the segment of $\eta$ from $v_{j-1}$ to $v_j$.
By an abuse of notation, for $0 \leq j < \kappa$, we let $k_j$ be such that $v_j \in y_a + L_{k_j}$ and $k_\kappa = 0$. As before, let $(\eta',S_1,S_2)$ be a witness for $\eta$.
For $1 \leq j \leq \kappa -1$, let $O_j = (y_a + \Pi_{k_{j-1}})\cap S_2$.
Finally, let $O_\kappa = (y_a + (\Pi_{-1} \cup \Pi_0 \cup \Pi_1)) \cap S_2$.
Then $(O_1,\ldots, O_\kappa)$ and $(\eta_1,\ldots,\eta_\kappa)$ satisfy the
assumptions of Lemma~\ref{lm :: splitting good curves} and therefore $(\eta_1,\ldots, \eta_\kappa)$ are good curves.
As mentioned at the beginning of the section, for $1 \leq j \leq \kappa$ we let $\gamma_j$
be the line segment from $v_{j-1}$ to $v_j$ and $\gamma(\eta) = \gamma_1 \ldots \gamma_\kappa$.

Having defined $\mathbf v$, we now specify the partition of $\mathcal H(R)$.
Let $s_\kappa = \frac{\epsilon^4\rho}{4\kappa}$ and for each $v \in s_\kappa R \cdot \mathbb Z^2$ let
$T_{\kappa,v}$ be the axis-aligned square of side-length $s_\kappa R$ centered at $v$. Note
that $\{T_{\kappa,v}\,:\, v \in s_\kappa R \cdot \mathbb Z^2\}$ is a tiling of $\mathbb R^2$.
For an integer $\kappa \geq 1$ and a sequence $\mathsf v = (\mathsf v_0,\ldots, \mathsf v_{\kappa})$
of points in $s_\kappa R \cdot \mathbb Z^2$, we let (by a slight abuse of notation)
$\mathcal H(R,\mathsf v)$ be the set of curves $\eta \in \mathcal H(R)$ such that
$\mathbf v(\eta) = (v_0,\ldots,v_{\kappa})$ satisfies $v_j \in T_{\kappa,\mathsf v_j}$ for
$0 \leq j \leq \kappa$. For $\kappa \geq 1$, we let $\mathcal  V_\kappa$ denote the set of sequences
$\mathsf v$ of length $\kappa+1$ such that $\mathcal H(R,\mathsf v)$ is not empty. We let
$\mathcal V = \cup_{\kappa \geq 1} \mathcal  V_\kappa$.

In order to conclude the proof of Lemma~\ref{lm :: lower bound inductive step}, we next bound
$R^{-1}\|X\|_{\mathcal H(R,\mathsf v)}$ for all $\mathsf v \in \mathcal V$ and then use
Lemma~\ref{lm :: generic bound for partitioned collection} to obtain a bound on
$R^{-1}\mathfrak X(R)$. To this end, we divide \eqref{eq :: induction decomposition} by $R$ to obtain
\begin{equation}\label{eq :: induction decomposition scaled}
\frac{X(\eta)}{R} = \frac{\epsilon \nu(\gamma(\eta))}{R} +
	\sum_{j = 1}^{\kappa} \frac{|\mathsf v_{j} - \mathsf v_{j-1}|}{R}\frac{X(\eta_j)}{|\mathsf v_{j}-\mathsf v_{j-1}|}\,.
\end{equation}

Next we will state two lemmas which provide bounds for the terms on the right-hand side
of \eqref{eq :: induction decomposition scaled} for $\eta \in \mathcal H(R,\mathsf v)$.
\begin{lemma}\label{lm :: one skeleton bound}
	There exists a constant $C_8 > 0$ such that the following holds. Let
	$\epsilon < C_8^{-1}$ and $R \geq 1$. For all integers $\kappa \geq 1$ and
	$\mathsf v \in \mathcal V_\kappa$,
	\begin{equation}\label{eq-explain-s_kappa}
	\mathbb E \big[
		\sup_{\eta \in \mathcal H(R,\mathsf v)} \nu(\gamma(\eta))
	\big] \leq C_8 \sqrt{s_\kappa} \kappa^{3/2} R \leq C_8 \epsilon^2 \sqrt{\rho} \kappa R\,.
	\end{equation}
\end{lemma}
We note that Lemma~\ref{lm :: one skeleton bound} applies for a fixed vector $\mathsf v$, and that the
second inequality follows directly from the first by the definition of $s_\kappa$. The
$\epsilon^2 \sqrt{\rho}$ scaling is somewhat arbitrary in that we will only need the
fact that it is of smaller order than $\epsilon^{1/3} \log(\epsilon^{-1})$. If we increased the power of $\epsilon$ in $s_\kappa$, the term on the right hand side of \eqref{eq-explain-s_kappa} would shrink.
This would simultaneously increase the power of $\epsilon$ in the number of tiles, but this only contributes to a factor of $\log \epsilon^{-1}$ in later analysis where the power of $\epsilon$ only changes the constant in front of $\log \epsilon^{-1}$ (thus, we have flexibility in the choice of $s_\kappa$).

For $\kappa \geq 1$ and a sequence of points
$\mathsf v = (\mathsf v_0,\ldots, \mathsf v_{\kappa})$ in $s_\kappa R \cdot \mathbb Z^2$, let
\[
\mathcal H(R,\mathsf v,j) =
	\{\eta_j \,:\, \eta \in \mathcal H(R,\mathsf v)\}
\]
be the collection of all possible $j$-th segments for curves in
$\mathcal H(R, \mathsf v)$.
\begin{lemma}\label{lm :: strip bound}
	There exists a constant $C_{9} > 0$ such that the following holds. Let $\epsilon$
	and $R$ satisfy the assumptions of Lemma~\ref{lm :: lower bound inductive step}.
	For $\kappa \geq 1$, $\mathsf v = (\mathsf v_0,\ldots,\mathsf v_{\kappa})$,
	and $1 \leq j \leq \kappa$, we have
	\[
	\mathbb E\left[
		\frac{\|X\|_{\mathcal H(R,\mathsf v, j)}}{|\mathsf v_{j}-\mathsf v_{j-1}|}
	\right] \leq
		\mathbb E \left[
			\frac{\mathfrak X(10^{-1}R)}{10^{-1}R}
		\right] + C_{9}\frac{R}{|\mathsf v_{j} - \mathsf v_{j-1}|}\frac{\rho^2}{\log(\epsilon^{-1})}\,.
	\]
\end{lemma}
We note that $\rho^2/\log(\epsilon^{-1}) = \epsilon^{4/3} \log(\epsilon^{-1})$ which
is the desired order for the increase in $\mathfrak X$. The proofs of Lemmas~\ref{lm :: one skeleton bound} and \ref{lm :: strip bound} are deferred to the next subsection.

Next, assuming
Lemma~\ref{lm :: one skeleton bound} and \ref{lm :: strip bound}, we provide a bound on
$R^{-1}\|X\|_{\mathcal H(R,\mathsf v)}$. To simplify notation, we let $S = 10^{-1}R$
and
$\mu_S = \mathbb E (\mathfrak X(S)/S)$. Plugging the bounds from Lemmas~\ref{lm :: one skeleton bound} and
\ref{lm :: strip bound} into \eqref{eq :: induction decomposition scaled}, we obtain that
\begin{align*}
\mathbb E \left[ \frac{\|X\|_{\mathcal H(R, \mathsf v)}}{R}\right]
	&\leq \mu_S \sum_{j = 1}^\kappa \frac{|\mathsf v_j - \mathsf v_{j-1}|}{R} +
		\frac{C_{9} \rho^2}{\log(\epsilon^{-1})} \kappa + C_8 \epsilon^3\sqrt{\rho}\kappa \\
	&\leq \mu_S \sum_{j = 1}^\kappa \frac{|\mathsf v_{j} - \mathsf v_{j-1}|}{R} +
		\frac{C \rho^2}{\log(\epsilon^{-1})}\kappa\,,
			\numberthis \label{eq :: crude max on one skeleton}
\end{align*}
where the second inequality follows from the fact that $\epsilon^3 \sqrt{\rho}$ is of
lower order than $\rho^2/\log(\epsilon^{-1})$. If $\kappa = 1$, then
$|\mathsf v_1 - \mathsf v_0| \geq (1-2\epsilon^4)R - \sqrt{2}s_0R \geq (1-3\epsilon^4)R$ so
\eqref{eq :: crude max on one skeleton} gives
\begin{equation}\label{eq-kappa=1}
\mathbb E \left[ \frac{\|X\|_{\mathcal H(R, \mathsf v)}}{R}\right]  \leq
	(1 - 3\epsilon^4)\mu_S + \frac{C \rho^2}{\log(\epsilon^{-1})} \leq \mu_S  + \frac{C \rho^2}{\log(\epsilon^{-1})}\,,
\end{equation}
where the second inequality follows by absorbing the $-3 \epsilon^4 \mu_S$ term and adjusting the value of $C$ (note that $\mathcal H(x,y)$ contains the line segment from $x$ to $y$, and thus $\mathfrak X(x,y) \geq -1$
almost surely and in particular $\mu_S \geq -1$).
To treat the case $\kappa > 1$, we need the following geometric bound.
\begin{claim} \label{clm :: skeleton length bound}
There exists a constant $c_{10}$ such that for all $\epsilon < c_{10}$, $R \geq1$,
$\kappa > 1$, and $\mathsf v \in \mathcal V_\kappa$
\[
\sum_{j = 1}^\kappa \frac{|\mathsf v_{j} - \mathsf v_{j-1}|}{R} \geq 1 + c_{10} \rho^2 \kappa\,.
\]
\end{claim}
We defer the proof of Claim~\ref{clm :: skeleton length bound} to the next subsection and move on with the proof of Lemma~\ref{lm :: lower bound inductive step}.
Plugging Claim~\ref{clm :: skeleton length bound} into
\eqref{eq :: crude max on one skeleton} and using
the assumption that $\mu_S \leq - 1/2$ gives
\[
\mathbb E \left[ \frac{\|X\|_{\mathcal H(R, \mathbf v})}{R}\right]
	\leq \mu_S - c \rho^2 \kappa\,.
\]
Combined with \eqref{eq-kappa=1}, it yields that for an absolute constant $C_{11} > 0$
\begin{equation}\label{eq :: max on one skeleton}
\mathbb E \left[
	\frac{\|X\|_{\mathcal H(R, \mathbf v)}}{R}
	\right] \leq \mu_S +  \frac{C_{11}\rho^2}{\log(\epsilon^{-1})} - \frac{2\rho^2}{C_{11}} (\kappa-1)\quad \kappa \geq 1\,.
\end{equation}
To apply Lemma~\ref{lm :: generic bound for partitioned collection}, we need bounds
on $|\mathcal V_\kappa|$ and
\[
\sigma_\kappa^2 := R^{-2} \max_{\mathsf v \in \mathcal V_\kappa}
	\sup_{\eta \in \mathcal H(R,\mathsf v)} \var[X(\eta)]\,.
\]
We first bound $|\mathcal V_\kappa|$. For $0 < j < \kappa$, given $\mathsf v_{j-1}$
there are at most $10^4/s_\kappa$ possible choices for $\mathsf v_{j}$ since $T_{\kappa,\mathsf v_j}$ must intersect
the following set
$$\{(x_j,y_j)\in (-9R,10R)\times(-9R,9R):  |y_j - y_{j-1}| = \rho R \mbox{ for some }
(x_{j-1},y_{j-1}) \in T_{\kappa,\mathsf v_{j-1}}\}\,.$$
For $j = 0$ (respectively $j = \kappa$), there are at most $16 \epsilon^8/s_\kappa^2$ possible choices for $\mathsf v_j$ since
$T_{\kappa,\mathsf v_j}$ must intersect the ball of radius $R \epsilon^4$ around $o$ (respectively $(R,0)$).
Therefore, we have that $|\mathcal V_\kappa| \leq 16^2 \times 10^{4(\kappa-1)}\epsilon^{16}s_\kappa^{-(\kappa+3)}$ and thus we get
\begin{equation}\label{eq-bound-V-kappa}
\log(|\mathcal V(\epsilon,\kappa)|) \leq c (\kappa+3)\log(s_\kappa^{-1}) \leq
	c\kappa(\log(\kappa) + \log(\epsilon^{-1}))\,.
\end{equation}
We next bound $\sigma_\kappa$. For $\mathsf v \in \mathcal V_\kappa$ and
every $\eta \in \mathcal H(R,\mathsf v)$, there exists a box of height
$\rho R(\kappa+3) \leq 4\rho R \kappa$ and width $19R$ that contains $\eta$.
Together with \eqref{eq :: H variance bound} this gives
\begin{equation}\label{eq :: induction horizontal variance bound}
\sigma_\kappa^2 \leq 10^4 \epsilon^2\min\{\rho \kappa, 1\}.
\end{equation}
Combining \eqref{eq :: induction horizontal variance bound} with \eqref{eq-bound-V-kappa}, we get that for a universal constant $C_{12}>0$,
\begin{equation}\label{eq :: induction horizontal aggregation first step bound}
\sigma_\kappa \sqrt{\log(|\mathcal V(\kappa,\epsilon)|)} \leq C_{12} \epsilon \sqrt{\rho \log(\epsilon^{-1})} \kappa
	= \frac{C_{12} \rho^2}{\log(\epsilon^{-1})} \kappa\,,
\end{equation}
where for $\kappa \leq \rho^{-1}$ we used the first bound on the minimum in \eqref{eq :: induction horizontal variance bound} and for $\kappa >\rho^{-1}$ we used the second bound on the minimum in \eqref{eq :: induction horizontal variance bound}.

We are now ready to apply Lemma~\ref{lm :: generic bound for partitioned collection}
with $\mathcal G = \mathcal H(R)$ and $Y = R^{-1} X$. We let $A = \mathcal V$, and
for $n \geq 0$, $A_n = \mathcal V_{n+1}$. We let
\[
\mu = \mu_S + \frac{C_{11}\rho^2}{\log(\epsilon^{-1})}\,,
\]
and $\alpha = \rho^2/C_{11}$.  Then by \eqref{eq :: max on one skeleton}, for $n \geq 0$
\[
\mu_n:= \max_{\mathsf v\in \mathcal V_n} \mathbb E \left[
	\frac{\|X\|_{\mathcal H(R, \mathbf v})}{R}
	\right] \leq \mu - 2\alpha n\,.
\]
We let $\beta = 10^4\epsilon^2 \rho$. By \eqref{eq :: induction horizontal variance bound},
$\sigma_{n}^2 \leq \beta (n+1)$ for $n \geq 0$.
Finally, we let $\gamma = 4C_{12}\rho^2/\log(\epsilon^{-1})$.
If $\epsilon$ is smaller than some fixed constant, $\gamma \leq \alpha$ and so we obtain from \eqref{eq :: induction horizontal aggregation first step bound} that
\[
\sigma_n\sqrt{\log(|A_n|)} \leq \frac{\gamma(n+1)}{4} \leq \frac{\gamma + \alpha n}{4}\,.
\]
By Lemma~\ref{lm :: generic bound for partitioned collection}, we have
\[
\mathbb E \left[\frac{\mathfrak X(R)}{R}\right] \leq
	\mu + \gamma + \sqrt{2\pi \beta}  +
	4 \frac{\beta}{\alpha} \frac{e^{-\alpha^2/4\beta}}{1 - e^{-\alpha^2/4\beta}}\,.
\]
Note that
\[
\mu + \gamma =
	\mu_S + \frac{(C_{11} + 4C_{12})\rho^2}{\log(\epsilon^{-1})}\,.
\]
There exists a constant $C > 0$ such that
$\beta/\alpha \leq C \epsilon^2/\rho = C \epsilon^{4/3} / \log(\epsilon^{-1})$,
$\alpha^2/\beta \geq C^{-1} \rho^3/\epsilon^2 = C^{-1}(\log(\epsilon^{-1}))^3$, and
$\sqrt{\beta} \leq C \epsilon \sqrt{\rho} = C \epsilon^{4/3} \sqrt{\log(\epsilon^{-1})}$.
Therefore, for $\epsilon$ small enough,
\[
\sqrt{2\pi \beta}  +
	4 \frac{\beta}{\alpha} \frac{e^{-\alpha^2/4\beta}}{1 - e^{-\alpha^2/4\beta}} \leq
	\epsilon^{4/3} \log(\epsilon^{-1})\,,
\]
and thus
\[
\mathbb E \left[\frac{\mathfrak X(R)}{R}\right] \leq
	\mu_S + C'\epsilon^{4/3} \log(\epsilon^{-1})\,,
\]
for some universal constant $C' > 0$.
This concludes the proof of Lemma~\ref{lm :: lower bound inductive step}.

\subsection{Proofs for deferred lemmas}\label{sec :: lower deferred}

In this subsection we provide proofs for a few lemmas in Subsection~\ref{sec :: lower induction}.

\subsubsection{Proof of Claim~\ref{clm :: skeleton length bound}}

Our proof consists of three steps, where in the first two steps we reduce the claim to a collection of more tractable sequences and in the third step we prove the claim for this collection.

\noindent {\bf Step 1.}	It suffices to show that there exists a constant $c$ such
	that the following holds: for all $\kappa > 1$ and for all $\eta \in \mathcal H(R)$ such that
	$\mathbf v(\eta)  = (v_0, \ldots, v_{\kappa})$ has length $\kappa+1$, we have that
	\begin{equation}\label{eq :: skeleton length bound 2}
	\sum_{j = 1}^\kappa \frac{|v_{j} - v_{j-1}|}{R} \geq 1 + c \rho^2 \kappa\,.
	\end{equation}	
	Indeed, if $\eta \in \mathcal H(R,\mathsf v)$, then $|\mathsf v_{j} - v_j| \leq
	2^{-1/2} s_\kappa R$ for $0 \leq j \leq \kappa$. Therefore
	\[
	\sum_{j = 1}^\kappa \frac{|\mathsf v_{j} - \mathsf v_{j-1}|}{R} \geq
	\sum_{j = 1}^\kappa \frac{|v_{j} - v_{j-1}|}{R} - \sqrt{2}\kappa s_\kappa \geq
	\sum_{j = 1}^\kappa \frac{|v_{j} - v_{j-1}|}{R} - \rho \epsilon^4\,.
	\]
	Since $\rho \epsilon^4$ is of smaller order than $\rho^2$,
	Claim~\ref{clm :: skeleton length bound} follows after plugging
	\eqref{eq :: skeleton length bound 2} into the last display.
	
\noindent {\bf Step 2.} It suffices to show that \eqref{eq :: skeleton length bound 2} holds
	for all $\kappa > 1$ and sequences $\mathbf v = (v_0,\ldots,v_\kappa)$ that satisfy the following
	conditions (recall the definition of $L_k$ in \eqref{eq-def-Pi-L})
	\begin{equation}\label{eq-geometric-second-reduction-condition}
\begin{split}
		&v_0 = o \mbox{ and } v_\kappa = (R,0);\\
		&\mbox{there exists } k_j \mbox{ such that } v_j \in L_{k_j}, \mbox { for } 0 \leq j \leq \kappa;\\
		&|k_j - k_{j-1}| = 1 \mbox{ for } 1 \leq j \leq \kappa\,.
\end{split}
	\end{equation}
	We will denote the set of all sequences that satisfy these conditions by $A$,
	and the set of all such sequences of length $\kappa+1$ by $A_\kappa$. To see that it suffices
	to prove \eqref{eq :: skeleton length bound 2} for sequences in $A$, let $\eta \in
	\mathcal H(R)$ be such that $\mathbf v(\eta)$
	has length $\kappa+1$, and let $\mathbf w = (w_0,\ldots,w_\kappa)$ be the sequence given by $w_j =
	v_j - v_0$ for $0 \leq j < \kappa$ and $w_\kappa = (R,0)$. Then $\mathbf w \in A_\kappa$. Since
	$|v_0|$ and $|v_\kappa - (R,0)|$ are bounded by $R \epsilon^4$,
	\[
	\sum_{j = 1}^\kappa \frac{|v_j - v_{j-1}|}{R} \geq
	\sum_{j = 1}^\kappa \frac{|w_j - w_{j-1}|}{R} - 2 \epsilon^4\,,
	\]
which proves  Claim~\ref{clm :: skeleton length bound} provided \eqref{eq :: skeleton length bound 2} holds for all $\mathbf w \in A_\kappa$.

\noindent{\bf Step 3.} We now prove \eqref{eq :: skeleton length bound 2} holds for all $\mathbf w \in A_\kappa$ for $\kappa>1$ by induction.
	Note that $A_\kappa$ is empty unless $\kappa$ is even since for every $\mathbf v \in A_\kappa$, the
	corresponding sequence $(k_0,\ldots,k_\kappa)$ describes a walk with increments in
	$\{-1,1\}$ that starts and ends at 0 (and therefore has an even number of steps).
	
	We first consider the base case (i.e, $\kappa = 2$). If $\mathbf v \in A_2$, then $(v_0,v_1,v_2)$ forms a triangle with
	base $R$ and height $\rho R$. Recall that if $T$ is a triangle with base $b$
	and height $h$ such that $h \leq b$ then
	\[
	l(\partial T) - b \geq 2\sqrt{\frac{b^2}{4} + h^2}  \geq b + \frac{h^2}{b}\,.
	\]
	Therefore,
	\[
	\frac{|v_1 - v_0| + |v_2 - v_1|}{R} \geq 1 + \rho^2\,.
	\]

	Now, let $\kappa > 2$ be even and assume \eqref{eq :: skeleton length bound 2} holds for all
	sequences in $A_{\kappa-2}$. Let $\mathbf v$ be a sequence in $A_\kappa$ and $\mathbf k = (k_0, \ldots, k_{\kappa})$ be the sequence
	such that $v_j \in L_{k_j}$ for $0 \leq j \leq \kappa$. There exists $1 \leq j \leq \kappa-2$
	such that $|k_j| > |k_{j-1}|$ and $|k_j| > |k_{j+1}|$ (and thus $k_{j-1} = k_{j+1}$). Let
	$\mathbf w = (w_0,\ldots,w_{\kappa-2})$ be the sequence obtained by removing $v_j$ and $v_{j+1}$
	from $\mathbf v$. Note that $\mathbf w \in A_{\kappa-2}$. Since $k_{j+1} = k_{j-1}$ and
	$|v_{j+1} - v_{j-1}| \leq 20 R$ we have that $(v_{j-1}, v_j,v_{j+1})$ forms a triangle
	of height $\rho R$ and base at most $20R$. This gives
	\begin{equation}\label{eq-triangle-inequality-strict}
	  \frac{|v_{j} - v_{j-1}| + |v_{j+1} - v_j| - |v_{j+1} - v_{j-1}|}{R} \geq
	\frac{\rho^2}{20}\,.
	\end{equation}
Therefore, we have that
	\begin{align*}
	\sum_{j = 1}^{\kappa} \frac{|v_{j} - v_{j-1}|}{R} -
	\sum_{i = 1}^{\kappa-2} \frac{|w_{i} - w_{i-1} |}{R} &=
	\frac{|v_{j} - v_{j-1}| + |v_{j+1} - v_j| + |v_{j+2} - v_{j+1}| - |v_{j+2} - v_{j-1}|}{R} \\
	&\geq \frac{|v_{j} - v_{j-1}| + |v_{j+1} - v_j| - |v_{j+1} - v_{j-1}|}{R} \geq  \frac{\rho^2}{20}\,,
	\end{align*}
where the first inequality follows from the triangle inequality and the second inequality follows from
\eqref{eq-triangle-inequality-strict}. Combined with the induction hypothesis, this proves
\eqref{eq :: skeleton length bound 2} holds for $\mathbf v \in A_\kappa$ and concludes the proof.

\subsubsection{Proof of Lemma~\ref{lm :: one skeleton bound}}
The proof is an application of Lemma~\ref{lm :: max over simple collections}, which requires us to set parameters as in the context of Lemma~\ref{lm :: max over simple collections}. For
$\mathsf v \in \mathcal V_\kappa$, $\eta, \eta' \in \mathcal H(R,\mathsf v)$ and
$1 \leq j \leq \kappa$, the vertical distance between $v_{j}$ and $v_{j-1}$ is $\rho R$ and
the horizontal distance is bounded by $19R$. Thus, $|v_{j} - v_{j-1}| \leq 20 R$
provided we take $\epsilon$ small enough. In addition,
$|v_\kappa - v'_0| \leq (1 + 2\epsilon^4)R \leq 20 R$. Thus, we let
$\Delta_1 = \Delta_2 = 20R$. For $0 \leq j \leq \kappa$, we have $v_j \in T_{\kappa,\mathsf v_j}$, and so we let
$\Delta_3 = 2^{1/2}s_\kappa R$.
Furthermore, by Claim~\ref{clm :: canonical distance bound}, for $\eta,\eta' \in \mathcal H(R,\mathsf v)$ we have
\begin{align*}
d_\nu(\gamma(\eta),\gamma(\eta')) \leq
	\sum_{j = 0}^{\kappa} \sqrt{\frac{|v_j - v_{j-1}| + |v'_{j+1} - v'_j|}{2} |v_j - v'_j|} \leq
	\sqrt{\Delta_2\Delta_3} (\kappa+1) \,.
\end{align*}
Therefore, we can set $\sigma = \sqrt{\Delta_2\Delta_3} (\kappa+1) = C \sqrt{s_\kappa}(\kappa+1)R$. As a result, we have
\begin{align*}
\Delta_4& =
	\min\left(\Delta_1 + \frac{1}{2\Delta_1} \left(\frac{\sigma}{\kappa+1}\right)^2, \Delta_3\right) +
			\frac{1}{\Delta_1} \left(\frac{\sigma}{\kappa+1}\right)^2 = 2\Delta_3\,,\\
\Delta_5 &= \sqrt{2\pi \Delta_1^{(\kappa-1)/(\kappa+1)} \Delta_4} \Delta_2^{1/(\kappa+1)}
	= \sqrt{2\pi\Delta_1\Delta_4}
	= \frac{2\sqrt{\pi} \sigma}{\kappa+1}\,.
\end{align*}
Applying Lemma~\ref{lm :: max over simple collections} with the aforementioned choices of parameters, we conclude that
\[
\mathbb E \left[
	\sup_{\eta \in \mathcal H(R,\mathsf v)} \nu(\gamma(\eta))
\right] \leq
	C_4 \sigma\sqrt{(\kappa+1)\log\left(\frac{\Delta_5 (\kappa+1)}{\sigma}\right)}  \leq
	C \sqrt{s_\kappa} \kappa^{3/2} R\,.
\]

\subsubsection{Proof of Lemma~\ref{lm :: strip bound}}
In the proof of Lemma~\ref{lm :: lower bound inductive step}, we have decomposed the curve depending on its vertical oscillations so that each sub-curve in the decomposition has small (vertical) height.
However, these sub-curves may have large (horizontal) widths and thus we may not yet be able to apply our induction hypothesis.
Therefore, to prove Lemma~\ref{lm :: strip bound} we will have to further decompose the sub-curve so that each sub-sub-curve has small width (in addition to small height).
Naturally, the proof of Lemma~\ref{lm :: strip bound} follows a similar outline as the proof of Lemma~\ref{lm :: lower bound inductive step}.
In principle, it is possible to merge the two proofs into a single one.
We chose not to do so since it may further complicates the presentation by mixing the difficulties and results in more cumbersome notation. Thus, we compromise on the length of the arguments with the hope of improving the readability for a non-trivial inductive argument.

For ease of notation, we will write $\mathcal H'$ for $\mathcal H(R,\mathsf v, j)$.
For the rest of the section $\mathsf v$, $\kappa$, and $j$ will be fixed and refer to the values used to define $\mathcal H'$.
However, any constants appearing below do not depend on $\mathsf v$, $\kappa$, or $j$.
We also let $\mathcal R = (-9R,10R) \times (-9R,9R)$.
To avoid confusion, we will denote curves in $\mathcal H'$ by $\tau$ instead of $\eta$.
Unless specified otherwise, we will denote by $a = (x_a, y_a)$ and $b = (x_b, y_b)$ the start and end points for $\tau$ in what follows.
Write $\Pi = \cup_{i=-1}^1\Pi_{i}$.
Then we have that $\tau \in y_a + \Pi$.
Note that the height of $\Pi$ is $4\rho R$ and that the distance from $b$ to the top and bottom of $y_a + \Pi$ is at most $3\rho R$.

Let $S = 10^{-1}R$ and $r = \min\{S/2, |\mathsf v_{j} - \mathsf v_{j-1}|\}$.
By construction, $|y_a - y_b| = \rho R$ unless $\kappa = 1$ in which case $|b-a| \geq R/2$. Thus,
\begin{equation}\label{eq :: r lower bound}
|\mathsf v_{j} - \mathsf v_{j-1}| \geq |a-b| - \sqrt{2}s_\kappa R \geq (1 - \epsilon^4)\rho R\,.
\end{equation}
 Since $a,b \in (-9R,10R) \times (-9R,9R)$, we have $|a-b| \leq \sqrt{19^2 + \rho^2} R$ and therefore
\begin{equation}\label{eq :: segment total displacement upper bound}
|\mathsf v_{j} - \mathsf v_{j-1}| \leq |a-b| + \sqrt{2}s_\kappa R \leq 20R\,.
\end{equation}
Combined with \eqref{eq :: r lower bound}, this yields that
\begin{equation}\label{eq-r-lower-bound-combined}
r \geq \max\{(1 - \epsilon^4)\rho R,   400^{-1} |\mathsf v_j - \mathsf v_{j-1}|\}.
\end{equation}
For each $\tau \in \mathcal H'$ we wish to construct a sequence $\mathbf u = \mathbf u(\tau)$ that decomposes $\tau$ into segments with width of order $r$.
To this end, for $m \in \mathbb Z$ define $\Pi'_m$ and $L'_m$ be a vertical strip and a vertical line, given by
\[
\Pi'_m = \{(x,y) \,:\, (m-1) r < x < (m+1)r \} \mbox { and } L'_m = \{(mr, y) \,:\, y \in \mathbb R\}\,.
\]
We will follow a similar procedure as that used to define $\mathbf v$ in terms of $\mathbf w$ in the proof of Lemma~\ref{lm :: lower bound inductive step}.
We begin by defining for each $\tau$ a sequence $\mathbf z(\tau)$.
We let $z_0 = a$ and $\tau'_0 = \tau$.
For $i\geq 0$, as long as $z_i \neq b$ we have $z_i \in x_a + L'_{m_i}$ for some $m_i$.
In this case, we let $z_{i+1}$ be the first point of $\tau'_{i}$ on $x_a + \partial \Pi'_{m_i}$ (or $b$ if no such point exists).
We also let $\phi_{i+1}$ be the segment of $\tau'_i$ from $z_i$ to $z_{i+1}$, and let $\tau'_{i+1}$ be the segment of $\tau'_i$ from $z_{i+1}$ to $b$.
Continuing this construction until reaching $b$ produces a sequence of points $\mathbf z(\tau) = (z_0,\ldots, z_{n})$ and  a sequence of curves $(\phi_1,\ldots,\phi_n)$
connecting these points. In addition, we have $z_0 = a$, $z_{n} = b$, and for each $0 \leq i < n$ there exists
$m_i$ such that $z_i \in x_a + L'_{m_i}$.

We now construct $\mathbf u$ from $\mathbf z$.
Let $m_b$ be such that $(x_a + L'_{m_b})$ is closest to $b$ among all vertical lines of form $(x_a + L'_m)$.
Let $i^* = \max\{0 \leq i  < n \,:\, |m_i - m_b| = 2\}$ with the convention that $i^* = -1$ if $|m_i - m_b| \leq 1$ for $0 \leq i < n$.
Let $\mathbf u(\tau) = (z_0,\ldots,z_{i^*+1},z_n)$.
We let $\chi = i^*+2$ and for $1 \leq i \leq \chi$, we let $\tau_i$ be the segment of
$\tau$ from $u_{i-1}$ to $u_i$.
By an abuse of notation, for $0 \leq i < \chi$, we let $m_i$ be such that $u_i \in x_a + L'_{m_i}$ and $m_\chi = m_b$.
Let $(\tau', S_1, S_2)$ be a witness for $\tau$, let $O_i = (x_a + \Pi'_{m_{i-1}}) \cap S_2$ for $1 \leq i \leq \chi-1$, and let $O_\chi = (x_a + (\Pi'_{m_b-1} \cup \Pi'_{m_b} \cup \Pi'_{m_b+1})) \cap S_2$.
Then $(O_1,\ldots, O_\chi)$ and $(\tau_1,\ldots,\tau_\chi)$ satisfy the assumptions of Lemma~\ref{lm :: splitting good curves} and therefore $(\tau_1, \ldots, \tau_\chi)$ are good curves.
For $1 \leq i \leq \chi$ we let $\gamma'_i$ be the line segment from $u_{i-1}$ to $u_i$ and let $\gamma'(\tau) = \gamma'_1 \ldots \gamma'_\chi$.

We next partition
$\mathcal H'$ using $\mathbf u$. For $\chi \geq 1$, let $s'_\chi = \epsilon^4\rho/(4\chi)$ and
for each $\mathbf u \in s'_\chi R \cdot \mathbb Z^2$ let $T_{\chi, \mathbf u}$ be the axis-aligned square of
side-length $s'_\chi R$ centered at $\mathbf u$. For a sequence $\mathsf u = (\mathsf u_0,\ldots, \mathsf u_{\chi})$ of
points in $s'_\chi R \cdot \mathbb Z^2$, we let $\mathcal H'(\mathsf u)$ be the set of curves $\tau \in \mathcal H'$ such that
$\mathbf u(\tau)$ has length $\chi+1$ and satisfies $u_i \in T_{\chi, \mathsf u_i}$ for
$0 \leq i \leq \chi$. For $\chi \geq 1$, we let $\mathcal U_\chi$ denote the set of sequences
$\mathsf u$ of length $\chi+1$ such that $\mathcal H'(\mathsf u)$ is not empty. We let
$\mathcal U = \cup_{\chi \geq 1} \mathcal U_\chi$.

As in the proof of Lemma~\ref{lm :: lower bound inductive step}, the goal is to bound
$|\mathsf v_j - \mathsf v_{j-1}|^{-1}\|X\|_{\mathcal H'(\mathsf u)}$ for all
$\mathsf u \in \mathcal U$ and then use Lemma~\ref{lm :: generic bound for partitioned collection}
to obtain a bound on $|\mathsf v_j - \mathsf v_{j-1}|^{-1}\|X\|_{\mathcal H'}$. To this end, we decompose
$X(\tau)$ (as in \eqref{eq :: induction decomposition}) by
\begin{equation}\label{eq :: induction2 decomposition scaled}
\frac{X(\tau)}{|\mathsf v_j - \mathsf v_{j-1}|} =
	\frac{\epsilon \nu(\gamma'(\tau))}{|\mathsf v_j - \mathsf v_{j-1}|} +
		\sum_{i = 1}^{\chi} \frac{|\mathsf u_{i} - \mathsf u_{i-1}|}{|\mathsf v_j - \mathsf v_{j-1}|}\frac{X(\tau_i)}{|\mathsf u_{i}- \mathsf u_{i-1}|}\,.
\end{equation}
As before, to bound  $\|X\|_{\mathcal H'(\mathsf u)}$ we need bounds on
$\nu(\gamma'(\tau))$ and $X(\tau_i)$. The next lemma bounds $\nu(\gamma'(\tau))$, and after proving
it we turn to the bound for $X(\tau_i)$.
\begin{lemma}\label{lm :: bound strip skeleton}
	There exists a constant $C_{13} > 0$ such that the following holds. Let
	$\epsilon < C_{13}^{-1}$, $R \geq 1$, $\chi \geq 1$, and
	$\mathsf u \in \mathcal U_\chi$. Then
	\begin{equation}\label{eq :: bound strip skeleton}
	\mathbb E \big[
	|\mathsf v_{j} - \mathsf v_{j-1}|^{-1} \sup_{\tau \in \mathcal H'(\mathsf u)} \nu(\gamma'(\tau))
	\big] \leq C_{13} \epsilon^2 \chi \,.
	\end{equation}
\end{lemma}
\begin{proof}
	The proof is an application of Lemma~\ref{lm :: max over simple collections}, which requires us to
	specify the parameters in Lemma~\ref{lm :: max over simple collections}. For
	$\mathsf u \in \mathcal U_\chi$, $\tau \in \mathcal H'(\mathsf u)$ and $1 \leq i \leq \chi$, the horizontal
	distance between $u_{i}$ and $u_{i-1}$ is at most $3r/2 \leq 3 |\mathsf v_j - \mathsf v_{j-1}|/2$,
	and the vertical distance is at most
	\[
	4\rho R \leq 4(1- \epsilon^4)^{-1}|\mathsf v_j - \mathsf v_{j-1}|\,.
	\]
	Thus, $|u_{i} - u_{i-1}| \leq 6|\mathsf v_j - \mathsf v_{j-1}|$ for $\epsilon < 1/2$.
	Additionally,
	\[
	|u_\chi - u_0| \leq |\mathsf v_j - \mathsf v_{j-1}| + \sqrt{2} s'_\chi R \leq 2|\mathsf v_j - \mathsf v_{j-1}|.
	\]
	In light of these, we let $\Delta_1 = \Delta_2 = 6|\mathsf v_j - \mathsf v_{j-1}|$.
	For $0 \leq i \leq \chi$, we have $u_i \in T_{\chi,\mathsf u_i}$ and thus we let
	$\Delta_3 = 2^{1/2}s'_\chi R \leq \epsilon^4 |\mathsf v_j - \mathsf v_{j-1}|/(2\chi)$.
	By Claim~\ref{clm :: canonical distance bound}, for $\tau,\tau' \in \mathcal H'(\mathsf u)$ we have (below we write $\mathbf u' = (u'_0, \ldots, u'_{\chi}) = \mathbf u(\tau')$)
	\begin{align*}
	d_\nu(\gamma'(\tau),\gamma'(\tau')) \leq
	\sum_{i = 0}^{\chi} \sqrt{\frac{|u_i - u_{i-1}| + |u'_{i+1} - u'_i|}{2} |u_i - u'_i|} \leq
	\sqrt{\Delta_2\Delta_3} (\chi+1)\,.
	\end{align*}
	Therefore, we let
	$\sigma = \sqrt{\Delta_2\Delta_3} (\chi+1) \leq c \epsilon^2 \sqrt{\chi}|\mathsf v_j - \mathsf v_{j-1}|$.
	We have
	\begin{align*}
	\Delta_4 &=
	\min\Big(\Delta_1 + \frac{1}{2\Delta_1} \Big(\frac{\sigma}{\chi+1}\Big)^2, \Delta_3\Big) +
	\frac{1}{\Delta_1} \Big(\frac{\sigma}{\chi+1}\Big)^2 = 2 \Delta_3\,,\\
	\Delta_5 &= \sqrt{2\pi \Delta_1^{(\chi-1)/(\chi+1)} \Delta_4} \Delta_2^{1/(\chi+1)}
	= \sqrt{2\pi\Delta_1\Delta_4}
	= \frac{2\sqrt{\pi} \sigma}{\chi+1}\,.
	\end{align*}
Applying	Lemma~\ref{lm :: max over simple collections}, we then conclude that
	\begin{equation*}
	\mathbb E \big[
	\sup_{\eta \in \mathcal H(R,\mathsf v)} \nu(\gamma(\eta))
	\big] \leq
		C_4 \sigma\sqrt{(\chi+1)\log\big(\tfrac{\Delta_5 (\chi+1)}{\sigma}\big)}  \leq
		C \epsilon^2 \chi |\mathsf v_j - \mathsf v_{j-1}| \,. \qedhere
	\end{equation*}
\end{proof}

To bound $X(\tau_i)$, we introduce some notation.
For $\mathsf u \in \mathcal U_\chi$ and $1 \leq i \leq \chi$, let
\[
\mathcal H'(\mathsf u,i) = \{\tau_i \,:\, \tau \in \mathcal H'(\mathsf u)\}
\]
be the set of possible values of $\tau_i$ among curves in $\mathcal H'(\mathsf u)$.
To apply the induction hypothesis we need to show that $|\mathsf u_i - \mathsf u_{i-1}| \leq S$ and
$\mathcal H'(\mathsf u,i) \subset \mathcal H(\mathsf u_{i-1}, \mathsf u_i)$ for all $\chi\geq 1$,
$\mathsf u \in \mathcal U_\chi$, and $1\leq i\leq \chi$. To prove the first
statement we recall that by construction, if $\tau \in \mathcal H'(\mathsf u)$,
the horizontal distance between $u_i$ and $u_{i-1}$ is at most $3r/2$. Since $r \leq S/2$
and $\tau$ is contained in a horizontal strip of height $4\rho R$, we conclude
$|u_i - u_{i-1}| \leq 7 S/8$. Furthermore, since $|\mathsf u_i - u_i| \leq 2^{-1/2} s'_\chi R$
and the same holds for $i-1$, we conclude $|\mathsf u_i - \mathsf u_{i-1}| \leq S$. It remains
to prove that $\mathcal H'(\mathsf u,i) \subset \mathcal H(\mathsf u_{i-1}, \mathsf u_i)$.
\begin{claim}\label{clm :: lower bound induction hypothesis}
We have
$\mathcal H'(\mathsf u, i) \subset \mathcal H(\mathsf u_i, \mathsf u_{i-1})$ for all $\chi \geq 1$, $1 \leq i \leq \chi$ and $\mathsf u \in \mathcal U_\chi$,
\end{claim}
\begin{proof}
It suffices to show that for $1\leq i\leq \chi$ every  $\tau_i \in \mathcal H'(\mathsf u,i)$ satisfies the following:
\begin{itemize}
	\item The balls of radius $\epsilon^4|\mathsf u_{i} - \mathsf u_{i-1}|$ centered at $\mathsf u_i$ and
	$\mathsf u_{i-1}$ contain $T_{\chi, \mathsf u_i}$ and $T_{\chi, \mathsf u_{i-1}}$, respectively;
\item $\tau_i$ is contained in the union of the balls of radius
	$9|\mathsf u_{i} - \mathsf u_{i-1}|$ centered at $\mathsf u_i$ and $\mathsf u_{i-1}$.
\end{itemize}
We now prove the first claim. By construction, for $\tau \in \mathcal H'(\mathsf u)$, we have
$|u_i - u_{i-1}| \geq r/2$ and $|u_i - \mathsf u_i| \leq 2^{-1/2}s'_\chi R$ for
$0 \leq i \leq \chi$. Combined with \eqref{eq-r-lower-bound-combined}, this gives that
\[
|\mathsf u_i - \mathsf u_{i-1}| \geq \frac{(1 - \epsilon^4)\rho R}{2} - \sqrt{2} s'_1 R
\geq \frac{(1 - 2 \epsilon^4)\rho R}{2}\geq \frac{\rho R}{4}\,.
\]
This implies that $\epsilon^4 |\mathsf u_{i} - \mathsf u_{i-1}| \geq s'_1R$, completing the verification of the first claim.

It remains to prove the second claim. To simplify notation, we let $B(t,z)$ be
the ball of radius $t$ centered at $z$. We consider the case $i< \chi$ and $i = \chi$ separately.

If $i < \chi$, we have $\tau_i \subset x_a + \Pi'_{m_{i-1}}$. It follows that
$\tau_i \subset B(\sqrt{r^2 + 16\rho^2 R^2}, u_{i-1})$.
By \eqref{eq :: r lower bound}, we have $\rho R \leq (1-\epsilon^4)^{-1}r$. Taking $\epsilon$
small enough, we conclude $\tau_i \subset B(5r,u_{i-1})$. By construction,
$|u_i - u_{i-1}| \geq r$ and as we argued above
\[
|\mathsf u_i - \mathsf u_{i-1}| \geq |u_i - u_{i-1}| - \sqrt{2}s'_1 R \geq (1- \epsilon^4)r\,,
\]
and $|\mathsf u_{i-1} - u_{i-1}| \leq \epsilon^4 r$. Therefore,
$\tau_i \subset B(9|\mathsf u_i - \mathsf u_{i-1}|, \mathsf u_{i-1})$.

For $i = \chi$,  we see that $\tau_\chi$ is contained in
the vertical strip $x_a + (\Pi'_{m_\chi - 1} \cup \Pi'_{m_\chi} \cup \Pi'_{m_\chi+1})$ which
has width $4r$. By our choice of $m_\chi$, the distance from $u_\chi = b$ to either side
of this strip is at most $5 r /2$. Recall that $y_a + \Pi$ contains $\tau$ and the distance from $b$ to either side of the strip is at most $3\rho R$. Therefore,
\[
\tau_\chi \subset B\left(\sqrt{\frac{25 r^2}{4} + 9\rho^2 R^2}, u_\chi\right)\,.
\]
By \eqref{eq :: r lower bound}, we have $\rho R \leq (1-\epsilon^4)^{-1}r$, so taking $\epsilon$ small enough we get that $\tau_\chi \subset B(4r, u_\chi)$.
By construction, $|u_{\chi-1} - u_\chi| \geq r/2$ so we conclude (once again taking $\epsilon$ small enough) $\tau_\chi \subset B(9|\mathsf u_\chi - \mathsf u_{\chi-1}|, \mathsf u_\chi)$.
\end{proof}
It follows from Claim~\ref{clm :: lower bound induction hypothesis} and the bound
$|\mathsf u_i - \mathsf u_{i-1}| \leq S$ that
\begin{equation}\label{eq-X-bound-S}
\mathbb E \left[
	\frac{\|X\|_{\mathcal H'(\mathsf u, i)}}{|\mathsf u_{i} - \mathsf u_{i-1}|}
\right] \leq \mathbb E \left[\frac{\mathfrak X(S)}{S}\right] = : \mu_S\,.
\end{equation}
Plugging \eqref{eq :: bound strip skeleton} and \eqref{eq-X-bound-S} into
\eqref{eq :: induction2 decomposition scaled} gives that
\begin{equation}\label{eq :: trivial bound on one strip skeleton}
\mathbb E \left[ \frac{\|X\|_{\mathcal H'(\mathsf u)}}{|\mathsf v_{j} - \mathsf v_{j-1}|}\right] \leq
	\mu_S \sum_{i = 1}^\chi \frac{|\mathsf u_{i} - \mathsf u_{i-1}|}{|\mathsf v_{j} - \mathsf v_{j-1}|} +
		C_{13} \epsilon^3 \chi\,.
\end{equation}
As in the proof of Lemma~\ref{lm :: lower bound inductive step}, we need to bound
the sum of $|\mathsf u_i - \mathsf u_{i-1}|$, as incorporated in the next claim.
Let $D$ be the horizontal distance between $\mathsf v_j$ and $\mathsf v_{j-1}$. Throughout the rest of the section, we let $(x)_+ = \max(x,0)$ and $\chi_0= 1 + \lfloor (D - s_\kappa R)/r\rfloor$. Note that if $\chi < \chi_0$ then $\mathcal U_\chi$ is empty.
If $|\mathsf v_j - \mathsf v_{j-1}| \leq S/2$, then $r = |\mathsf v_j - \mathsf v_{j-1}|$ and $\chi_0 = 1$;
otherwise, $r = S/2 = R/20$ and $\chi_0 \leq 400$ (since $\mathcal R$ has width $19R$).
\begin{claim}\label{clm :: strip skeleton length bound}
There exists a constant $C_{14} > 0$ such that for all $\epsilon < C_{14}^{-1}$,
$\chi \geq \chi_0$, and $\mathsf u \in \mathcal U_\chi$,
\[
\sum_{i = 1}^\chi \frac{|\mathsf u_{i} - \mathsf u_{i-1}|}{|\mathsf v_{j} - \mathsf v_{j-1}|}
	\geq 1 - C_{14}\epsilon^4 + \frac{5(\chi-\chi_0-500)_+}{C_{14}}\,.
\]
\end{claim}
\begin{proof}
We could have proved a stronger version of the claim where we replace $(\chi-\chi_0-500)$ by $(\chi-\chi_0-1)$,
by essentially the same argument as in the proof of Claim~\ref{clm :: skeleton length bound}.
We chose to present a weaker bound as it suffices and is almost obvious.
For $\chi \leq \chi_0+500$, this simply follows form triangle inequality together
 with the fact that $|\mathsf u_i - u_i| \leq s'_\chi R \leq R \epsilon^4 \rho/(4\chi)$.
 For $\chi \geq \chi_0+500$, this follows since
 \begin{align*}
 \sum_{i = 1}^\chi |\mathsf u_{i} - \mathsf u_{i-1}| &\geq r(\chi-1) - 2 \sum_{i = 0}^\chi |\mathsf u_{i} - u_{i}| \\
 	&\geq r(\chi-1) - \sqrt{2}(\chi + 1) s'_\chi R \\
 	&\geq |\mathsf v_{j} - \mathsf v_{j-1}| + \frac{r(\chi-\chi_0-500)}{2}\,,
 \end{align*}
where the last inequality follows from \eqref{eq-r-lower-bound-combined} and $s'_\chi = \epsilon^4\rho/(4\chi)$.
This implies the claim (by \eqref{eq-r-lower-bound-combined} again).
\end{proof}

Recall that $-1 \leq \mu_S \leq -1/2$ and $\chi_0 \leq 400$. Combining Claim~\ref{clm :: strip skeleton length bound} and
\eqref{eq :: trivial bound on one strip skeleton} we get that for $\epsilon$ small
enough and some constant $C_{15} > 0$,
\begin{equation}\label{eq :: max on one strip skeleton}
\mathbb E \left[ \frac{\|X\|_{\mathcal H'(\mathsf u)}}{|\mathsf v_{j} - \mathsf v_{j-1}|}\right] \leq
	\mu_S + C_{15} \epsilon^{3} - \frac{2(\chi-\chi_0-500)_+}{C_{15}}\,.
\end{equation}
Therefore, we will apply Lemma~\ref{lm :: generic bound for partitioned collection} with
$Y = |\mathsf v_j - \mathsf v_{j-1}|^{-1} X$, $\mathcal G = \mathcal H'$, $A = \mathcal U$,
$A_0 = \cup_{i=0}^{500}\mathcal U_{\chi_0+i}$, and $A_n = \mathcal U_{\chi_0+500+n}$ for $n \geq 1$.
Write $\mu = \mu_S + C_{15} \epsilon^3$ and $\alpha = C_{15}^{-1}$. By
\eqref{eq :: max on one strip skeleton} we get that
\[
\mu_n := \sup_{\mathsf u\in A_n} \mathbb E \left[ \frac{\|X\|_{\mathcal H'(\mathsf u)}}{|\mathsf v_{j} - \mathsf v_{j-1}|}\right]
	\leq \mu - 2 \alpha n \,.
\]
To conclude, we need bounds
on $|A_n|$ and
\[
\sigma_n^2 := |\mathsf v_j - \mathsf v_{j-1}|^{-2} \max_{\mathsf u \in \mathcal A_n}
\sup_{\tau \in \mathcal H'(\mathsf u)} \var[X(\tau)]\,.
\]
We first bound $|A_n|$ (by bounding $|\mathcal U_\chi|$). Note that for $0 < i < \chi$, given $\mathsf u_{i-1}$ there
are at most $100 \rho /s'_\chi$ possible values for $\mathsf u_i$ since $T_{\chi, \mathsf u_i}$ must intersect
the following set:
\[
\{(x_1,y_1): |x_1 - x_2| = r \mbox{ and }
|y_1 - y_2| \leq 4\rho R  \mbox{ for some }
	(x_2,y_2) \in T_{\chi,\mathsf u_{i-1}}\}\,.
\]
 For $i \in \{0,\chi\}$,
there are at most $9\chi^2$ possible values for $\mathsf u_i$ since $T_{\chi, \mathsf u_i}$ must intersect $T_{\kappa,\mathsf v_{j-1}}$
if $i = 0$ and $T_{\kappa,\mathsf v_{j}}$ if $i = \chi$ (recall that $s_\kappa/s'_\chi \leq \chi$). Therefore, for $\epsilon$ small enough,
\[
\log(|\mathcal U_\chi|) \leq 5(\chi+3) [\log(\chi) + \log(\epsilon^{-1})]\,.
\]
Using once again the fact that $\chi_0 \leq 400$, we conclude that there exists $C > 0$
such that
\begin{equation}\label{eq-bound-A-n}
\log(|A_n|) \leq C(n+1)[\log(n+1) + \log(\epsilon^{-1})] \mbox{ for all } n\geq 0\,.
\end{equation}
To bound $\sigma_n$, note that all curves in $\mathcal H'$ are contained in a box of height $4\rho R$ and width $19 R$ and thus
\begin{equation}\label{eq :: sigma_n bound 2}
\sigma^2_n \leq \frac{10^4 \epsilon^2 \rho R^2}{|\mathsf v_{j} - \mathsf v_{j-1}|^2}\,.
\end{equation}
Therefore, we let $\beta = 10^4 \epsilon^2 \rho R^2/|\mathsf v_{j} - \mathsf v_{j-1}|^2$.
Combining \eqref{eq-bound-A-n}  with \eqref{eq :: sigma_n bound 2}, we obtain that (recall $\rho = \epsilon^{2/3}\log(\epsilon^{-1})$)
\[
\sigma_n \sqrt{\log(|A_n|) } \leq C' \epsilon \sqrt{\rho} \sqrt{\frac{\log(\max\{n,\epsilon^{-1}\})}{n+1}} \frac{R}{|\mathsf v_j - \mathsf v_{j-1}|} (n+1)
	\leq C' \frac{\rho^2}{\log(\epsilon^{-1})} \frac{R}{|\mathsf v_j - \mathsf v_{j-1}|} (n+1)\,.
\]
By \eqref{eq :: r lower bound}, $R/|\mathsf v_j - \mathsf v_{j-1}| \leq 2/\rho$.
Therefore, there exists a constant $C_{16} > 0$ such that
\[
\sigma_n \sqrt{\log(|A_n|) } \leq \frac{\gamma + \alpha n}{4} \quad \mbox{ for all } n \geq 0\,.
\]
where $\gamma = \frac{C_{16} \rho^2}{\log(\epsilon^{-1})} \frac{R}{|\mathsf v_j - \mathsf v_{j-1}|}$.
Finally, we apply Lemma~\ref{lm :: generic bound for partitioned collection} to obtain
\begin{align*}
\mathbb E \left[ \frac{\|X\|_{\mathcal H'}}{|\mathsf v_j - \mathsf v_{j-1}|}\right]
	&\leq \mu + \gamma + \sqrt{2\pi \beta}  +
		4 \frac{\beta}{\alpha} \frac{e^{-\alpha^2/4\beta}}{1 - e^{-\alpha^2/4\beta}}\,.
\end{align*}
Note that
\begin{align*}
\mu+\gamma \leq \mu_S + \frac{C\rho^2}{\log(\epsilon^{-1})} \frac{R}{|\mathsf v_j - \mathsf v_{j-1}|}\,,
	\mbox{ and }
\sqrt{\beta} \leq C \epsilon \sqrt{\rho} \frac{R}{|\mathsf v_j - \mathsf v_{j-1}|}
	\leq \frac{C\rho^2}{\log(\epsilon^{-1})} \frac{R}{|\mathsf v_j - \mathsf v_{j-1}|} \,.
\end{align*}
Further, $\beta/\alpha \leq C \epsilon^2 R/|\mathsf v_j - \mathsf v_{j-1}|$ and $\alpha^2/ \beta \geq c \rho \epsilon^{-2}$. Therefore,
\[
4 \frac{\beta}{\alpha} \frac{e^{-\alpha^2/4\beta}}{1 - e^{-\alpha^2/4\beta}} \leq
	C \epsilon^2 \exp(-c\epsilon^{-4/3}) \frac{R}{|\mathsf v_j - \mathsf v_{j-1}|}\,.
\]
Combining the last two displays concludes the proof of Lemma~\ref{lm :: strip bound}.

\medskip

\noindent {\bf Acknowledgement.} We thank Ron Peled for reinforcing our interest in the correlation length of RFIM, and for pointing out references \cite{bovier2006, Chalker83, SKBP14}.
We thank Jiaming Xia for discussions at an early stage of the project. Much of the work was carried out when J.D. was a faculty member at University of Pennsylvania.


\small

\begin{thebibliography}{10}

\bibitem{Adler90}
R.~J. Adler.
\newblock {\em An introduction to continuity, extrema, and related topics for
  general {G}aussian processes}, volume~12 of {\em Institute of Mathematical
  Statistics Lecture Notes---Monograph Series}.
\newblock Institute of Mathematical Statistics, Hayward, CA, 1990.

\bibitem{AHP20}
M.~Aizenman, M.~Harel, and R.~Peled.
\newblock Exponential decay of correlations in the {$2D$} random field {I}sing
  model.
\newblock {\em J. Statist. Phys.}, 180:304--331, 2020.

\bibitem{AizenmanPeled19}
M.~Aizenman and R.~Peled.
\newblock A power-law upper bound on the correlations in the {$2D$} random
  field {I}sing model.
\newblock {\em Comm. Math. Phys.}, 372(3):865--892, 2019.

\bibitem{AizenmanWehr89}
M.~Aizenman and J.~Wehr.
\newblock Rounding of first-order phase transitions in systems with quenched
  disorder.
\newblock {\em Phys. Rev. Lett.}, 62(21):2503--2506, 1989.

\bibitem{AizenmanWehr90}
M.~Aizenman and J.~Wehr.
\newblock Rounding effects of quenched randomness on first-order phase
  transitions.
\newblock {\em Comm. Math. Phys.}, 130(3):489--528, 1990.

\bibitem{AKT84}
M.~Ajtai, J.~Koml\'{o}s and G.~Tusn\'{a}dy.
\newblock On optimal matchings
\newblock {\em Combinatorica}, 4 (1984), no. 4, 259--264.

\bibitem{BN22}
Y.~Bar-Nir.
\newblock Upper and Lower Bounds for the Correlation Length of the Two-Dimensional Random-Field Ising Model.
\newblock {\em Preprint}, arXiv:2205.01522. 

\bibitem{Berretti85}
A.~Berretti.
\newblock Some properties of random {I}sing models.
\newblock {\em J. Statist. Phys.}, 38(3-4):483--496, 1985.



\bibitem{Binder83}
K.~Binder.
\newblock Random-field induced interface widths in ising systems.
\newblock {\em Zeitschrift für Physik B Condensed Matter}, 50:343--352, 1983.

\bibitem{Borell75}
C.~Borell.
\newblock The {B}runn-{M}inkowski inequality in {G}auss space.
\newblock {\em Invent. Math.}, 30(2):207--216, 1975.

\bibitem{bovier2006}
A.~Bovier.
\newblock {\em Statistical Mechanics of Disordered Systems: A Mathematical
  Perspective}.
\newblock Cambridge Series in Statistical and Probabilistic Mathematics.
  Cambridge University Press, 2006.

\bibitem{Bray_1985}
A.~J. Bray and M.~A. Moore.
\newblock Scaling theory of the random-field ising model.
\newblock {\em Journal of Physics C: Solid State Physics}, 18(28):L927--L933,
  oct 1985.

\bibitem{BricmontKupiainen88}
J.~Bricmont and A.~Kupiainen.
\newblock The hierarchical random field {I}sing model.
\newblock {\em J. Statist. Phys.}, 51(5-6):1021--1032, 1988.
\newblock New directions in statistical mechanics (Santa Barbara, CA, 1987).

\bibitem{BK88}
J.~Bricmont and A.~Kupiainen.
\newblock Phase transition in the $3$d random field {I}sing model.
\newblock {\em Comm. Math. Phys.}, 116(4):539--572, 1988.

\bibitem{Cacoullos82}
T.~Cacoullos.
\newblock On upper and lower bounds for the variance of a function of a random
  variable.
\newblock {\em Ann. Probab.}, 10(3):799--809, 1982.

\bibitem{CJN18}
F.~Camia, J.~Jiang, and C.~M. Newman.
\newblock A note on exponential decay in the random field {I}sing model.
\newblock {\em J. Stat. Phys.}, 173(2):268--284, 2018.

\bibitem{Chalker83}
J. Chalker.
\newblock On the lower critical dimensionality of the Ising model in a random field.
\newblock{\em J. Phys. C}, 16 (34): 6615--6622, 1983.


\bibitem{Chatterjee18}
S.~Chatterjee.
\newblock On the decay of correlations in the random field {I}sing model.
\newblock {\em Comm. Math. Phys.}, 362(1):253--267, 2018.

\bibitem{CGGK93}
J.~T. Cox, A.~Gandolfi, P.~S. Griffin, and H.~Kesten.
\newblock Greedy lattice animals. {I}. {U}pper bounds.
\newblock {\em Ann. Appl. Probab.}, 3(4):1151--1169, 1993.

\bibitem{DGK01}
A.~Dembo, A.~Gandolfi, and H.~Kesten.
\newblock Greedy lattice animals: negative values and unconstrained maxima.
\newblock {\em Ann. Probab.}, 29(1):205--241, 2001.

\bibitem{DS84}
B.~Derrida and Y.~Shnidman.
\newblock Possible line of critical points for a random field ising model in
  dimension 2.
\newblock {\em J. Physique Lett.}, 45(12):577--581, 1984.

\bibitem{DG19}
J.~Ding and S.~Goswami.
\newblock Upper bounds on {L}iouville first-passage percolation and
  {W}atabiki's prediction.
\newblock {\em Comm. Pure Appl. Math.}, 72(11):2331--2384, 2019.

\bibitem{DSS22}
J.~Ding, J.~Song and R.~Sun.
\newblock A New Correlation Inequality for Ising Models with External Fields.
\newblock {\em Probab. Theory Relat. Fields}, 2022.


\bibitem{DingXia19}
J.~Ding and J.~Xia.
\newblock Exponential decay of correlations in the two-dimensional random field
  {I}sing model.
\newblock {\em Inventiones}, 224:999–-1045 (2021).

\bibitem{DZ22}
J.~Ding and Z.~Zhuang
\newblock Long range order for random field Ising and Potts models.
\newblock {\em Communications on Pure and Applied Mathematics}, accepted.

\bibitem{Dudley67}
R.~M. Dudley.
\newblock The sizes of compact subsets of {H}ilbert space and continuity of
  {G}aussian processes.
\newblock {\em J. Functional Analysis}, 1:290--330, 1967.

\bibitem{Fernique71}
X.~Fernique.
\newblock R\'{e}gularit\'{e} de processus gaussiens.
\newblock {\em Invent. Math.}, 12:304--320, 1971.

\bibitem{FFS84}
D.~S. Fisher, J.~Fr\"{o}hlich, and T.~Spencer.
\newblock The {I}sing model in a random magnetic field.
\newblock {\em J. Statist. Phys.}, 34(5-6):863--870, 1984.

\bibitem{FKG}
C.~M. Fortuin, P.~W. Kasteleyn, and J.~Ginibre.
\newblock Correlation inequalities on some partially ordered sets.
\newblock {\em Comm. Math. Phys.}, 22:89--103, 1971.

\bibitem{FI84}
J.~Fr\"{o}hlich and J.~Z. Imbrie.
\newblock Improved perturbation expansion for disordered systems: beating
  {G}riffiths singularities.
\newblock {\em Comm. Math. Phys.}, 96(2):145--180, 1984.

\bibitem{FV99}
C.~Frontera and E.~Vives.
\newblock Numerical signs for a transition in the two-dimensional random field
  ising model at $t=0$.
\newblock {\em Phys. Rev. E}, 59:R1295--R1298, 1999.

\bibitem{GK94}
A.~Gandolfi and H.~Kesten.
\newblock Greedy lattice animals. {II}. {L}inear growth.
\newblock {\em Ann. Appl. Probab.}, 4(1):76--107, 1994.

\bibitem{GristenMa82}
G.~Grinstein and S.-K. Ma.
\newblock Roughening and lower critical dimension in the random-field ising
  model.
\newblock {\em Phys. Rev. Lett.}, 49:685--688, 1982.

\bibitem{GM83}
G.~Grinstein and S.-K. Ma.
\newblock Surface tension, roughening, and lower critical dimension in the
  random-field ising model.
\newblock {\em Phys. Rev. B}, 28:2588--2601, 1983.

\bibitem{Hammond06}
A.~Hammond.
\newblock Greedy lattice animals: geometry and criticality.
\newblock {\em Ann. Probab.}, 34(2):593--637, 2006.

\bibitem{Imbrie85}
J.~Z. Imbrie.
\newblock The ground state of the three-dimensional random-field {I}sing model.
\newblock {\em Comm. Math. Phys.}, 98(2):145--176, 1985.

\bibitem{ImryMa75}
Y.~Imry and S.-K. Ma.
\newblock Random-field instability of the ordered state of continuous symmetry.
\newblock {\em Phys. Rev. Lett.}, 35:1399--1401, 1975.

\bibitem{Lee93}
S.~Lee.
\newblock An inequality for greedy lattice animals.
\newblock {\em Ann. Appl. Probab.}, 3(4):1170--1188, 1993.

\bibitem{Lee97}
S.~Lee.
\newblock The continuity of {$M$} and {$N$} in greedy lattice animals.
\newblock {\em J. Theoret. Probab.}, 10(1):87--100, 1997.

\bibitem{Lee97b}
S.~Lee.
\newblock The power laws of {$M$} and {$N$} in greedy lattice animals.
\newblock {\em Stochastic Process. Appl.}, 69(2):275--287, 1997.

\bibitem{LS89}
T.~Leighton and P.~Shor.
\newblock Tight bounds for minimax grid matching with applications to the average case analysis of algorithms.
\newblock {\em Combinatorica}, 9 (1989), no. 2, 161--187.
\bibitem{Martin02}
J.~B. Martin.
\newblock Linear growth for greedy lattice animals.
\newblock {\em Stochastic Process. Appl.}, 98(1):43--66, 2002.

\bibitem{PS02}
G.~Parisi and N.~Sourlas.
\newblock Scale invariance in disordered systems: The example of the
  random-field ising model.
\newblock {\em Phys. Rev. Lett.}, 89:257204, 2002.

\bibitem{PIM81}
E.~Pytte, Y.~Imry, and D.~Mukamel.
\newblock Lower critical dimension and the roughening transition of the
  random-field ising model.
\newblock {\em Phys. Rev. Lett.}, 46:1173--1177, 1981.

\bibitem{Rigger95}
H.~Rieger.
\newblock Critical behavior of the three-dimensional random-field ising model:
  Two-exponent scaling and discontinuous transition.
\newblock {\em Phys. Rev. B}, 52:6659--6667, 1995.

\bibitem{Rieger_1993}
H.~Rieger and A.~P. Young.
\newblock Critical exponents of the three-dimensional random field ising model.
\newblock {\em Journal of Physics A: Mathematical and General},
  26(20):5279--5284, 1993.

\bibitem{Rigollet15}
P.~Rigollet and J.~H\"{u}ter.
\newblock {\em High Dimensional Statistics}.
\newblock 2017.
\newblock Lecture notes in progress, available at
  \textit{http://www-math.mit.edu/~rigollet/PDFs/RigNotes17.pdf}.

\bibitem{Rudin87}
W.~Rudin.
\newblock {\em Real and complex analysis}.
\newblock McGraw-Hill Book Co., New York, third edition, 1987.

\bibitem{SA01}
E.~T. Sepp\"al\"a and M.~J. Alava.
\newblock Susceptibility and percolation in two-dimensional random field ising
  magnets.
\newblock {\em Phys. Rev. E}, 63:066109, 2001.

\bibitem{SPA98}
E.~T. Sepp\"al\"a, V.~Pet\"aj\"a, and M.~J. Alava.
\newblock Disorder, order, and domain wall roughening in the two-dimensional
  random field ising model.
\newblock {\em Phys. Rev. E}, 58:R5217--R5220, 1998.

\bibitem{SKBP14}
G.~P. Shrivastav, M.~Kumar, V.~Banerjee, and S.~Puri.
\newblock Ground-state morphologies in the random-field ising model: Scaling
  properties and non-porod behavior.
\newblock {\em Phys. Rev. E}, 90:032140, Sep 2014.

\bibitem{SudakovTsirelson74}
V.~N. Sudakov and B.~S. Tsirel'son.
\newblock Extremal properties of half-spaces for spherically invariant
  measures.
\newblock {\em J. Sov. Math}, 9:9--18, 1978.

\bibitem{Talagrand87}
M.~Talagrand.
\newblock Regularity of {G}aussian processes.
\newblock {\em Acta Math.}, 159(1-2):99--149, 1987.

\bibitem{Talagrand14}
M.~Talagrand.
\newblock {\em Upper and Lower Bounds for Stochastic Processes}.
\newblock Springer-Verlag, Berlin Heidelberg, first edition, 2014.

\bibitem{Van16}
R.~van Handel.
\newblock {\em Probability in High Dimension}.
\newblock Lecture notes in progress, available at
  \textit{https://web.math.princeton.edu/rvan/APC550.pdf}.

\bibitem{vDHKP95}
H.~von Dreifus, A.~Klein, and J.~F. Perez.
\newblock Taming {G}riffiths' singularities: infinite differentiability of
  quenched correlation functions.
\newblock {\em Comm. Math. Phys.}, 170(1):21--39, 1995.


\bibitem{Yoav22}
B.~Yoav.
\newblock Upper and Lower Bounds for the Correlation Length of the Two-Dimensional Random-Field Ising Model.
\newblock {\em Preprint}, arXiv:2205.01522. 

\bibitem{YN85}
A.~P. Young and M.~Nauenberg.
\newblock Quasicritical behavior and first-order transition in the $d=3$
  random-field ising model.
\newblock {\em Phys. Rev. Lett.}, 54:2429--2432, 1985.

\end{thebibliography}
\end{document}